\documentclass[12pt]{amsart}
\usepackage[margin=1in]{geometry}

\usepackage{amsmath, amsthm, amssymb, hyperref, color}
\usepackage{MnSymbol} 
\usepackage{graphicx}
\usepackage{caption}
\usepackage[all]{xypic}
\usepackage{verbatim}
\usepackage{chemfig, chemnum}
\usepackage{tikz}
\usepackage[linesnumbered,lined,commentsnumbered]{algorithm2e}
\usepackage{caption}
\usepackage{subcaption}
\usepackage{float}
\usepackage{tikz-cd}
\usetikzlibrary{calc}
\usetikzlibrary{patterns}
\usepackage{extarrows}

\usepackage{makecell}
\usepackage{array}
\usepackage{fancyvrb}
\usepackage{enumitem}
\usepackage{standalone}
\usepackage{subcaption}

\usepackage{amsfonts}

\usepackage[T1]{fontenc}
\usepackage{imakeidx}
\makeindex

\makeatletter
\def\l@section{\@tocline{1}{0pt}{1pc}{}{}}
\def\l@subsection{\@tocline{2}{0pt}{3pc}{5pc}{}}
\def\l@subsubsection{\@tocline{3}{0pt}{5pc}{7pc}{}}
\makeatother

\usepackage[
    backend=biber,
    style=alphabetic,,
    sorting=nyt,
    maxnames=99,
    maxalphanames=4
]{biblatex}
\renewbibmacro{in:}{}
\AtEveryBibitem{%
    \clearfield{doi}%
    \clearfield{url}%
    \clearfield{urldate}%
    \clearfield{isbn}%
    \clearfield{issn}%
    \clearfield{language}%
    \clearfield{month}%
    \clearfield{day}%
    \clearfield{eprintclass}%
    \iffieldundef{journaltitle}
        {}
        {%
            \clearfield{eprint}%
            \clearfield{eprinttype}%
        }%
}

\DeclareFieldFormat{eprint:arxiv}{%
  \href{https://arxiv.org/abs/#1}{\texttt{arXiv #1}}%
}

\DeclareLabelalphaTemplate{
  \labelelement{
    \field[final]{shorthand}
    \field{label}
    \field[strwidth=1,strside=left]{labelname}
  }
}

\DeclareLabelalphaTemplate{
  \labelelement{
    \field[final]{shorthand}
    \field{label}
    \field[strwidth=2,strside=left,ifnames=1]{labelname}
    \field[strwidth=1,strside=left]{labelname}
  }
}

\hypersetup{
	colorlinks,
	citecolor=black,
	filecolor=black,
	linkcolor=black,
	urlcolor=blue
}
\usepackage{zref-clever}
\zcsetup{cap,nameinlink}

\newcommand{\bR}{\mathbb R}
\newcommand{\bM}{\mathbb M}
\newcommand{\bC}{\mathbb C}
\newcommand{\bZ}{\mathbb Z}

\newcommand{\bP}{\mathbb P}

\newcommand{\id}{\text{id}}

\renewcommand{\phi}{\varphi}

\newcommand{\im}{\text{im }}
\newcommand{\Hom}{\text{Hom}}

\newcommand{\Gr}{{\rm Gr}}

\newcommand{\GL}{{\rm GL}}

\newcommand{\onshellgraph}[5]{
    \begin{tikzpicture}[scale=#5]
        \foreach \pos\index in #1{
            \node(\index) at \pos {\index};
        }
    
        \foreach \pos\index in #2{
            \node(\index)[draw,circle] at \pos {};
        }

        \foreach \pos\index in #3{
            \node(\index)[draw,fill,circle] at \pos {};
        }

        \foreach \first/\second in #4{
            \draw[white,line width=2.5pt] (\first) -- (\second);    
            \draw[line width=1.5pt] (\first) -- (\second);
        }
    \end{tikzpicture}
}

\newcommand{\coloredonshellgraph}[5]{
    \begin{tikzpicture}[scale=#5]
        \foreach \pos\index in #1{
            \node(\index) at \pos {\index};
        }
    
        \foreach \pos\index in #2{
            \node(\index)[draw,circle] at \pos {};
        }

        \foreach \pos\index in #3{
            \node(\index)[draw,fill,circle] at \pos {};
        }

        \foreach \first/\second/\color in #4{
            \draw[color=white,line width=2.5pt] (\first) -- (\second);    
            \draw[color=\color,line width=1.5pt] (\first) -- (\second);
        }
    \end{tikzpicture}
}

\newcommand{\onshellgraphwithcode}[6]{
    \begin{tikzpicture}[scale=#5]
        \foreach \pos\index in #1{
            \node(\index) at \pos {\index};
        }
    
        \foreach \pos\index in #2{
            \node(\index)[draw,circle] at \pos {};
        }

        \foreach \pos\index in #3{
            \node(\index)[draw,fill,circle] at \pos {};
        }

        \foreach \first/\second/\color in #4{
            \draw[color=white,line width=2.5pt] (\first) -- (\second);    
            \draw[line width=1.5pt,color=\color] (\first) -- (\second);
        }

        #6
    \end{tikzpicture}
}

\newcommand{\ta}[0]{\text{a}}
\newcommand{\tb}[0]{\text{b}}
\newcommand{\ttc}[0]{\text{c}}
\newcommand{\td}[0]{\text{d}}
\newcommand{\te}[0]{\text{e}}
\newcommand{\tf}[0]{\text{f}}

\ExplSyntaxOn
\newcommand{\zcrefglobalstringname}[2]{
  \__zrefclever_opt_varname_lang_type:nnnn{\languagename}{#1}{#2}{tl}
}
\newcommand{\zcreflocalstringname}[2]{
  \__zrefclever_opt_varname_type:een{#1}{#2}{tl}
}
\newcommand{\zcrefgetstring}[2]{
  \__zrefclever_provide_langfile:e { \languagename }
  \__zrefclever_process_language_settings:
  \ifcsvoid{\zcreflocalstringname{#1}{#2}}
    {\csname \zcrefglobalstringname{#1}{#2} \endcsname}
    {\csname \zcreflocalstringname{#1}{#2} \endcsname}
}
\ExplSyntaxOff
\newcommand{\zcreftypename}[1]{\zcrefgetstring{#1}{Name-sg}}
\newtheorem{sharedtheoremcounter}{}[section]
\undef\sharedtheoremcounter 

\NewDocumentCommand{\newzctheorem}{mO{#1}O{\zcreftypename{#2}}}{
  \newtheorem{#1}[sharedtheoremcounter]{#3}
    \AddToHook{env/#1/begin}{%
      \zcsetup{countertype={sharedtheoremcounter=#2}}}
}

\theoremstyle{remark}

\theoremstyle{plain}
\newzctheorem{theorem}[theorem]
\newzctheorem{lemma}[lemma]
\newzctheorem{proposition}[proposition]
\newzctheorem{claim}[claim]
\newzctheorem{cor}[corollary]

\theoremstyle{definition}
\newzctheorem{definition}
\AtBeginEnvironment{definition}{%
  \pushQED{\qed}%
}
\AtEndEnvironment{definition}{\popQED}

\newzctheorem{notation}[notation][Notation]
\AtBeginEnvironment{notation}{%
  \pushQED{\qed}%
}
\AtEndEnvironment{notation}{\popQED}

\newzctheorem{eg}[example]
\AtBeginEnvironment{eg}{%
  \pushQED{\qed}%
}
\AtEndEnvironment{eg}{\popQED}

\newzctheorem{prob}[problem]
\newzctheorem{ex}[exercise]
\newzctheorem{conj}[conjecture]

\newzctheorem{rem}[remark]
\AtBeginEnvironment{rem}{%
  \pushQED{\qed}%
}
\AtEndEnvironment{rem}{\popQED}
\newzctheorem{question}[question]

\DeclareMathOperator{\cross}{cross}
\DeclareMathOperator{\rowsp}{rowsp}
\def\bdG{\partial G}
\def\bu{\mathbf{u}}
\def\bfR{\mathbf{R}}
\DeclareMathOperator{\tot}{tot}
\DeclareMathOperator{\match}{Match}
\DeclareMathOperator{\conv}{conv}

\title{A mathematical perspective on nonplanar on-shell forms}
\author{Artyom Lisitsyn, Elizabeth Pratt, Melissa Sherman-Bennett, Jaroslav Trnka}

\begin{document}
\begin{abstract}
    On-shell forms are differential forms on the Grassmannian which arise in particle physics. They are defined using bipartite graphs with $n$ distinguished ``boundary'' vertices. Mathematical investigation of on-shell forms has largely focused on the case where the graph is planar, in which case one can utilize combinatorial tools pioneered by Postnikov in the study of the totally nonnegative Grassmannian. In this article, we investigate on-shell forms for arbitrary graphs. We first discuss how to extend various tools from the planar case to arbitrary graphs. We then prove a determinantal formula for a class of on-shell forms which first appeared in physics literature \cite{Arkani-Hamed:2014bca}. We explain the relation between this class of forms and the hypertree divisors of $M_{0,n}$, introduced by Castravet--Tevelev \cite{CastravetTevelev}.
\end{abstract}

\maketitle

\tableofcontents

\section{Introduction}

On-shell forms are differential forms on the Grassmannian and its subvarieties which arise in physics, specifically in the study of particle scattering \cite{Arkani-Hamed:2012zlh}. An on-shell form $\omega_G$ is determined by a bipartite graph $G$ with distinguished boundary vertices $1, \dots, n$, which we call a \emph{graph with boundary}. For particular choices of graph, the on-shell form recovers the celebrated Parke-Taylor form, which is used to compute tree-level scattering amplitudes for gluons \cite{Parke:1986gb}.

Mathematical investigation of on-shell forms has thus far focused on the case where the graph $G$ is a \emph{plabic graph}, meaning $G$ is planar and embedded in a disk with $1, \dots, n$ in order on the boundary. Plabic graphs were introduced by Postnikov to parametrize \emph{positroid cells} in the totally nonnegative Grassmannian \cite{postnikov}, and have since been studied extensively, in works including \cite{psw2009,Lam-grassmann-polytope, Speyer16, MSTwist}. When $G$ is plabic, one can use the combinatorics and geometry of the totally nonnegative Grassmannian to aid the study of the form $\omega_G$. For example, in the plabic case $\omega_G$ can be rigorously defined via Postnikov's boundary measurement map for parametrizing the positroid cell $C_G$. One can use statistics on plabic graphs to determine when the on-shell forms for different graphs coincide. 

On-shell forms also connect to the burgeoning mathematical field of positive geometries, introduced by Arkani-Hamed--Bai--Lam \cite{ahbl2017}. A positive geometry is a semialgebraic subset $X^{\ge 0}$ of a normal projective variety $X$ which admits a unique \emph{canonical form} $\Omega$. The canonical form is a nonzero top degree rational differential form which reflects the boundary structure of $X^{\ge0}$ in a precise way; roughly speaking, $\Omega$ has logarithmic singularities exactly on the boundary of $X^{\ge 0}$, and its residues have the same property for ``boundaries of boundaries'' of $X^{\ge 0}$. Examples of positive geometries include polytopes in projective space and the nonnegative parts of toric varieties. Plabic on-shell forms $\omega_G$ are examples of canonical forms: positroid cells $C_G$ are positive geometries \cite{Lam-pos-geom}, with canonical form $\omega_G$.

This paper investigates on-shell forms for arbitrary graphs with boundary. Such forms are of substantial interest in scattering amplitudes, and have been considered in a number of physics papers \cite{Arkani-Hamed:2014bca,Franco:2015rma,Bourjaily:2016mnp}. There are also indications that these forms are very interesting from a mathematical perspective. For example, \cite[Eq. (2.15)]{Arkani-Hamed:2014bca} posits a beautiful determinantal formula for $k=2$ ``top-cell'' on-shell forms. The formula implies that, surprisingly, these forms vanish to order 2 on a \emph{hypertree divisor} of $\overline{M_{0,n}}$, defined by Castravet--Tevelev \cite{CastravetTevelev}. Here, $\overline{M_{0,n}}$ is the Deligne--Mumford compactification of the moduli space $M_{0,n}$ of genus 0 curves with $n$ marked points. Hypertree divisors are non-boundary divisors which give extremal rays of the effective cone of $\overline{M_{0,n}}$. One motivation for this paper was to provide a mathematical proof for the rational formula \cite[Eq. (2.15)]{Arkani-Hamed:2014bca} for $k=2$ forms and explain the relationship between these forms and hypertree divisors, which we do in \zcref{subsec:det formula,subsec:hypertree}, respectively.

We begin the paper by extending some well-known concepts from the setting of plabic graphs to arbitrary bipartite graphs with boundary (see \zcref{sec:boundary-measurement}); we need these in order to define the on-shell form. These concepts include gauge transformations on the space of edge weights $(\bC^*)^E$, the boundary measurement map $\bM^G:(\bC^*)^E/\sim_\text{gauge} \to \Gr_{k,n}$ via perfect matchings, moves on graphs, and the notion of a reduced graph. These definitions are essentially already present in the literature; in particular, the well-definedness of our boundary measurement map is work of Speyer \cite{Speyer16}. Some subtleties arise in comparison to the plabic case. For example, the boundary measurement map $\bM^G$ involves signs which depend on a choice of drawing of $G$ in a disk. For reduced non-plabic $G$, $\bM^G$ is in general a rational map and may be many-to-one. The notion of reducedness is also more complicated in the non-plabic case, as combinatorial tools for plabic graphs such as trips are not available.

Extending work of \cite{psw2009} in the plabic case, we define a toric variety $X_G$ for each graph with boundary (see \zcref{sec:toric}). The moment polytope of $X_G$ has vertices in bijection with the perfect matchings of $G$. We show that $X_G$ is also a GIT quotient of $\bC^E$. The nonnegative part $X_G^{\ge 0}$ of $X_G$ is a positive geometry, and we give a formula for its canonical form $\eta_G$ involving spanning trees of $G$. Essentially by definition, the boundary measurement map factors through $X_G$, and we use this to define on-shell forms.

In \zcref{sec:form}, we define the on-shell form $\omega_G$ as the push-forward of the canonical form $\eta_G$ of $X_G$ along the boundary measurement map:
\[\omega_G:=\bM^G_*~\eta_G.\]
This definition is consistent with the on-shell form for plabic graphs in \cite[Section 14]{Lam-grassmann-polytope} and is a reasonable formalization of the on-shell form for nonplanar graphs in the physics literature \cite{Arkani-Hamed:2014bca,Bourjaily:2016mnp,Franco:2015rma}. The heuristic for positive geometries \cite[Heuristic 4.1]{ahbl2017} predicts that $\omega_G$ is the canonical form of $\overline{\bM^G(X_G^{>0})}$, provided that the latter is a positive geometry. This is true for plabic $G$, where $\overline{\bM^G(X_G^{>0})}$ is simply the positroid cell.

With the mathematical definition of $\omega_G$ in hand, in \zcref{sec:k=2} we turn to the $k=2$ ``top cell'' on-shell forms studied in \cite{Arkani-Hamed:2014bca}. These forms arise from graphs $G$ where the boundary measurement map $\bM^G: (\bC^E)^* \dashrightarrow \Gr_{2,n}$ is dominant. We characterize $k=2$ top-cell graphs, and prove the rational formula for $\omega_G$ conjectured in \cite[Eq. (2.15)]{Arkani-Hamed:2014bca}. In order to do so, we introduce a decomposition of graphs with boundary into ``lobes'', and analyze the interaction of this decomposition with the boundary measurement map $\bM^G$ and the on-shell form $\omega_G$. We then describe the precise relationship between $k=2$ top-cell graphs $G$ and the hypertrees $\Gamma$ of Castravet--Tevelev \cite{CastravetTevelev}; between the numerator of $\omega_G$ and the equation cutting out the hypertree divisor $D_{\Gamma}$; and between the hypertree divisor $D_\Gamma$ and the boundary measurement map $\bM^G$.

We conclude the paper with a number of questions and further directions.

\vspace{12pt}

\noindent\textbf{Relation to prior work.} Above we mention prior mathematical work on plabic graphs and their on-shell forms; we also mention several physics references for non-plabic on-shell forms. Non-plabic graphs with boundary have also appeared in the mathematics literature. In \cite{TW}, the authors use a certain class of such graphs to parametrize Deodhar components in the Grassmannian. Their parametrization is somewhat different from the boundary measurement map introduced here. In \cite{GLSBS}, the authors use a different class of non-plabic graphs with boundary to generalize the cluster structure on positroid varieties to Richardson varieties. We do not know if the graphs in this present paper can be used to obtain a cluster structure on the image of the boundary measurement map. Finally, the ideas behind the boundary measurement map for nonplanar graphs (formulated using paths rather than matchings) go back to at least to Mason in the 1970s \cite{Mason}, and is closely related to the class of matroids known as gammoids.

\noindent \textbf{Summary for physics audience:}
 In this paper, we study non-planar on-shell diagrams and the connection to the Grassmannian $\Gr_{k,n}$. We find new ways of constructing the boundary measurement matrix using \emph{perfect matchings} and \emph{vector-relation configurations}. The main result is a proof of the determinantal formula for the $k=2$ on-shell form of \cite{Arkani-Hamed:2014bca}, which we rigorously construct without any reference to momentum delta functions used in the original construction. Our method opens the door for possible higher-$k$ extensions.

\noindent \textbf{LLM Use.} ChatGPT 5.6 Sol was used in the proof of \zcref{lem:jacobian}. Claude Fable 5 was used in the proof of \zcref{prop:bdry-meas-torus-vs-dense-torus} and to generated a preliminary referee report, which caught minor errors. Locally-hosted models (Gemma-4-31B, Qwen3.8-27B, Qwen3.8-Flash-Next) aided with writing code verifying \zcref{eg:threesix} and editing graphs in \zcref{fig:twosix} and \zcref{fig:threesix}. All text was written solely by the authors. 

\noindent \textbf{Acknowledgements.} MSB is partially supported by the National Science Foundation under Award No.~DMS-2444020. JT and AL are supported by the US Department of Energy under grant No. SC0009999, and by the funds of the University of California. EP was supported by NSF GRFP no. 2023358166, a UC Davis Chancellor’s Fellowship, and a Perimeter Institute Postdoctoral Fellowship. Part of this work was conducted while EP and AL were at the Max Planck Institute for Mathematics in the Natural Sciences in Leipzig, Germany.
  Any opinions, findings, and conclusions or recommendations expressed in this material are
those of the author(s) and do not necessarily reflect the views of the National Science
Foundation.

\section{Graphs with boundary and the boundary measurement map}\label{sec:boundary-measurement}

In this section, we take many concepts that are standard for \emph{plabic graphs}, that is, planar bipartite graphs embedded in a disk with some vertices on the boundary, and extend them to nonplanar graphs. Plabic graphs were introduced by Postnikov \cite{postnikov}, who used them to index and parametrize cells of the totally nonnegative Grassmannian. Their combinatorics has been studied extensively. Many of the definitions for plabic graphs extend naturally to nonplanar graphs; some definitions, such as the boundary measurement map, become slightly more complicated; and other concepts, such as reducedness, become substantially more complicated.

\subsection{Basic notions on graphs}\label{subsec:graph basics}

A \emph{graph with boundary} is a graph $G$ with a distinguished subset of \emph{boundary vertices} $\bdG$, labeled $1, \dots, n$. All other vertices of $G$ are \emph{internal}. Edges incident to boundary vertices are \emph{boundary edges}.

Throughout this section, we assume that $G$ is a bipartite graph with boundary, and all boundary vertices are white. We use $W, B, E$ for the sets of internal white vertices, black vertices, and edges of $G$, respectively, and $V=W \cup B$ for the set of internal vertices. We set $k_G:= |\bdG| + |W|-|B|$. A \emph{type $(k,n)$ graph} $G$ is a graph satisfying the assumptions of this paragraph, with $k_G = k$.

A \emph{perfect matching} of a type $(k,n)$ graph $G$ is a subset $M \subset E$ of the edges such that every internal vertex is contained in exactly one edge in $M$ and every boundary vertex is contained in at most one edge of $M$. The \emph{boundary} $\partial M$ of the matching $M$ is the $(n-k)$-element set
    \[\{i \in \bdG: i \text{ is in some }e \in M \}.\]
See \zcref{fig:twosix-labeling} for an example.

\begin{figure}
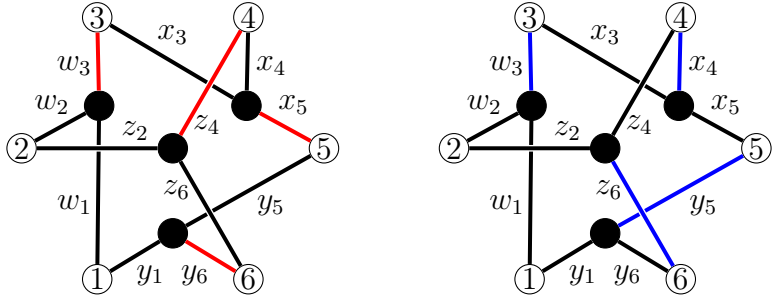

    \[
        \onshellgraphwithcode{{(-120:2)/1,(-180:2)/2,(120:2)/3,(60:2)/4,(0:2)/5,(-60:2)/6}}{{(-120:2)/W1,(-180:2)/W2,(120:2)/W3,(60:2)/W4,(0:2)/W5,(-60:2)/W6}}{{(150:1.125)/B1,(30:1.125)/B2,(-90:1.125)/B3,(0,0)/B4}}{{W1/B1/black, W2/B1/black, W3/B1/red, W3/B2/black, W4/B2/black, W5/B2/red, W1/B3/black, W5/B3/black, W6/B3/red, B4/W2/black, B4/W4/red, B4/W6/black}}{1}{
        \node() at (210:1.5) {$w_1$};
        \node() at (160:1.7) {$w_2$};
        \node() at (140:1.7) {$w_3$};
        \node() at (90:1.5) {$x_3$};
        \node() at (40:1.7) {$x_4$};
        \node() at (20:1.7) {$x_5$};
        \node() at (330:1.5) {$y_5$};
        \node() at (280:1.7) {$y_6$};
        \node() at (260:1.7) {$y_1$};
        \node() at (155:0.55) {$z_2$};
        \node() at (35:0.55) {$z_4$};
        \node() at (275:0.55) {$z_6$};
        }
        \hspace{+3em}
        \onshellgraphwithcode{{(-120:2)/1,(-180:2)/2,(120:2)/3,(60:2)/4,(0:2)/5,(-60:2)/6}}{{(-120:2)/W1,(-180:2)/W2,(120:2)/W3,(60:2)/W4,(0:2)/W5,(-60:2)/W6}}{{(150:1.125)/B1,(30:1.125)/B2,(-90:1.125)/B3,(0,0)/B4}}{{W1/B1/black, W2/B1/black, W3/B1/blue, W3/B2/black, W4/B2/blue, W5/B2/black, W1/B3/black, W5/B3/blue, W6/B3/black, B4/W2/black, B4/W4/black, B4/W6/blue}}{1}{
        \node() at (210:1.5) {$w_1$};
        \node() at (160:1.7) {$w_2$};
        \node() at (140:1.7) {$w_3$};
        \node() at (90:1.5) {$x_3$};
        \node() at (40:1.7) {$x_4$};
        \node() at (20:1.7) {$x_5$};
        \node() at (330:1.5) {$y_5$};
        \node() at (280:1.7) {$y_6$};
        \node() at (260:1.7) {$y_1$};
        \node() at (155:0.55) {$z_2$};
        \node() at (35:0.55) {$z_4$};
        \node() at (275:0.55) {$z_6$};
        }
    \]
 
    \caption{A type (2,6) graph with edge weights and two perfect matchings with boundary $\partial M = \{3,4,5,6\}$. The edges connected to vertices $\{1,2\}$ are excluded from both perfect matchings. 
    }\label{fig:twosix-labeling}
\end{figure}
    
To avoid trivialities, we will always assume that $G$ has at least one perfect matching. This implies that $0 \le k_G \le n$.

We will often think of $G$ with edges weighted by nonzero weights $x_e \in \bC^*$. The \emph{weighted adjacency matrix} of $G$ is $K=(K_{b,w})$, which has rows indexed by $B$ (for some fixed total order), columns indexed by $W \cup \bdG$ (for some fixed total order which ends in $1, \dots, n$), and has entries

   \[K_{b,w} = \begin{cases}
        x_e & \text{ if }e=\{b,w\} \in E\\
        0 & \text{ else}.
\end{cases}\] 

In the future, we will sometimes need to change the signs of the entries of $K$; for $\sigma =(\sigma_e)_{e \in E} \in \{\pm 1\}^E$, we will denote by $\sigma K$ the matrix with nonzero entries $\sigma_e x_e$.

\subsection{Gauge transformation and the torus of edge weights modulo gauge}

For each internal vertex $v \in V$, we define the \emph{gauge transformation} at $v$ to be the action of $\bC^*$ on $(\bC^*)^E$ given by
\[t \cdot x_e = \begin{cases}
    t x_e & \text{ if } v \in e\\
x_e & \text{ else.}
\end{cases}\]
In this manner each vertex determines a co-character $\lambda_v: \bC^* \hookrightarrow (\bC^*)^E.$ The collection of vertices determines a map from $ (\bC^*)^V$ to $(\bC^*)^E,$ given concretely by the $|E| \times |V|$ incidence matrix $Q$ of the internal vertices. We define the torus $(\bC^*)^{E-V}$ as the cokernel 
    \begin{equation}\label{eq:cokernel}
        (\bC^*)^V \xrightarrow{Q} (\bC^*)^E \to (\bC^*)^{E - V} \to 0.
    \end{equation}
Concretely, we think of $(\bC^*)^{E - V}$ as the space of nonzero edge weights modulo gauge transformations at internal vertices.

We will later need some coordinates on $(\bC^*)^{E-V}$, i.e. a basis of characters. One way to obtain such coordinates is to choose a specific representative in $(\bC^*)^E$ for each gauge-equivalence class in $(\bC^*)^{E-V}$, for example by specifying that certain edge weights are 1 to ``gauge fix''. 
Another way to obtain coordinates is to find ``enough'' gauge-invariant functions on $(\bC^*)^E$. To formalize this, we use the homology group\footnote{We will always use integral coefficients. For $e=\{b,w\}$, we define the boundary map by $e \mapsto b-w$.} $H_1(G, \bdG)$. 

Recall that 
\[H_1(G, \bdG) = \{p \in \bZ^E: \text{ for all }v \in V, ~\sum_{v \ni e} p_e =0\}.\]
It is spanned by the ``alternating'' indicator vectors of cycles in $G$ and paths from $\bdG$ to $\bdG$, obtained from the indicator vectors by changing the sign of every other entry.

\begin{lemma}
   For $p \in \bZ^E$, the character $(x_e) \mapsto \prod x_e^{p_e}$ is gauge-invariant if and only if $p \in H_1(G, \bdG)$. 
\end{lemma}
\begin{proof}
    Gauge transformation by $t$ at $v$ changes $p(x_e)$ by $\prod_{v \ni e}t^{p_e}= t^{\sum_{v \ni e}p_e}$. 
\end{proof}

Let $G/\partial G$ be the graph obtained from $G$ by identifying all boundary vertices; note the edge set remains the same.

\begin{cor}\label{cor:basis-for-torus-mod-gauge}
    Any basis of $H_1(G, \bdG)$ gives a basis of characters of $(\bC^*)^{E-V}$. Thus, if $c$ is the number of connected components of $G/\bdG$, then $\dim(\bC^*)^{E-V} = |E|-|V| + (c-1)$.
\end{cor}
\begin{proof} 
    Dualizing the cokernel sequence \eqref{eq:cokernel} yields 
    \[0 
    \to \Hom((\bC^*)^{E - V}, \bC^*) \to \Hom((\bC^*)^E, \bC^*) \xrightarrow{{Q}^T}  \Hom((\bC^*)^V, \bC^*). \]
    The boundary map is given by $Q^T$ with every row corresponding to a white vertex multiplied by $-1.$ This operation preserves the kernel, and identifies $\Hom((\bC^*)^{E - V}, \bC^*)$ with $H_1(G, \partial G).$ The statement about dimension is because the dimension of $H_1(G, \bdG)$ is the number of edges in the complement of a spanning forest of $G/\bdG$. A spanning forest of $G/\bdG$ contains $|V| + 1$ vertices and thus $|V| + 1 - c$ edges. The complement contains $|E|-|V| + (c-1)$ edges.
\end{proof}

\begin{eg}
    Consider the graph $G$ in \zcref{fig:basic example} with $|W| = |B| = 1,$ and its contraction $G / \partial G.$
    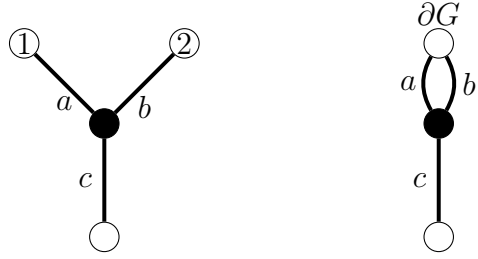
\begin{figure}[!h]
        \[
            \begin{tikzpicture}
                \node[draw,circle](W1) at (-1.06,1.06) {};
                \node at (-1.06,1.06) {1};
                \node[draw,circle](W2) at (1.06,1.06) {};
                \node at (1.06,1.06) {2};
                \node[draw,circle](Wi) at (-90:1.5) {};
                \node[draw,fill,circle](B) at (0,0) {};

                \draw[line width=1.5pt] (W1) -- (B) node[midway,below] {$a$};
                \draw[line width=1.5pt] (W2) -- (B) node[midway,below] {$b$};
                \draw[line width=1.5pt] (Wi) -- (B) node[midway,left] {$c$};
            \end{tikzpicture}
            \hspace{6em}
            \begin{tikzpicture}
                \node[draw,circle](W1) at (0,1.06) {};
                \node at (0,1.45) {$\partial G$};
                \node[draw,circle](Wi) at (-90:1.5) {};
                \node[draw,fill,circle](B) at (0,0) {};

                \draw[line width=1.5pt] (W1) to[out=-120,in=120] (B);
                \node at (-0.4,0.53) {$a$};
                \draw[line width=1.5pt] (W1) to[out=-60,in=60] (B);
                \node at (0.4,0.53) {$b$};
                \draw[line width=1.5pt] (Wi) -- (B) node[midway,left] {$c$};
            \end{tikzpicture}
        \]
        \caption{A graph $G$ and its contraction $G / \partial G$}\label{fig:basic example}
    \end{figure}
    Define the matrices 
    \[\tilde{Q} = \begin{bmatrix}
        1 & 1 & 1 \\
        0 & 0 & -1
    \end{bmatrix}, \quad P = \begin{bmatrix}
        -1 & -1 & 0 \\
        1 & 1 & 1 \\
        0 & 0 & -1
    \end{bmatrix}.\]
    The boundary map is the map from $\bZ^E$ to $\bZ^V$ given concretely by the matrix $\tilde{Q},$ whereas the matrix $P$ is the oriented incidence matrix of $G / \partial G,$ with columns corresponding in order to $a, b, c.$ The kernels of $\tilde{Q}$ and $P$ both equal $H_1(G, \partial G)$, and consist of the span of $[1:-1:0]^T,$ which is the alternating incidence vector of the unique cycle in $G / \partial G.$
\end{eg}

\begin{rem}\label{rem:torus-coordinates}
  Here is a concrete way to obtain coordinates for $(\bC^*)^{E-V}$. Choose a spanning forest $F$ of $G / \partial G$. Now, consider an edge $e \in G / \partial G$ which is not in $F$. The graph $F \cup e$ has a single cycle (necessarily including the edge $e$), with edges $e=e_1, e_2, \dots, e_{2m}$. We define
    \[y_e:=x_{e_1} x_{e_2}^{-1} x_{e_3}\dots x_{e_{2m}}^{-1}.\]
    Then $\{y_e: e \in E \setminus  F\}$ are a basis of characters for $(\bC^*)^{E-V}$. If one would prefer to obtain coordinates by gauge-fixing instead, one can set $x_{e'} =1$ for all $e' \in F$, in which case the coordinates of $(\bC^*)^{E-V}$ are $\{x_e : e \notin F\}$.
    
   When $G$ is planar, sometimes $F$ is chosen so that the cycle in $F \cup e$ is a face of $G$, and $y_e$ becomes an alternating product of edge weights around faces. 
\end{rem}

\subsection{The boundary measurement map}

We will use the edge-weighted graph $G$ to define the boundary measurement map, a rational map from $(\bC^*)^{E-V}$ to $\Gr_{k,n}$. Our definition is essentially present in work of Speyer \cite{Speyer16}, and will require choosing a specific drawing of $G$ in the disk $D$. Our allowed drawings follow the conventions of \cite{Speyer16}. A \emph{disk immersion} of $G$ is a drawing of $G$ in $D$ such that each edge is a line segment\footnote{One could relax this to piecewise linear at the expense of additional bookkeeping.} of nonzero length, no vertex is in the interior of any edge, all intersections of edges are transverse, and the boundary vertices $\bdG$ are on $\partial D$ in the correct order. If $G$ has more than one edge between a pair of vertices, we slightly deform these edges to be piecewise linear with exactly 2 linear pieces. From now on, we will fix both $G$ and a disk immersion, in which case we call $G$ an \emph{immersed} graph.

\begin{notation}
     We use $[n]:=\{1, \dots, n\}$ and $\binom{[n]}{k}:=\{I \subset [n]: |I|=k\}$. For $I \subset [n]$, when $n$ is clear from context, we use notation $I^c:= [n]\setminus I$. Set $\tot(I):= \sum_{i \in I} i$. 
\end{notation}

For a matching $M$ of an immersed graph $G$, we define the \emph{crossing number} $\cross(M)$ of $M$ as the number of intersection points of edges in $M$ in the immersion.

\begin{definition}\label{def:bdry-meas}
    Let $G$ be an immersed type $(k,n)$ graph, with edges weighted by formal parameters $x_e$. For $I \in \binom{[n]}{k}$, we define
    \[\bM^G_I(x_e):= (-1)^{\tot(I)} \sum_{M: \partial M=I^c} (-1)^{\cross(M)}\prod_{e \in M} x_e.\]
    The \emph{boundary measurement map} is
    \begin{align*}
        \bM^G:(\bC^*)^E &\dashrightarrow \Gr_{k,n}\\
        (x_e) & \mapsto (\bM_I^G)_{I \in \binom{[n]}{k}}
    \end{align*} and descends to a rational map $\bM^G:(\bC^*)^{E-V} \dashrightarrow \Gr_{k,n}$.
\end{definition}
Note the locus of indeterminacy of $\bM^G$ is the subset of $(\bC^*)^{E}$ or $(\bC^*)^{E-V}$ where all $\bM_I^G(x_e)$ vanish. 

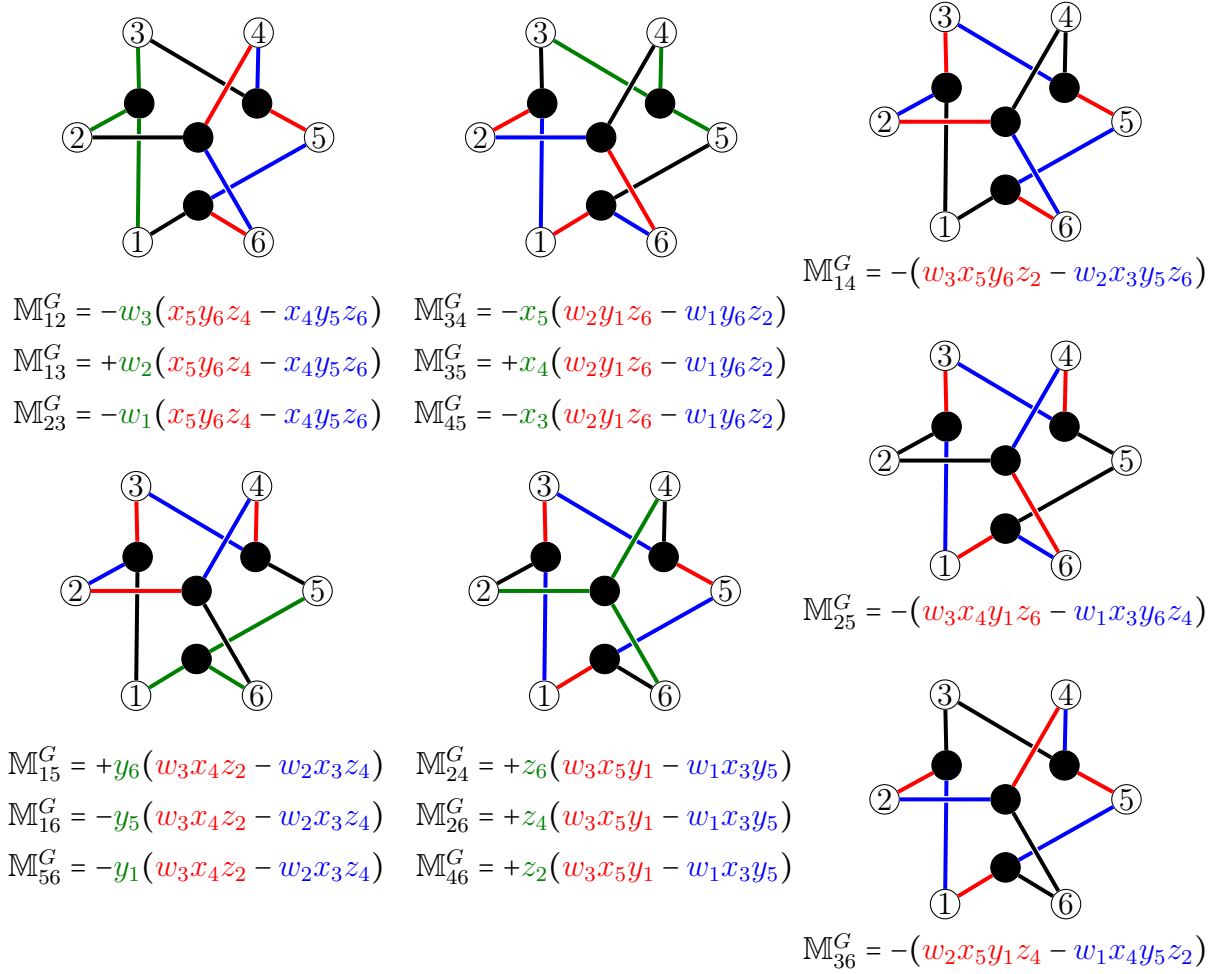
\begin{figure}[h]
    \begin{center}
        \begin{tikzpicture}
            \node() at (0,0) {
                \begin{tikzpicture}
                    \node at (0,0) {\coloredonshellgraph{{(-120:2)/1,(-180:2)/2,(120:2)/3,(60:2)/4,(0:2)/5,(-60:2)/6}}{{(-120:2)/W1,(-180:2)/W2,(120:2)/W3,(60:2)/W4,(0:2)/W5,(-60:2)/W6}}{{(150:1.125)/B1,(30:1.125)/B2,(-90:1.125)/B3,(0,0)/B4}}{{W1/B1/green!50!black, W2/B1/green!50!black, W3/B1/green!50!black, W3/B2/black, W4/B2/blue, W5/B2/red, W1/B3/black, W5/B3/blue, W6/B3/red, B4/W2/black, B4/W4/red, B4/W6/blue}}{0.8}};
            
                    \node at (0,-3) {$\begin{aligned}
                            \mathbb M^G_{12} &= -{\color{green!50!black}w_3} ({\color{red}x_5 y_6 z_4} - {\color{blue}x_4 y_5 z_6}) \\
                            \mathbb M^G_{13} &= +{\color{green!50!black}w_2} ({\color{red}x_5 y_6 z_4} - {\color{blue}x_4 y_5 z_6}) \\
                            \mathbb M^G_{23} &= -{\color{green!50!black}w_1} ({\color{red}x_5 y_6 z_4} - {\color{blue}x_4 y_5 z_6})
                        \end{aligned}$};
                \end{tikzpicture}
                \begin{tikzpicture}
                    \node at (0,0) {\coloredonshellgraph{{(-120:2)/1,(-180:2)/2,(120:2)/3,(60:2)/4,(0:2)/5,(-60:2)/6}}{{(-120:2)/W1,(-180:2)/W2,(120:2)/W3,(60:2)/W4,(0:2)/W5,(-60:2)/W6}}{{(150:1.125)/B1,(30:1.125)/B2,(-90:1.125)/B3,(0,0)/B4}}{{W1/B1/blue, W2/B1/red, W3/B1/black, W3/B2/green!50!black, W4/B2/green!50!black, W5/B2/green!50!black, W1/B3/red, W5/B3/black, W6/B3/blue, W2/B4/blue, W4/B4/black, W6/B4/red}}{0.8}};
                
                    \node at (0,-3) {$\begin{aligned}
                            \mathbb M^G_{34} &= -{\color{green!50!black}x_5} ({\color{red}w_2 y_1 z_6} - {\color{blue}w_1 y_6 z_2}) \\
                            \mathbb M^G_{35} &= +{\color{green!50!black}x_4} ({\color{red}w_2 y_1 z_6} - {\color{blue}w_1 y_6 z_2}) \\
                            \mathbb M^G_{45} &= -{\color{green!50!black}x_3} ({\color{red}w_2 y_1 z_6} - {\color{blue}w_1 y_6 z_2})
                        \end{aligned}$};
                \end{tikzpicture}
            };
    
            \node() at (0,-6) {
                \begin{tikzpicture}
                    \node at (0,0) {\coloredonshellgraph{{(-120:2)/1,(-180:2)/2,(120:2)/3,(60:2)/4,(0:2)/5,(-60:2)/6}}{{(-120:2)/W1,(-180:2)/W2,(120:2)/W3,(60:2)/W4,(0:2)/W5,(-60:2)/W6}}{{(150:1.125)/B1,(30:1.125)/B2,(-90:1.125)/B3,(0,0)/B4}}{{W1/B1/black, W2/B1/blue, W3/B1/red, W3/B2/blue, W4/B2/red, W5/B2/black, W1/B3/green!50!black, W5/B3/green!50!black, W6/B3/green!50!black, W2/B4/red, W4/B4/blue, W6/B4/black}}{0.8}};
                
                    \node at (0,-3) {$\begin{aligned}
                            \mathbb M^G_{15} &= +{\color{green!50!black}y_6} ({\color{red}w_3 x_4 z_2} - {\color{blue}w_2 x_3 z_4}) \\
                            \mathbb M^G_{16} &= -{\color{green!50!black}y_5} ({\color{red}w_3 x_4 z_2} - {\color{blue}w_2 x_3 z_4}) \\
                            \mathbb M^G_{56} &= -{\color{green!50!black}y_1} ({\color{red}w_3 x_4 z_2} - {\color{blue}w_2 x_3 z_4})
                        \end{aligned}$};
                \end{tikzpicture}
                \begin{tikzpicture}
                    \node at (0,0) {\coloredonshellgraph{{(-120:2)/1,(-180:2)/2,(120:2)/3,(60:2)/4,(0:2)/5,(-60:2)/6}}{{(-120:2)/W1,(-180:2)/W2,(120:2)/W3,(60:2)/W4,(0:2)/W5,(-60:2)/W6}}{{(150:1.125)/B1,(30:1.125)/B2,(-90:1.125)/B3,(0,0)/B4}}{{W1/B1/blue, W2/B1/black, W3/B1/red, W3/B2/blue, W4/B2/black, W5/B2/red, W1/B3/red, W5/B3/blue, W6/B3/black, W2/B4/green!50!black, W4/B4/green!50!black, W6/B4/green!50!black}}{0.8}};
                
                    \node at (0,-3) {$\begin{aligned}
                            \mathbb M^G_{24} &= +{\color{green!50!black}z_6} ({\color{red}w_3 x_5 y_1} - {\color{blue}w_1 x_3 y_5}) \\
                            \mathbb M^G_{26} &= +{\color{green!50!black}z_4} ({\color{red}w_3 x_5 y_1} - {\color{blue}w_1 x_3 y_5}) \\
                            \mathbb M^G_{46} &= +{\color{green!50!black}z_2} ({\color{red}w_3 x_5 y_1} - {\color{blue}w_1 x_3 y_5})
                        \end{aligned}$};
                \end{tikzpicture}
            };
    
            \node() at (8,-3.5) {
                $
                \begin{gathered}
                \begin{tikzpicture}
                    \node at (0,0) {\coloredonshellgraph{{(-120:2)/1,(-180:2)/2,(120:2)/3,(60:2)/4,(0:2)/5,(-60:2)/6}}{{(-120:2)/W1,(-180:2)/W2,(120:2)/W3,(60:2)/W4,(0:2)/W5,(-60:2)/W6}}{{(150:1.125)/B1,(30:1.125)/B2,(-90:1.125)/B3,(0,0)/B4}}{{W1/B1/black, W2/B1/blue, W3/B1/red, W3/B2/blue, W4/B2/black, W5/B2/red, W1/B3/black, W5/B3/blue, W6/B3/red, W2/B4/red, W4/B4/black, W6/B4/blue}}{0.8}};
                
                    \node at (0,-2) {$\mathbb M^G_{14} = -{(\color{red}w_3 x_5 y_6 z_2} - {\color{blue}w_2 x_3 y_5 z_6})$};
                \end{tikzpicture}\\
                \begin{tikzpicture}
                    \node at (0,0) {\coloredonshellgraph{{(-120:2)/1,(-180:2)/2,(120:2)/3,(60:2)/4,(0:2)/5,(-60:2)/6}}{{(-120:2)/W1,(-180:2)/W2,(120:2)/W3,(60:2)/W4,(0:2)/W5,(-60:2)/W6}}{{(150:1.125)/B1,(30:1.125)/B2,(-90:1.125)/B3,(0,0)/B4}}{{W1/B1/blue, W2/B1/black, W3/B1/red, W3/B2/blue, W4/B2/red, W5/B2/black, W1/B3/red, W5/B3/black, W6/B3/blue, W2/B4/black, W4/B4/blue, W6/B4/red}}{0.8}};
                
                    \node at (0,-2) {$\mathbb M^G_{25} = -({\color{red}w_3 x_4 y_1 z_6} - {\color{blue}w_1 x_3 y_6 z_4})$};
                \end{tikzpicture}\\
                \begin{tikzpicture}
                    \node at (0,0) {\coloredonshellgraph{{(-120:2)/1,(-180:2)/2,(120:2)/3,(60:2)/4,(0:2)/5,(-60:2)/6}}{{(-120:2)/W1,(-180:2)/W2,(120:2)/W3,(60:2)/W4,(0:2)/W5,(-60:2)/W6}}{{(150:1.125)/B1,(30:1.125)/B2,(-90:1.125)/B3,(0,0)/B4}}{{W1/B1/blue, W2/B1/red, W3/B1/black, W3/B2/black, W4/B2/blue, W5/B2/red, W1/B3/red, W5/B3/blue, W6/B3/black, W2/B4/blue, W4/B4/red, W6/B4/black}}{0.8}};
                
                    \node at (0,-2) {$\mathbb M^G_{36} = -({\color{red}w_2 x_5 y_1 z_4} - {\color{blue}w_1 x_4 y_5 z_2})$};
                \end{tikzpicture}
                \end{gathered}
                $
            };
        \end{tikzpicture}
    \end{center}
    \caption{Boundary measurements $\bM_{ij}^G$ for $G$ from \zcref{fig:twosix-labeling}. When $i,j$ are adjacent to the same black vertex $b$ (left four images), the two matchings contributing to $\bM_{ij}^G$ use a fixed edge incident to $b$, which is one of the {\color{green!50!black}green} edges. The remaining edges of the matchings are in 
     {\color{red}red} (with zero crossings), and in {\color{blue}blue} (with one crossing). The overall sign is $(-1)^{i+j}$. Variables are colored in the same way as the edges they label.
     }
     \label{fig:twosix}
\end{figure}

\begin{eg}\label{eg:twosix} 
An immersed type $(2,6)$ graph is pictured in \zcref{fig:twosix-labeling}, as well as the two perfect matchings used to compute $\bM^G_{12}$. The remaining $\bM^G_{ij}$ are computed in \zcref{fig:twosix}. 
Note that all $\bM^G_{ij}$ vanish simultaneously for many choices of weights, such as when all edge weights are equal, and thus $\bM^G$ is \emph{not} defined on all of $(\bC^*)^{E-V}$.
\end{eg}

The sign $(-1)^{\tot(I)}$  in the boundary measurement map may be removed if one passes to the dual Grassmannian $\Gr_{n-k, n}.$ The following lemma is sometimes called \emph{Grassmann duality}.

\begin{lemma}\label{lem:duality-plucker-coordinates}
    If $V \in \Gr_{k,n}$ has Pl\"ucker coordinates $p_I(V)$, then $V^{\perp} \in \Gr_{n-k, n}$ has Pl\"ucker coordinates
    \[p_{I^c}(V^\perp) = (-1)^{\tot(I)} p_I(V).\]
\end{lemma}

To show that the boundary measurement map is well-defined, we need a lemma on the adjacency matrix $K$.

\begin{lemma}[{\cite[Proposition 1.1, Remark 1.2]{Speyer16}}]\label{lem:signs-for-K}
    Fix an immersed type $(k,n)$ graph $G$. There is a choice of signs $\sigma \in \{\pm 1\}^E$ such that for all $I \in \binom{[n]}{k}$, the maximal minor of $\sigma K$ using columns labeled by $W \cup I^c$ is equal to $(-1)^{\tot(I)}\bM_I^G(x_e)$. 
\end{lemma}

We will call any choice of signs $\sigma$ for which \zcref{lem:signs-for-K} holds \emph{immersion signs}.

\begin{rem}
    Suppose that the black vertices lie in order on a straight line in the disk, and the white vertices are on one side of the line, in order from left to right. Then by the proof of \cite[Proposition 1.1]{Speyer16}, the immersion signs may be taken to all be one.
\end{rem}

\begin{lemma}\label{lem:bdry-meas-well-defined}
    For an immersed type $(k,n)$ graph $G$, the boundary measurement map $\bM^G$ is well-defined.
\end{lemma}
\begin{proof} For ease of notation, we write $\bM_I(x_e)$ for $\bM_I^G(x_e)$.

Suppose $(x_e) \in (\bC^*)^{E}$ is in the dense open set where at least one $\bM_I(x_e)$ is nonzero. By \zcref{lem:signs-for-K}, for any choice of immersion signs $\sigma$, the maximal minor $\Delta_{W \cup I^c}$ of $\sigma K$ is equal to $(-1)^{\tot(I)}\bM_I(x_e)$. 
Since at least one of these minors is nonzero, we may perform invertible row operations to put $\sigma K$ in the form
    \[\begin{bmatrix}
        \mathrm{Id} & X\\
        0 & A
    \end{bmatrix},\]
    where $A$ is a $ (|\partial G| - k) \times |\partial G| = (n-k) \times n$ matrix and the identity is a $|W| \times |W|$ matrix. 
    Note $A$ is full rank, so $V= \rowsp A$ is in $\Gr_{n-k, n}$, and we have $p_{I^c}(V) = (-1)^{\tot(I)}\bM_I(x_e)$. Applying \zcref{lem:duality-plucker-coordinates}, we have that $V^\perp \in \Gr_{k,n}$ and has Pl\"ucker coordinates $p_I(V) = \bM_I(x_e)$. So $\bM^G$ is well-defined.

    To see that $\bM^G$ descends to $(\bC^*)^{E-V}$, note that gauge transformation at $v$ multiplies either a row or one of the first $|W|$ columns of $\sigma K$ by the same nonzero scalar, and so multiplies all $\bM_I(x_e)$ by the same nonzero scalar.
\end{proof}

\begin{rem}\label{rem:bdry-meas-vs-Postnikov}
    The boundary measurement map was first introduced by Postnikov \cite{postnikov} in the context of \emph{plabic} graphs, which are immersed graphs that are planar, i.e. no edges cross. Postnikov originally used perfect orientations rather than perfect matchings, but these are equivalent; see \cite[Remark 3.4]{MSTwist} for additional history. When $G$ is plabic, the signs $\cross(M)$ disappear from $\bM_I^G$. The signs $(-1)^{\tot(I)}$ are a result of applying Grassmannian duality (cf. \zcref{lem:duality-plucker-coordinates}). So when $G$ is plabic, $\bM^G$ as defined here is Postnikov's boundary measurement map followed by the map $V \mapsto V^{\perp}$. When $G$ is plabic, $\bM^G$ has domain $(\bC^*)^{E-V}$ and is an injection \cite{MSTwist} so long as $G$ is \emph{reduced}. We caution the reader that these properties do not necessarily hold when $G$ is not plabic, in which case $\bM^G$ can have a nontrivial exceptional locus (cf. \zcref{eg:twosix}), and can have degree greater than 1 (cf. \zcref{eg:threesix}). 
\end{rem}

\begin{eg}\label{eg:threesix}
    Consider the type (3,6) graph in \zcref{fig:threesix-labeling}. 
    The boundary measurement map $\bM^G$, computed in \zcref{fig:threesix}, is degree $2$. For example, both of the choices of weights in the table below

    \setcounter{MaxMatrixCols}{20}
    \[\begin{matrix}
a_2 &a_{\textbf{c}}& a_{\textbf{a}}& b_4& b_{\textbf{c}}& b_{\textbf{a}}& c_6& c_{\textbf{c}}& c_{\textbf{a}}& x_3& x_5& x_{\textbf{r}}& y_1& y_5& y_{\textbf{r}}& z_1& z_3& z_{\textbf{r}}\\
1& 2& 1& 1& 2& 1& 1& -2& 1& -2& 1& 1& -1& 2& 1& -1& 2& 1\\
1& -1& 1& 1& -1& 1& 1& -\frac{1}{2}& 1& -2& -2& 1& -4& 2& 1& -1& -4& 1
    \end{matrix}
        \]
    have the same image $V$ under boundary measurement, with Pl\"ucker coordinates
    \[
        (p_I(V))_{I \in \binom{[6]}{3}} = 
        (-2,-12,4,8,2,-1,-3,-2,-10,2,4,-3,-5,-10,-14,-2,1,-1,2,6).
    \]
    Thus, the fiber over $V$ has two points. One can compute that the Jacobian of $\bM^G$ does not vanish at either of the points in the fiber and so $\bM^G$ has degree at least 2. One can use the Bernstein–Khovanskii–Kushnirenko theorem to show that the degree is at most 2.
    This graph coincides with the graph for $\Omega_{24}$ of \cite[Appendix A]{Bourjaily:2016mnp} up to relabeling of boundary vertices. 
\end{eg}

\begin{figure}[!h]
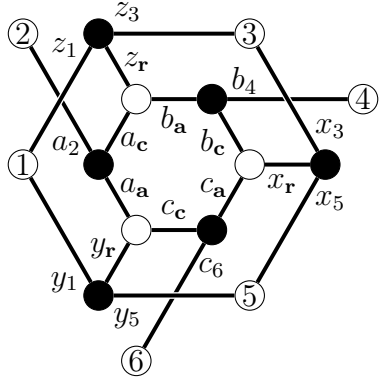

    \[
        \onshellgraphwithcode{{(139.1:2.645)/2,(19.1:2.645)/4,(-100.9:2.645)/6,(180:2)/1,(60:2)/3,(300:2)/5}}{{(0:1)/A,(120:1)/C,(240:1)/B,(139.1:2.645)/W2,(19.1:2.645)/W4,(-100.9:2.645)/W6,(180:2)/W1,(60:2)/W3,(300:2)/W5}}{{(180:1)/B1,(60:1)/B2,(300:1)/B3,(0:2)/B4,(240:2)/B5,(120:2)/B6}}{{B1/W2/black,B1/B/black,B1/C/black,B2/A/black,B2/W4/black,B2/C/black,B3/A/black,B3/B/black,B3/W6/black,B4/A/black,B4/W3/black,B4/W5/black,B5/W1/black,B5/B/black,B5/W5/black,B6/W1/black,B6/W3/black,B6/C/black}}{1}{
            \node() at (210:0.6) {$a_{\textbf{a}}$};
            \node() at (150:0.6) {$a_{\textbf{c}}$};
            \node() at (90:0.6) {$b_{\textbf{a}}$};
            \node() at (30:0.6) {$b_{\textbf{c}}$};
            \node() at (330:0.6) {$c_{\textbf{a}}$};
            \node() at (270:0.6) {$c_{\textbf{c}}$};
            \node() at (170:1.45) {$a_2$};
            \node() at (50:1.45) {$b_4$};
            \node() at (290:1.45) {$c_6$};
            \node() at (110:1.45) {$z_{\textbf{r}}$};
            \node() at (-10:1.45) {$x_{\textbf{r}}$};
            \node() at (230:1.45) {$y_{\textbf{r}}$};
            \node() at (133:2.12) {$z_1$};
            \node() at (107:2.12) {$z_3$};
            \node() at (13:2.12) {$x_3$};
            \node() at (-13:2.12) {$x_5$};
            \node() at (253:2.12) {$y_5$};
            \node() at (227:2.12) {$y_1$};
        }
    \]
    \caption{Labeling of edge weights on a type (3,6) graph. Subscripts indicate either the label of the incident boundary vertex or whether the edge leaves the black vertex in the \textbf{c}lockwise, \textbf{a}nticlockwise, or \textbf{r}adial direction.}\label{fig:threesix-labeling}
\end{figure}

\begin{figure}
    \[
        \begin{tikzpicture}[scale=0.7]
        	\node at (-5,0) {\coloredonshellgraph{{(139.1:2.645)/2,(19.1:2.645)/4,(-100.9:2.645)/6,(180:2)/1,(60:2)/3,(300:2)/5}}{{(0:1)/A,(120:1)/C,(240:1)/B,(139.1:2.645)/W2,(19.1:2.645)/W4,(-100.9:2.645)/W6,(180:2)/W1,(60:2)/W3,(300:2)/W5}}{{(180:1)/B1,(60:1)/B2,(300:1)/B3,(0:2)/B4,(240:2)/B5,(120:2)/B6}}{{B1/W2/black,B1/B/blue,B1/C/teal,B2/A/teal,B2/W4/black,B2/C/blue,B3/A/blue,B3/B/teal,B3/W6/black,B4/A/black,B4/W3/red,B4/W5/orange,B5/W1/orange,B5/B/black,B5/W5/red,B6/W1/red,B6/W3/orange,B6/C/black}}{0.7}};
        	\node at (-5,-3) {$\mathbb M^G_{246} = ({\color{blue} a_{\textbf{a}} b_{\textbf{a}} c_{\textbf{a}}} + {\color{teal} a_{\textbf{c}} b_{\textbf{c}} c_{\textbf{c}}})({\color{red} x_3 y_5 z_1} + {\color{orange} x_5 y_1 z_3})$};
        	
        	\node at (5,-0) {\coloredonshellgraph{{(139.1:2.645)/2,(19.1:2.645)/4,(-100.9:2.645)/6,(180:2)/1,(60:2)/3,(300:2)/5}}{{(0:1)/A,(120:1)/C,(240:1)/B,(139.1:2.645)/W2,(19.1:2.645)/W4,(-100.9:2.645)/W6,(180:2)/W1,(60:2)/W3,(300:2)/W5}}{{(180:1)/B1,(60:1)/B2,(300:1)/B3,(0:2)/B4,(240:2)/B5,(120:2)/B6}}{{B1/W2/blue,B1/B/black,B1/C/black,B2/A/black,B2/W4/blue,B2/C/black,B3/A/black,B3/B/black,B3/W6/blue,B4/A/blue,B4/W3/black,B4/W5/black,B5/W1/black,B5/B/blue,B5/W5/black,B6/W1/black,B6/W3/black,B6/C/blue}}{0.7}};
        	\node at (5,-3) {$\mathbb M^G_{135} = -{\color{blue} a_2 b_4 c_6 x_{\textbf{r}} y_{\textbf{r}} z_{\textbf{r}}}$};
        	
        	\node at (-5,-6) {\coloredonshellgraph{{(139.1:2.645)/2,(19.1:2.645)/4,(-100.9:2.645)/6,(180:2)/1,(60:2)/3,(300:2)/5}}{{(0:1)/A,(120:1)/C,(240:1)/B,(139.1:2.645)/W2,(19.1:2.645)/W4,(-100.9:2.645)/W6,(180:2)/W1,(60:2)/W3,(300:2)/W5}}{{(180:1)/B1,(60:1)/B2,(300:1)/B3,(0:2)/B4,(240:2)/B5,(120:2)/B6}}{{B1/W2/black,B1/B/blue,B1/C/black,B2/A/black,B2/W4/blue,B2/C/black,B3/A/black,B3/B/black,B3/W6/blue,B4/A/blue,B4/W3/black,B4/W5/black,B5/W1/black,B5/B/black,B5/W5/blue,B6/W1/black,B6/W3/black,B6/C/blue}}{0.7}};
        	\node at (-5,-9) {$\mathbb M^G_{123} = -{\color{blue} a_{\textbf{a}} b_4 c_6 x_{\textbf{r}} y_5 z_{\textbf{r}}}$};
        	
        	\node at (5,-6) {\coloredonshellgraph{{(139.1:2.645)/2,(19.1:2.645)/4,(-100.9:2.645)/6,(180:2)/1,(60:2)/3,(300:2)/5}}{{(0:1)/A,(120:1)/C,(240:1)/B,(139.1:2.645)/W2,(19.1:2.645)/W4,(-100.9:2.645)/W6,(180:2)/W1,(60:2)/W3,(300:2)/W5}}{{(180:1)/B1,(60:1)/B2,(300:1)/B3,(0:2)/B4,(240:2)/B5,(120:2)/B6}}{{B1/W2/black,B1/B/orange,B1/C/red,B2/A/black,B2/W4/blue,B2/C/black,B3/A/black,B3/B/black,B3/W6/blue,B4/A/blue,B4/W3/black,B4/W5/black,B5/W1/orange,B5/B/red,B5/W5/black,B6/W1/red,B6/W3/black,B6/C/orange}}{0.7}};
        	\node at (5,-9) {$\mathbb M^G_{235} = {\color{blue} b_4 c_6 x_{\textbf{r}}}({\color{red} a_{\textbf{c}} y_{\textbf{r}} z_1} + {\color{orange} a_{\textbf{a}} y_1 z_{\textbf{r}}})$};
            
        	\node at (-6,-12) {\coloredonshellgraph{{(139.1:2.645)/2,(19.1:2.645)/4,(-100.9:2.645)/6,(180:2)/1,(60:2)/3,(300:2)/5}}{{(0:1)/A,(120:1)/C,(240:1)/B,(139.1:2.645)/W2,(19.1:2.645)/W4,(-100.9:2.645)/W6,(180:2)/W1,(60:2)/W3,(300:2)/W5}}{{(180:1)/B1,(60:1)/B2,(300:1)/B3,(0:2)/B4,(240:2)/B5,(120:2)/B6}}{{B1/W2/black,B1/B/black,B1/C/blue,B2/A/blue,B2/W4/black,B2/C/black,B3/A/black,B3/B/black,B3/W6/blue,B4/A/black,B4/W3/black,B4/W5/blue,B5/W1/black,B5/B/blue,B5/W5/black,B6/W1/black,B6/W3/blue,B6/C/black}}{0.7}};
        	\node at (0,-12) {\coloredonshellgraph{{(139.1:2.645)/2,(19.1:2.645)/4,(-100.9:2.645)/6,(180:2)/1,(60:2)/3,(300:2)/5}}{{(0:1)/A,(120:1)/C,(240:1)/B,(139.1:2.645)/W2,(19.1:2.645)/W4,(-100.9:2.645)/W6,(180:2)/W1,(60:2)/W3,(300:2)/W5}}{{(180:1)/B1,(60:1)/B2,(300:1)/B3,(0:2)/B4,(240:2)/B5,(120:2)/B6}}{{B1/W2/black,B1/B/red,B1/C/black,B2/A/red,B2/W4/black,B2/C/black,B3/A/black,B3/B/black,B3/W6/red,B4/A/black,B4/W3/red,B4/W5/black,B5/W1/black,B5/B/black,B5/W5/red,B6/W1/black,B6/W3/black,B6/C/red}}{0.7}};
        	\node at (6,-12) {\coloredonshellgraph{{(139.1:2.645)/2,(19.1:2.645)/4,(-100.9:2.645)/6,(180:2)/1,(60:2)/3,(300:2)/5}}{{(0:1)/A,(120:1)/C,(240:1)/B,(139.1:2.645)/W2,(19.1:2.645)/W4,(-100.9:2.645)/W6,(180:2)/W1,(60:2)/W3,(300:2)/W5}}{{(180:1)/B1,(60:1)/B2,(300:1)/B3,(0:2)/B4,(240:2)/B5,(120:2)/B6}}{{B1/W2/black,B1/B/orange,B1/C/black,B2/A/black,B2/W4/black,B2/C/orange,B3/A/black,B3/B/black,B3/W6/orange,B4/A/orange,B4/W3/black,B4/W5/black,B5/W1/black,B5/B/black,B5/W5/orange,B6/W1/black,B6/W3/orange,B6/C/black}}{0.7}};
        	\node at (0,-15) {$\mathbb M^G_{124} = -({\color{blue} a_{\textbf{c}} b_{\textbf{c}} c_6 x_5 y_{\textbf{r}} z_3} - {\color{red} a_{\textbf{a}} b_{\textbf{c}} c_6 x_3 y_5 z_{\textbf{r}}} - {\color{orange} a_{\textbf{a}} b_{\textbf{a}} c_6 x_{\textbf{r}} y_5 z_3})$};
            
        	\node at (-6,-18) {\coloredonshellgraph{{(139.1:2.645)/2,(19.1:2.645)/4,(-100.9:2.645)/6,(180:2)/1,(60:2)/3,(300:2)/5}}{{(0:1)/A,(120:1)/C,(240:1)/B,(139.1:2.645)/W2,(19.1:2.645)/W4,(-100.9:2.645)/W6,(180:2)/W1,(60:2)/W3,(300:2)/W5}}{{(180:1)/B1,(60:1)/B2,(300:1)/B3,(0:2)/B4,(240:2)/B5,(120:2)/B6}}{{B1/W2/black,B1/B/black,B1/C/blue,B2/A/blue,B2/W4/black,B2/C/black,B3/A/black,B3/B/black,B3/W6/blue,B4/A/black,B4/W3/blue,B4/W5/black,B5/W1/black,B5/B/blue,B5/W5/black,B6/W1/blue,B6/W3/black,B6/C/black}}{0.7}};
        	\node at (0,-18) {\coloredonshellgraph{{(139.1:2.645)/2,(19.1:2.645)/4,(-100.9:2.645)/6,(180:2)/1,(60:2)/3,(300:2)/5}}{{(0:1)/A,(120:1)/C,(240:1)/B,(139.1:2.645)/W2,(19.1:2.645)/W4,(-100.9:2.645)/W6,(180:2)/W1,(60:2)/W3,(300:2)/W5}}{{(180:1)/B1,(60:1)/B2,(300:1)/B3,(0:2)/B4,(240:2)/B5,(120:2)/B6}}{{B1/W2/black,B1/B/red,B1/C/black,B2/A/red,B2/W4/black,B2/C/black,B3/A/black,B3/B/black,B3/W6/red,B4/A/black,B4/W3/red,B4/W5/black,B5/W1/red,B5/B/black,B5/W5/black,B6/W1/black,B6/W3/black,B6/C/red}}{0.7}};
        	\node at (6,-18) {\coloredonshellgraph{{(139.1:2.645)/2,(19.1:2.645)/4,(-100.9:2.645)/6,(180:2)/1,(60:2)/3,(300:2)/5}}{{(0:1)/A,(120:1)/C,(240:1)/B,(139.1:2.645)/W2,(19.1:2.645)/W4,(-100.9:2.645)/W6,(180:2)/W1,(60:2)/W3,(300:2)/W5}}{{(180:1)/B1,(60:1)/B2,(300:1)/B3,(0:2)/B4,(240:2)/B5,(120:2)/B6}}{{B1/W2/black,B1/B/orange,B1/C/black,B2/A/black,B2/W4/black,B2/C/orange,B3/A/black,B3/B/black,B3/W6/orange,B4/A/orange,B4/W3/black,B4/W5/black,B5/W1/orange,B5/B/black,B5/W5/black,B6/W1/black,B6/W3/orange,B6/C/black}}{0.7}};
        	\node at (0,-21) {$\mathbb M^G_{245} = -({\color{blue} a_{\textbf{c}} b_{\textbf{c}} c_6 x_3 y_{\textbf{r}} z_1} + {\color{red} a_{\textbf{a}} b_{\textbf{c}} c_6 x_3 y_1 z_{\textbf{r}}} + {\color{orange} a_{\textbf{a}} b_{\textbf{a}} c_6 x_{\textbf{r}} y_1 z_3})$};
        \end{tikzpicture}
    \]
    \caption{Some boundary measurements $\bM_{ijk}^G$ for a type (3,6) graph, using the labeling from \zcref{fig:threesix-labeling}. Those $\bM_{ijk}^G$ not shown can be obtained by rotation or reflection. In each expression, the variable is the same color as the edge it labels. For e.g. $\bM_{235}^G$, the matchings consist of the blue edges together with either the red edges or the orange edges. }
    \label{fig:threesix}
\end{figure}
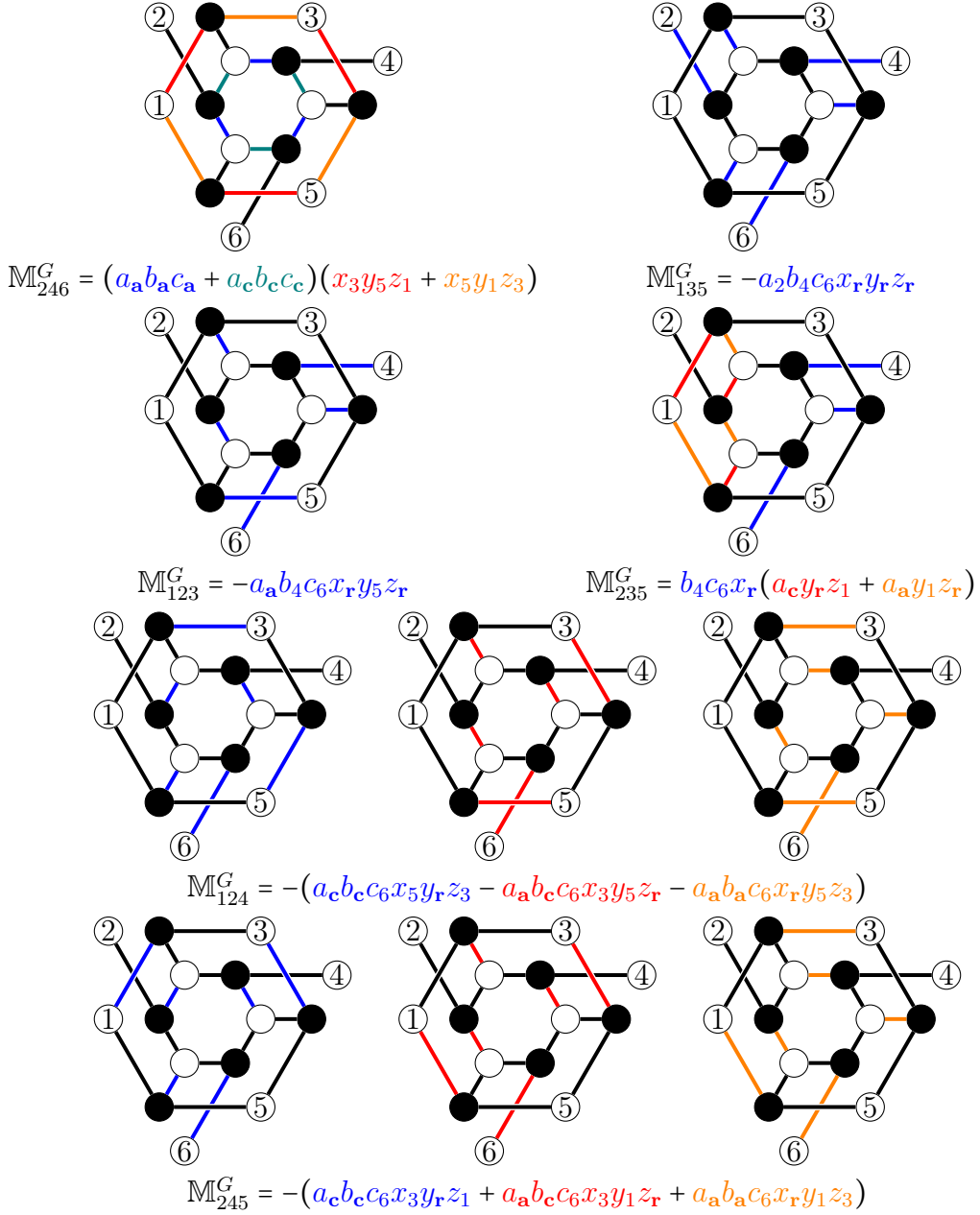

Finally, we note that the boundary measurement map respects the natural action of $(\bC^*)^n$ on the Grassmannian $\Gr_{k,n}.$ This action can be defined either by scaling the columns of a $k \times n$ matrix representative of each point, or directly on the Pl\"ucker coordinates by 
\begin{equation}\label{eq:torus action grkn}
    (t_1, \dots, t_n) \star p_{i_1 \dots i_k} := t_{i_1} \dots t_{i_k} p_{i_1 \dots i_k} \quad \forall 1 \leq i_1 < \ldots < i_k \leq n.
\end{equation}
Similarly, we define an action of $(\bC^*)^n$ on $(\bC^*)^E$ by
\begin{equation}\label{eq:torus action edges}
    (t_1, \dots, t_n) \cdot x_e  = \begin{cases}
        t_ix_e & \text{if } \text{boundary vertex }i \in e \\
        x_e & \text{else}.
    \end{cases}
\end{equation}
The boundary measurement map transforms the torus action via the following lemma.
\begin{lemma}\label{lem:bdry-meas-torus-eq}
    Fix an immersed type $(k,n)$ graph $G$. Let $ t=(t_1, \dots, t_n)\in (\bC^*)^n$.
    Then for every $I \in \binom{[n]}{k},$ we have \[t \star \bM^G_I((x_e)_{e \in E}) = t_1 \dots t_n \bM^G_I(t^{-1} \cdot (x_e)_{e \in E}).\]
\end{lemma}
\begin{proof}
    Using the action on the edge variables, we have 
    \begin{align*}
        t \star\bM^G_I((x_e)_{e \in E})) & = (\prod_{i \in I} t_i) \, \bM^G_I((x_e)_{e \in E}))\\& = (-1)^{\tot(I)} t_1 \dots t_n \prod_{i \in I^c} t^{-1}_i \sum_{M: \partial M=I^c} (-1)^{\cross(M)}\prod_{e \in M} x_e\\  
        & = t_1 \dots t_n\bM^G_I(t^{-1} \cdot (x_e)_{e \in E}).
    \end{align*}
\end{proof}

\subsection{Moves and reduced graphs}

There are a number of local moves for weighted immersed graphs which preserve boundary measurement, shown in \zcref{fig:moves}. We emphasize that these moves can be performed only in sub-disks where $G$ looks exactly as in \zcref{fig:moves}; no extra edges or vertices should be present.  Note that each of these moves preserves the type $(k,n)$ of a graph, as it leaves $\partial G$ fixed and increases or decreases $|W|$ and $|B|$ by the same number. 

\begin{figure}[h]
    \centering
    \[
        \def\outline{gray}
        \begin{tikzpicture}
            \node() at (0,0) {
                $
                    \vcenter{\hbox{
                    \begin{tikzpicture}
                        \draw[dashed,gray] (0,0) circle (1.5);
                        \clip (0,0) circle (1.5);
                        \draw[line width=1.5pt] (0,-0.7) -- ++(120:2.5);
                        \draw[line width=1.5pt] (0,-0.7) -- ++(60:2.5);
                        \draw[line width=1.5pt] (0,-0.7) -- ++(-120:2.5);
                        \draw[line width=1.5pt] (0,-0.7) -- ++(-60:2.5);
                        \node() at (0,-0.3) {$\cdots$};
                        \node() at (0,-1.2) {$\cdots$};
                        \draw[line width=2.5pt,white] (-1.5,0) -- (1.5,0);
                        \draw[line width=1.5pt] (-1.5,0) -- (1.5,0);
                        \node() at (-1,0.15) {\tiny $x_q$};
                        \node[fill,white,circle] at (0,-0.7) {};
                        \draw[pattern=north east lines] (0,-0.7) circle (0.2);
                    \end{tikzpicture}
                    }}
                    \underset{\text{(PV)}}{\longleftrightarrow}
                    \vcenter{\hbox{
                    \begin{tikzpicture}
                        \draw[dashed,gray] (0,0) circle (1.5);
                        \clip (0,0) circle (1.5);
                        \draw[line width=1.5pt] (0,0.7) -- ++(120:2.5);
                        \draw[line width=1.5pt] (0,0.7) -- ++(60:2.5);
                        \draw[line width=1.5pt] (0,0.7) -- ++(-120:2.5);
                        \draw[line width=1.5pt] (0,0.7) -- ++(-60:2.5);
                        \node() at (0,1.1) {$\cdots$};
                        \node() at (0,0.2) {$\cdots$};
                        \draw[line width=2.5pt,white] (-1.5,0) -- (1.5,0);
                        \draw[line width=1.5pt] (-1.5,0) -- (1.5,0);
                        \node() at (-1,0.15) {\tiny $-x_q$};
                        \node[fill,white,circle] at (0,0.7) {};
                        \draw[pattern=north east lines] (0,0.7) circle (0.2);
                    \end{tikzpicture}
                    }}
                $
            };
            \node() at (0,-4) {
                $
                    \vcenter{\hbox{
                    \begin{tikzpicture}
                        \draw[dashed,gray] (0,0) circle (1.5);
                        \clip (0,0) circle (1.5);
                        \draw[line width=1.5] (-1.5,-1) -- (-0.7,0);
                        \draw[line width=1.5] (-1.5,1) -- (-0.7,0);
                        \draw[line width=1.5] (1.5,-1) -- (0.7,0);
                        \draw[line width=1.5] (1.5,1) -- (0.7,0);
                        \node() at (-1.2,0) {$\vdots$};
                        \node() at (1.2,0) {$\vdots$};
                        \draw[line width=1.5] (-0.7,0) -- (0,0.3) -- (0.7,0);
                        \draw[line width=1.5] (-0.7,0) -- (0,-0.3) -- (0.7,0);
                        \node() at (0,0.55) {\tiny $x_1$};
                        \node() at (0,-0.55) {\tiny $x_2$};
                        \node[draw,fill,circle] at (0.7,0) {};
                        \node[fill,white,circle] at (-0.7,0) {};
                        \node[draw,circle] at (-0.7,0) {};
                    \end{tikzpicture}
                    }}
                    \underset{\text{(B)}}{\longrightarrow}
                    \vcenter{\hbox{
                    \begin{tikzpicture}
                        \draw[dashed,gray] (0,0) circle (1.5);
                        \clip (0,0) circle (1.5);
                        \draw[line width=1.5] (-1.5,-1) -- (-0.7,0);
                        \draw[line width=1.5] (-1.5,1) -- (-0.7,0);
                        \draw[line width=1.5] (1.5,-1) -- (0.7,0);
                        \draw[line width=1.5] (1.5,1) -- (0.7,0);
                        \node() at (-1.2,0) {$\vdots$};
                        \node() at (1.2,0) {$\vdots$};
                        \draw[line width=1.5] (-0.7,0) -- (0.7,0) node[midway,above] {\tiny $x_1 + x_2$};
                        \node[draw,fill,circle] at (0.7,0) {};
                        \node[fill,white,circle] at (-0.7,0) {};
                        \node[draw,circle] at (-0.7,0) {};
                    \end{tikzpicture}
                    }}
                $
            };
            \node() at (8,1) {
                $
                    \vcenter{\hbox{
                    \begin{tikzpicture}
                        \draw[dashed,gray] (0,0) ellipse (1.5 and 0.7);
                        \clip (0,0) ellipse (1.5 and 0.7);
                        \draw[line width=1.5] (-1.5,-1) -- (-0.7,0);
                        \draw[line width=1.5] (-1.5,1) -- (-0.7,0);
                        \draw[line width=1.5] (1.5,-1) -- (0.7,0);
                        \draw[line width=1.5] (1.5,1) -- (0.7,0);
                        \node() at (-1.2,0) {$\vdots$};
                        \node() at (1.2,0) {$\vdots$};
                        \draw[line width=1.5] (-0.5,0) -- (-0.2,0) node[midway,above] {\tiny $1$};
                        \draw[line width=1.5] (0.2,0) -- (0.5,0) node[midway,above] {\tiny $1$};

                        \node[draw,fill,circle] at (0,0) {};
                        \node[fill,white,circle] at (0.7,0) {};
                        \node[fill,white,circle] at (-0.7,0) {};
                        \node[draw,circle] at (0.7,0) {};
                        \node[draw,circle] at (-0.7,0) {};
                    \end{tikzpicture}
                    }}
                    \underset{\text{(EC)}}{\longleftrightarrow}
                    \vcenter{\hbox{
                    \begin{tikzpicture}
                        \draw[dashed,gray] (0,0) ellipse (1.5 and 0.7);
                        \clip (0,0) ellipse (1.5 and 0.7);
                        \draw[line width=1.5] (-1.5,-1) -- (0,0);
                        \draw[line width=1.5] (-1.5,1) -- (0,0);
                        \draw[line width=1.5] (1.5,-1) -- (0,0);
                        \draw[line width=1.5] (1.5,1) -- (0,0);
                        \node() at (-1.2,0) {$\vdots$};
                        \node() at (1.2,0) {$\vdots$};
                        \node[fill,white,circle] at (0,0) {};
                        \node[draw,circle] at (0,0) {};
                    \end{tikzpicture}
                    }}
                $
            };
            \node() at (8,-1) {
                $
                    \vcenter{\hbox{
                    \begin{tikzpicture}
                        \draw[dashed,gray] (0,0) ellipse (1.5 and 0.7);
                        \clip (0,0) ellipse (1.5 and 0.7);
                        \draw[line width=1.5] (-1.5,-1) -- (-0.7,0);
                        \draw[line width=1.5] (-1.5,1) -- (-0.7,0);
                        \draw[line width=1.5] (1.5,-1) -- (0.7,0);
                        \draw[line width=1.5] (1.5,1) -- (0.7,0);
                        \node() at (-1.2,0) {$\vdots$};
                        \node() at (1.2,0) {$\vdots$};
                        \draw[line width=1.5] (-0.5,0) -- (-0.2,0) node[midway,above] {\tiny $1$};
                        \draw[line width=1.5] (0.2,0) -- (0.5,0) node[midway,above] {\tiny $1$};

                        \node[draw,circle] at (0,0) {};
                        \node[fill,circle] at (0.7,0) {};
                        \node[fill,circle] at (-0.7,0) {};
                    \end{tikzpicture}
                    }}
                    \underset{\text{(EC)}}{\longleftrightarrow}
                    \vcenter{\hbox{
                    \begin{tikzpicture}
                        \draw[dashed,gray] (0,0) ellipse (1.5 and 0.7);
                        \clip (0,0) ellipse (1.5 and 0.7);
                        \draw[line width=1.5] (-1.5,-1) -- (0,0);
                        \draw[line width=1.5] (-1.5,1) -- (0,0);
                        \draw[line width=1.5] (1.5,-1) -- (0,0);
                        \draw[line width=1.5] (1.5,1) -- (0,0);
                        \node() at (-1,0) {$\vdots$};
                        \node() at (1,0) {$\vdots$};
                        \node[fill,circle] at (0,0) {};
                    \end{tikzpicture}
                    }}
                $
            };
            \node() at (8,-4) {
                $
                    \vcenter{\hbox{
                    \begin{tikzpicture}
                        \draw[dashed,gray] (0,0) circle (1.5);
                        \clip (0,0) circle (1.5);
                        \node[draw,circle](TL) at (135:1.1) {};
                        \node[draw,fill,circle](TR) at (45:1.1) {};
                        \node[draw,fill,circle](BL) at (-135:1.1) {};
                        \node[draw,circle](BR) at (-45:1.1) {};
                        \node[rotate=45] at (135:1.35) {$\cdots$};
                        \node[rotate=-45] at (45:1.35) {$\cdots$};
                        \node[rotate=-45] at (-135:1.45) {$\cdots$};
                        \node[rotate=45] at (-45:1.45) {$\cdots$};
                        \draw[line width=1.5] (TL) -- ++(-1,0);
                        \draw[line width=1.5] (TL) -- ++(0,1);
                        \draw[line width=1.5] (TR) -- ++(1,0);
                        \draw[line width=1.5] (TR) -- ++(0,1);
                        \draw[line width=1.5] (BL) -- ++(-1,0);
                        \draw[line width=1.5] (BL) -- ++(0,-1);
                        \draw[line width=1.5] (BR) -- ++(1,0);
                        \draw[line width=1.5] (BR) -- ++(0,-1);
                        \draw[line width=1.5] (TL) -- (TR) node[above,midway] {\tiny $x_1$};
                        \draw[line width=1.5] (BR) -- (TR) node[right,midway] {\tiny $x_2$};
                        \draw[line width=1.5] (BR) -- (BL) node[below,midway] {\tiny $x_3$};
                        \draw[line width=1.5] (TL) -- (BL) node[left,midway] {\tiny $x_4$};
                    \end{tikzpicture}
                    }}
                    \underset{\text{(S)}}{\longleftrightarrow}
                    \vcenter{\hbox{
                    \begin{tikzpicture}
                        \draw[dashed,gray] (0,0) circle (1.5);
                        \clip (0,0) circle (1.5);
                        \node[draw,circle](TL) at (135:1.1) {};
                        \node[draw,fill,circle](TR) at (45:1.1) {};
                        \node[draw,fill,circle](BL) at (-135:1.1) {};
                        \node[draw,circle](BR) at (-45:1.1) {};
                        \node[rotate=45] at (135:1.35) {$\cdots$};
                        \node[rotate=-45] at (45:1.35) {$\cdots$};
                        \node[rotate=-45] at (-135:1.45) {$\cdots$};
                        \node[rotate=45] at (-45:1.45) {$\cdots$};
                        \node[draw,fill,circle](TLp) at (135:0.45) {};
                        \node[draw,circle](TRp) at (45:0.45) {};
                        \node[draw,circle](BLp) at (-135:0.45) {};
                        \node[draw,fill,circle](BRp) at (-45:0.45) {};
                        \draw[line width=1.5] (TL) -- ++(-1,0);
                        \draw[line width=1.5] (TL) -- ++(0,1);
                        \draw[line width=1.5] (TR) -- ++(1,0);
                        \draw[line width=1.5] (TR) -- ++(0,1);
                        \draw[line width=1.5] (BL) -- ++(-1,0);
                        \draw[line width=1.5] (BL) -- ++(0,-1);
                        \draw[line width=1.5] (BR) -- ++(1,0);
                        \draw[line width=1.5] (BR) -- ++(0,-1);
                        \draw[line width=1.5] (TLp) -- (TRp) node[above,midway] {\tiny $x_3'$};
                        \draw[line width=1.5] (BRp) -- (TRp) node[right,midway] {\tiny $x_4'$};
                        \draw[line width=1.5] (BRp) -- (BLp) node[below,midway] {\tiny $x_1'$};
                        \draw[line width=1.5] (TLp) -- (BLp) node[left,midway] {\tiny $x_2'$};
                        \draw[line width=1.5] (TLp) -- (TL) node[shift={(0.1,0.1)},midway] {\tiny $1$};
                        \draw[line width=1.5] (TRp) -- (TR) node[shift={(-0.1,0.1)},midway] {\tiny $1$};
                        \draw[line width=1.5] (BRp) -- (BR) node[shift={(-0.1,-0.1)},midway] {\tiny $1$};
                        \draw[line width=1.5] (BLp) -- (BL) node[shift={(0.1,-0.1)},midway] {\tiny $1$};
                    \end{tikzpicture}
                    }}
                $
            };
            \draw (-4,-2) -- (12,-2);
            \draw (4,2) -- (4,-6);
        \end{tikzpicture}
    \]
    \caption{Local moves for weighted immersed graphs preserving boundary measurement. All edge weights not shown are the same in both graphs. In (EC), top, and (B) one white vertex may be a boundary vertex; in (PV) and (S) all vertices are internal. In the square move, we define $x'_i := \frac{x_i}{x_1x_3 + x_2 x_4}$. %
    }
    \label{fig:moves}
\end{figure}
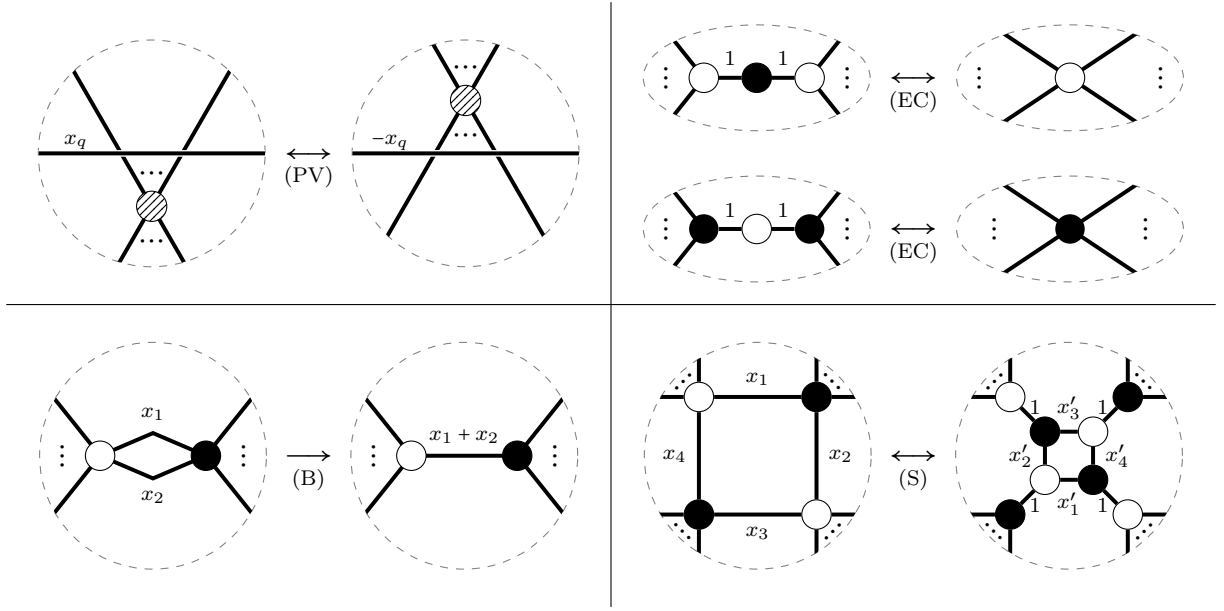

The expand/contract move (EC), square move (S), and bigon removal (B) are familiar from plabic graphs. For nonplanar graphs, we also have a ``push vertex'' move (PV) which pushes a vertex across an edge; this changes the immersion but not the underlying graph. By \cite[Proof of Proposition 1.1]{Speyer16}, any disk immersion of $G$ can be taken to any other by a sequence of (PV) moves.

\begin{lemma}\label{lem:boundary-measurement-moves}
    Let $G, G'$ be immersed type $(k,n)$ graphs which differ by a (PV), (EC), (S) or (B) move (see \zcref{fig:moves}). Let $\phi: (\bC^*)^{E-V} \dashrightarrow (\bC^*)^{E'-V'}$ be the corresponding (rational) map of edge weights, indicated in \zcref{fig:moves}.
    Then $\bM^{G'} \circ \phi = \bM^{G}$. 
\end{lemma}
\begin{proof}
    For the moves (EC), (S), (B), this is well-known from the plabic case, see e.g. \cite[Proposition 4.8]{Lam-grassmann-polytope}. 
    
    For (PV): let $v$ be the vertex pushed across edge $q$. Identifying the edge and vertex sets of $G, G'$, the effect of the map $\phi$ is to change $x_q$ to $-x_q$ and leave all other edge weights as they are. So $\bM^{G'} \circ \phi(x_e)$ is computed by taking the formula for $\bM^{G'}$ and replacing $x_q$ with $-x_q$. Let $x_q':=-x_q$ and $x_e':=x_e$ for $e \neq q$.

    To show $\bM^{G'} \circ \phi = \bM^{G}$, we need to verify that for any matching $M$ of $G$,
    \[(-1)^{\cross_G(M)} \prod_{e \in M} x_e = (-1)^{\cross_{G'}(M)} \prod_{e \in M} x_e',\]
    where $\cross_G(M)$ indicates crossings computed in the immersion $G$.
    If $q \notin M$, this equality is clear. If $q \in M$, then $(-1)^{\cross_G(M)} = -(-1)^{\cross_{G'}(M)}$. Indeed, $M$ must include exactly one edge $e$ adjacent to $v$; this edge crosses $q$ in $G$ if and only if it does not cross $q$ in $G'$. We also have $\prod_{e \in M} x_e' = -\prod_{e \in M} x_e$ since $q \in M$.
\end{proof}

Since all disk immersions of $G$ are related by (PV) moves, the proof of \zcref{lem:boundary-measurement-moves} for the (PV) moves implies the following corollary.

\begin{cor}\label{cor:image-doesnt-depend-on-immersion}
    The image of $\bM^G$ for an immersed graph $G$ does not depend on the choice of immersion.
\end{cor}

In the case of the (B) move, \zcref{lem:boundary-measurement-moves} shows that one can delete one edge of a bigon without changing the image of boundary measurement. In particular, the image of boundary measurement is smaller in dimension than its domain. We will be most interested in graphs where this does not happen. For technical reasons, we will need the following combinatorial notion: $G$ is \emph{leafless} if all degree 1 internal vertices are adjacent to a boundary vertex\footnote{In particular, any degree 1 vertex must be black.}.

\begin{definition}\label{def:reduced-top-cell} 
    We call $G$ \emph{reduced} if it is leafless and the image of $\bM^G$ is the same dimension as $(\bC^*)^{E-V}$, or equivalently, if $\bM^G$ is generically finite onto its image.  We call $G$ a \emph{top-cell graph} if it is reduced and $\bM^G$ is a dominant rational map, equivalently its image is Zariski-dense in $\Gr_{k,n}$.
\end{definition}

Note that \zcref{lem:boundary-measurement-moves} and \zcref{cor:image-doesnt-depend-on-immersion} imply that the properties of being reduced or top-cell are properties of type $(k,n)$ graphs rather than immersed graphs.

When $G$ is plabic, $G$ is reduced if and only if no graph related to $G$ by (EC) and (S) moves admits a (B) move \cite{postnikov}. There are various other combinatorial characterizations of reducedness for plabic graphs involving \emph{trips}. In the non-planar setting, we do not have a characterization of reduced graphs, but do have some combinatorial information.

\begin{lemma}\label{lem:reduced-conditions}
    Suppose $G$ is reduced. Then 
    \begin{enumerate}
        \item every edge of $G$ is used in some perfect matching;
        \item every connected component of $G$ contains a boundary vertex, or equivalently $G/\bdG$ is connected;
        \item any graph obtained from $G$ by a sequence of (PV), (EC), (S) moves is also reduced.
    \end{enumerate}

    On the other hand, if any graph obtained from $G$ by (PV), (EC) and (S) moves admits a (B) move, then $G$ is not reduced.
\end{lemma}

\begin{proof}
    For (1) and (2), we prove the contrapositive.

    (2): Suppose $G$ has a connected component $C$ disjoint from $\bdG$. Let $G'$ denote the graph obtained by deleting all vertices of $C$ from $G$. Choose an immersion where the edges of $C$ do not intersect any edges of $G'$. Note that any perfect matching $M$ of $G$ restricts to a (usual) perfect matching of $C$, where each vertex of $C$ is in exactly one edge of $M$. So, for this immersion, each $\bM^G_I(x_e)$ factors as $\bM^C(x_e) \cdot \bM^{G'}_I(x_e)$, where $\bM^C(x_e)$ is the sum over all perfect matchings $M$ of $C$ of $(-1)^{\cross(M)} \prod_{e \in M} x_e$. The edges of $C$ contribute only to $\bM^C(x_e)$. So the image of $\bM^G$ is the same as the image of $\bM^{G'}$. 
    
    Say $C$ has edges $E_C$, and vertices $V_C$ (all of which are internal), while $G'$ has edges $E'$ and internal vertices $V'$. Let $c$ be the number of connected components of $G/\bdG$, so $G'/\bdG'$ has $c-1$ connected components.
    
    We have 
    $E= E' \sqcup E_C, V= V' \sqcup V_C$ so
    \[ \dim (\bC^*)^{E'-V'}=|E'| - |V'| + (c-2) = |E| -|V| + (c-1) -(|E_C| -|V_C| + 1).\]
    Note that $C$ is not a tree since $G$ is leafless, and thus $|E_C| -|V_C| + 1 > 0$. By \zcref{cor:basis-for-torus-mod-gauge}, the domain $(\bC^*)^{E' - V'}$ of the map $\bM^{G'}$ has dimension smaller than the domain $(\bC^*)^{E - V}$ of $\bM^G$. So $G$ cannot be reduced.
    
    (1): Suppose $G$ has an edge $e$ that is not used in any perfect matching. We may assume $G$ has no isolated components. Performing (EC) moves involving only internal vertices will not change this property, so we may assume that each endpoint of $e$ which is not a boundary vertex or adjacent to a boundary vertex has degree at least 3. Let $G'=G \setminus e$; note that $G'$ is again leafless. Then $\bM^G(x_p) = \bM^{G'}(x_p)$, where on the right the edge weight $x_e$ is omitted. So the images of $\bM^G$ and $\bM^{G'}$ are the same.
    
    Either $G'$ has an isolated component, or it does not. In the former case, \zcref{cor:basis-for-torus-mod-gauge} implies the domains of $\bM^G$ and $\bM^{G'}$ have the same dimension. By the above argument, the domain of $\bM^{G'}$ has dimension larger than the image. The same is true of $\bM^G$, so $G$ is not reduced. In the latter case, by \zcref{cor:basis-for-torus-mod-gauge}, the domain of $\bM^{G'}$ has dimension one smaller than the domain of $\bM^G$, so the image of $\bM^G$ has dimension smaller than its domain, and again $G$ is not reduced.

    (3): This follows from \zcref{lem:boundary-measurement-moves}, which shows that doing a move has the effect of pre-composing boundary measurement with a birational map. This will not change the dimension of the image. 

    The statement that a graph $G$ with a bigon is not reduced also follows from \zcref{lem:boundary-measurement-moves}.
\end{proof}

\subsection{Vector-relation configurations}
There is another perspective on the boundary measurement map, using \emph{vector-relation configurations}. These are very closely related to the ``relation space'' of an (oriented, not necessarily planar) bicolored graph $N$, studied in \cite[Section 14]{Lam-grassmann-polytope}. Vector-relation configurations were introduced for planar graphs in \cite{Affolter_2023}, with slightly different conventions. Both \cite[Proposition 7.3]{Affolter_2023} and \cite[Theorem 14.6]{Lam-grassmann-polytope} relate vector-relation configurations/relation spaces to the boundary measurement map in the planar case, and both were inspired by the physics formulation of on-shell forms\footnote{In \cite{Arkani-Hamed:2012zlh}, the authors associate a pair of vectors $\lambda_e,\widetilde{\lambda}_e\in\bC^2$ to each edge. Each black and white vertex imposes constraints on one type of vectors, analogous to (2) from \zcref{def:VRC}.
} in \cite{Arkani-Hamed:2012zlh}.

\begin{definition}\label{def:VRC}
    Fix a type $(k,n)$ graph $G$. A \emph{vector-relation configuration} (VRC) on $G$ is an assignment of vectors $u_w \in \bC^k$ 
    to each white vertex of $G$ and numbers $r_e \in \bC^*$ to each edge of $G$, such that 
    \begin{enumerate}
    \item the \emph{boundary vectors} $u_1, \dots, u_n$ on vertices $1, \dots, n$ span $\bC^k$;
    \item
    for each black vertex $b$, the following relation holds
    \[\sum_{\{b,w\}=e \in E} r_e u_w =0.\]
   \end{enumerate} 
    We write $(\bu, \bfR):=(\{u_w\}_{w \in W \cup [n]}, \{r_e\}_{e \in E})$ for a VRC. The \emph{boundary} of $(\bu, \bfR)$ is the $k \times n$ matrix of boundary vectors
    \[\partial(\bu, \bfR):= \begin{bmatrix}
        u_1 \dots u_n
    \end{bmatrix}.\]
\end{definition}

Note that $\GL_k$ acts simultaneously on all vectors in a VRC. Further, VRCs also admit a gauge action of $\bC^*$ at each internal vertex: at a vertex $v$, $\bC^*$ acts by $r_e \mapsto t r_e$ for all adjacent edges $e$, and, if the vertex $v=w$ is white, additionally by rescaling the vector $u_w \mapsto 1/t u_w$. We write $[\bu, \bfR]$ for the gauge- and $\GL_k$-equivalence class of $(\bu, \bfR)$, and $\partial[\bu, \bfR]$ for the element of $\Gr_{k,n}$ represented by the boundary matrix of $(\bu, \bfR)$.

In the next lemma, if $A$ is a matrix and $I$ is a subset of columns, we say that $A$ is \emph{$I$-full rank} if $A$ is full rank and the submatrix $A_I$ on columns $I$ is also full-rank. Recall that $K$ denotes the $|B| \times |W \cup \bdG|$ adjacency matrix of $G$. 

\begin{lemma}\label{lem:bdry-meas-vs-VRC}
    Suppose $G$ is type $(k,n)$. Fix a disk immersion of $G$, and choose immersion signs $\sigma=(\sigma_e)$. Let $[r_e] \in (\bC^*)^{E-V}$.
    \begin{itemize}
        \item If $K(r_e)$ is full rank, then there is at most one VRC on $G$ with coefficients $(r_e)$. If $K(r_e)$ is $W$-full rank, then there is exactly one VRC on $G$ with coefficients $(r_e)$
        \item  Suppose that $(\sigma_er_e)$ is in the domain of $\bM^G$, or equivalently that $K(r_e)$ is $W$-full rank.
    Then $ \bM^G(\sigma_e r_e)$ is equal to the boundary $\partial[\bu, (r_e)]$ of the unique VRC on $G$ with coefficients $(r_e)$.
    \end{itemize}
\end{lemma}

\begin{proof}
    A VRC with coefficients $(r_e)$ is the same as a choice of vectors $u_{w_1}, \dots u_{w_N}, u_1, \dots, u_n \in \bC^{k}$ such that $u_1, \dots, u_n$ span $\bC^k$ and satisfy
    \[K(r_e) \begin{bmatrix}
        - u_{w_1}^T - \\
        \vdots\\
          - u_{n}^T - \\
    \end{bmatrix}=0.\]
    That is, a VRC $[\bu, (r_e)]$ is the same as the choice of a $k$-dimensional subspace of the kernel of $K(r_e)$, where at least one Pl\"ucker coordinate using only the last $n$ indices is nonzero. By assumption, $K(r_e)$ is full rank. Since $K(r_e)$ has $k$ more columns than rows, its kernel is $k$-dimensional. So there is at most one VRC with coefficients $(r_e)$. If $K$ is $W$-full rank, some minor $\Delta_{W \cup I^c}$ of $K(r_e)$ is nonzero. The minor $\Delta_I$ of any representative of the kernel will be nonzero by \zcref{lem:duality-plucker-coordinates}, and there is exactly one VRC with coefficients $(r_e)$. 
    This establishes the first point. 

    For the second point, recall from the proof of \zcref{lem:bdry-meas-well-defined} that for $(x_e)$ in the domain of $\bM^G$, $\bM^G(x_e)$ is obtained by transforming $\sigma K(x_e)$ into
      \[\begin{bmatrix}
        \mathrm{Id} & X\\
        0 & A
    \end{bmatrix}\]
    using invertible row operations and then taking the orthogonal complement of $\rowsp A$. In this instance, $x_e= \sigma_e r_e$, so $\sigma K(x_e) = K(r_e)$. By assumption, $(x_e)$ is in the domain of $\bM^G$ and $\sigma K(x_e)$ can indeed be put in the desired form. As described above, the unique VRC with coefficients $(r_e)$ is given by the kernel of $K(r_e)$, which is unaffected by invertible row operations. So we have 
     \[\begin{bmatrix}
        \mathrm{Id} & X\\
        0 & A
    \end{bmatrix}\begin{bmatrix}
        - u_{w_1}^T - \\
        \vdots\\
          - u_{n}^T - \\
    \end{bmatrix}=0.\]
    This equation implies that $A$ is orthogonal to $\partial[\bu, (r_e)]$, so $\partial [\bu, (r_e)]$ is contained in $\bM^G(\sigma_er_e)$. Since both subspaces have dimension $k$, they are equal. 
\end{proof}

\begin{cor}\label{cor:fiber-vs-VRC-set}
    Fix $G$ a type $(k,n)$ graph and fix any immersion of $G$. For $Y \in \Gr_{k,n}$, the fiber
    \[(\bM^G)^{-1}(Y)\]
    is in bijection with the set of VRCs
    \[\{[\bu, (r_e)]: \partial[\bu, (r_e)] = Y\} \cap \{[\bu, (r_e)]: K(r_e) \text{ is $W$-full rank}\}.\]
\end{cor}

\begin{proof}
    Call the set of VRCs in the corollary statement $X$. Fix a choice of immersion signs $\sigma$. \zcref{lem:bdry-meas-vs-VRC} shows that if $\bM^G(x_e) = Y$, then there is a unique VRC with $(r_e)= (\sigma_e x_e)$ and its boundary is $Y$. Since $(x_e)$ is in the domain of boundary measurement, $\sigma K(x_e) = K(r_e)$ is $W$-full rank. So $Y \mapsto [\bu, (\sigma_e x_e)]$ is a well-defined injection from the fiber to $X$. That it is also surjective follows from the second bullet point of \zcref{lem:bdry-meas-vs-VRC}.
\end{proof}

\zcref{cor:fiber-vs-VRC-set} says that generically, finding a VRC with boundary $V$ is the same as finding a pre-image of $V$ under the boundary measurement map. In the following lemma, we use this connection to VRCs to gain additional combinatorial information about reduced graphs. 

\begin{lemma}\label{lem:reduced-black-vt-bounded-degree}
    Suppose $G$ is a reduced type $(k,n)$ graph. Then the black vertices of $G$ have degree at most $k+1$.
\end{lemma}

\begin{proof}
We will argue by contradiction. We assume that $G$ is reduced and has a black vertex of degree 
$p\ge k+2$. We will then use \zcref{cor:fiber-vs-VRC-set} to show that generic fibers of $\bM^G$ are infinite, a contradiction.

We first need to choose an appropriate gauge-fix (cf. \zcref{rem:torus-coordinates}). Say the edges incident to $b$ are $e_1, \dots, e_p$ and have white endpoints $w_1, \dots, w_p$. By \zcref{lem:reduced-conditions}, $G/\bdG$ is connected since $G$ is reduced. Without loss of generality, there is a path from $b$ to a boundary vertex which uses $e_1$ and no other edges incident to $b$. We claim that there is a spanning tree $T$ of $G/\bdG$ which contains $e_1$ and none of $e_2, \dots, e_p$. Suppose for the sake of contradiction that this does not exist. This means that deleting the edges $e_2, \dots, e_p$ splits $G$ into $C_1 \sqcup C_2$, where $C_1$ contains $b$ and all of the boundary vertices, and all vertices of $C_2$ are internal. Since $G$ is reduced, by \zcref{lem:reduced-conditions}, there is a matching $M$ of $G$ which contains $e_1$ and hence does not contain $e_2, \dots ,e_p$. So $M$ restricts to a perfect matching of $C_2$. This implies that $C_2$ has the same number of black and white vertices. Now, $C_2$ also contains the white endpoint $w$ of some edge $e_i$. There is a matching $M'$ of $G$ which contains $e_i$. We have that $M'$ restricts to a perfect matching of $C_2 \setminus w$. This implies that $C_2$ has one more white vertex than black vertex, a contradiction.

We henceforth gauge-fix the edge weights $\{x_e: e \in T\}$ to 1.

Consider a point $X$ in the image of the boundary measurement map. By \zcref{lem:bdry-meas-vs-VRC}, $X$ is the boundary of a VRC $[\bu, (r_e)]$ on $G$. The chosen gauge fix identifies a distinguished representative $(\bu, (r_e))$. Let $u_i$ be the vector on $w_i$ and $r_i$ be the coefficient on $e_i$. We have $u_1 + \sum_{i=2}^p r_i u_i =0$. Since the $u_i$ are in $\bC^k$ and we have at least $k+2$ of them, we have another linear dependence
\[d_2 u_2 + \dots + d_{k+2}u_{k+2} =0.\]
For any $\epsilon \in \bC^*\setminus \{-r_i/d_i: i \in [2, k+2], d_i \neq 0\}$, changing the coefficients of the VRC by $r_i \mapsto r_i(\epsilon):= \epsilon d_i + r_i$ for $i \in [2, k+2]$ will produce a VRC $(\bu, (r_e(\epsilon)))$ with the same vectors and different coefficients. Because of the gauge fix, we have that $[\bu, (r_e(\epsilon))] \neq [\bu, (r_e)]$. We also note that the minors of $K(r_e(\epsilon))$ are polynomials in $\epsilon$. Since $K(r_e)$ has a non-vanishing minor using columns $W$, for $\epsilon$ sufficiently close to 0, this minor will also be non-vanishing for $K(r_e(\epsilon))$. This shows that for $X$, the fiber 
\[(\bM^G)^{-1}(X)=\{[\bu, (r_e)]: \partial[\bu, (r_e)] = X\} \cap \{[\bu, (r_e)]: K(r_e) \text{ is $W$-full rank}\}\]
is infinite. Thus $G$ is not reduced, as all points in the image of $\bM^G$ have infinite fiber.
\end{proof}

\section{Lobes} \label{sec:lobes}

In this section, we discuss a decomposition of the graph $G$ into subgraphs $G_i$ called \emph{lobes}. We show the boundary measurement map on the full graph $G$ can be identified as the product of boundary measurement maps on each lobe, followed by a ``gluing'' map $\Phi$ which records how the lobes combine to make $G$. This will be useful in our analysis of top cell $k=2$ graphs in \zcref{sec:k=2}, where the lobes are particularly simple.

\begin{definition}
    Let $\tilde{G_1}, \ldots, \tilde{G_\ell}$ be the connected components of $G \setminus \partial G.$ For $1 \leq i \leq \ell,$ let $G_i$ be the induced subgraph of $G$ on the vertices of $\tilde{G_i}$ together with the incident boundary vertices. We call the graph $G_i$ a \emph{lobe} of $G$, and denote its edges by $E_i$ and its internal vertices by $V_i$. See \zcref{fig:lobe-example} for an example. 
\end{definition}

\begin{figure}[h!]
    \centering
    \begin{tikzpicture}
        \node at (0,0) {\onshellgraph{ {(0,0)/1,(0,-1)/2,(0,-2)/3,(-3,0)/5,(-3,-2)/4,(2,-1)/6} }{ {(0,0)/W1,(0,-1)/W2,(0,-2)/W3,(-3,0)/W5,(-3,-2)/W4,(2,-1)/W6,(-2,-1)/A} }{ {(-2,0)/B1,(-1,-1)/B2,(-2,-2)/B3,(1,-1)/B4} }{ {W1/B1,W5/B1,A/B1,W1/B2,W2/B2,W3/B2,A/B2,W3/B3,A/B3,W4/B3,W1/B4,W2/B4,W3/B4,W6/B4} }{1}};
        \node at (0,-1.5) {$G$};
        \draw (3,1.5) -- (3,-1.5);
        \node at (5,0) {\onshellgraph{ {(0,0)/1,(0,-1)/2,(0,-2)/3,(-3,0)/5,(-3,-2)/4} }{ {(0,0)/W1,(0,-1)/W2,(0,-2)/W3,(-3,0)/W5,(-3,-2)/W4,(-2,-1)/A} }{ {(-2,0)/B1,(-1,-1)/B2,(-2,-2)/B3} }{ {W1/B1,W5/B1,A/B1,W1/B2,W2/B2,W3/B2,A/B2,W3/B3,A/B3,W4/B3} }{1}};
        \node at (5,-1.5) {$G_1$};
        \node at (9,0) {\onshellgraph{ {(0,0)/1,(0,-1)/2,(0,-2)/3,(2,-1)/6} }{ {(0,0)/W1,(0,-1)/W2,(0,-2)/W3,(2,-1)/W6} }{ {(1,-1)/B4} }{ {W1/B4,W2/B4,W3/B4,W6/B4} }{1}};
        \node at (9,-1.5) {$G_2$};
    \end{tikzpicture}
    \caption{A type (3,6) top-cell graph $G$, alongside its lobes $G_1$ and $G_2$.}\label{fig:lobe-example}
\end{figure}
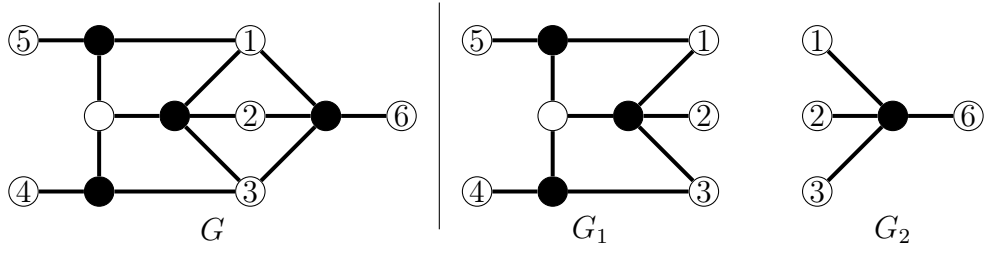

The decomposition of $G$ into lobes gives decompositions $E = E_1 \cup \dots \cup E_\ell$ and $V = V_1 \cup \dots \cup V_\ell.$ If we order the edges and internal vertices accordingly, the $|E| \times |V|$ incidence matrix $Q$  is block-diagonal. Thus,
\begin{align*}
    (\bC^*)^{E - V} = \prod_{i = 1}^\ell (\bC^*)^{E_i - V_i}.
\end{align*}

The following lemma shows that $\mathbb{M}^G$ factors as the product of boundary measurement maps from each lobe, followed by a gluing map.

\begin{lemma}\label{lem:gluing exists}
    Let $G$ be a type $(k,n)$ graph with lobes $\{G_i\}_{i \in [\ell]}$ that are type $k_i$. Suppose that if $j$ is a boundary vertex of $G_i$, there is exactly one edge of $G_i$ incident to $j$. Choose an immersion of $G$ such that if $e,e'$ cross and belong to different lobes, they are both boundary edges. Then, there exists a gluing map $\Phi^G : \prod_{i =1}^\ell (\mathrm{im}\,\bM^{G_i}) \dashrightarrow \Gr_{k,n}$ such that
    \[
        \bM^G = \Phi^G \circ \prod_{i=1}^\ell \bM^{G_i},
    \]
    where each $\bM^{G_i}$ has as its domain the factor $(\bC^*)^{E_i - V_i}$ of $(\bC^*)^{E-V}$.
\end{lemma}

\begin{proof}
Let $N_i:=\bdG_i$ be the set of boundary vertices of lobe $G_i$ and set $n_i :=|N_i|$.

    The restriction of a perfect matching $M$ to a lobe $G_i$ is itself a perfect matching $M_i := M \cap G_i$ on the lobe. If lobe $G_i$ is a type $(k_i,n_i)$ graph, then $M_i$ has boundary
    \[
        \partial M_i := \{j \in \partial G_i : j \text{ is in some } e \in M_i\}
    \]
    which is of size $(n_i-k_i)$. For a fixed boundary $\partial M$, there can be multiple possibilities for $(\partial M_1, \dots , \partial M_\ell)$. For $J \subset \partial G$ of size $n-k,$ let
    \[
        L_J := \left\{ (J_1,\dots,J_\ell) : J_i \in \binom{N_i}{n_i-k_i},~\bigcup_{i = 1}^\ell J_i = J, ~J_i \cap J_j = \emptyset \right\}.
    \] 
    In words, $L_J$ consists of all the possibilities for $(\partial M_1, \dots , \partial M_\ell)$ if $M$ is a matching with boundary $J$.
    
    We can re-express the sum from \zcref{def:bdry-meas} as a sum over the lobes' perfect matchings:
    \[
        \begin{aligned}
            \bM^G_I(x_e) &= (-1)^{\tot(I)} \sum_{M : \partial M = I^c} (-1)^{\cross(M)} \prod_{e \in M} x_e \\
            &= (-1)^{\tot(I)} \sum_{(J_1,\dots,J_\ell) \in L_{I^c}} \sum_{M_1 : \partial M_1 = J_1} \dots \sum_{M_\ell : \partial M_\ell = J_\ell} (-1)^{\cross(\bigcup_{i = 1}^\ell M_i)} \prod_{e \in \bigcup_{i =1}^\ell M_i} x_e.
        \end{aligned}
    \]
    In the last line, the $i$th sum is a sum over perfect matchings $M_i$ of $G_i$.
    
    By our assumption on the immersion, the crossings of a perfect matching either occur within a lobe or are crossings between boundary edges of distinct lobes. This allows one to write $\cross(M) = \sum_{i = 1}^\ell \cross(M_i) + \cross(\partial M_1,\dots,\partial M_\ell)$, where the last term records only crossings between boundary edges of different lobes. Note that since any boundary vertex $j$ is incident to at most one edge in any lobe $G_i$, which boundary edges of different lobes cross is completely determined by the boundary sets $\partial M_1,\dots,\partial M_\ell$.

We have
    \[
        \begin{aligned}
            \bM^G_I(x_e) &= (-1)^{\tot(I)} \sum_{(J_1,\dots,J_\ell) \in L_{I^c}} \sum_{M_1 : \partial M_1 = J_1} \dots \sum_{M_\ell : \partial M_\ell = J_\ell} (-1)^{\cross(\bigcup_{i = 1}^\ell M_i)} \prod_{e \in \bigcup_{i =1}^\ell M_i} x_e \\
            &= (-1)^{\tot(I)} \sum_{(J_1,\dots,J_\ell) \in L_{I^c}} (-1)^{\cross(J_1,\dots,J_\ell)} \prod_{i = 1}^\ell \bigg( \sum_{M_i : \partial M_i = J_i} (-1)^{\cross(M_i)} \prod_{e \in M_i} x_e\bigg) \\
            &= (-1)^{\tot(I)} \sum_{(J_1,\dots,J_\ell) \in L_{I^c}} (-1)^{\cross(J_1,\dots,J_\ell)} \prod_{i = 1}^\ell \bigg( (-1)^{\sigma(J_i)} \bM^{G_i}_{N_i \setminus J_i}(x_e)\bigg)
        \end{aligned}
    \]
    where $(-1)^{\sigma(J_i)}$ is the appropriate sign.
    
    We define the map $\Phi^G$ to send a point in the product $\prod_{i =1}^\ell (\im \bM^{G_i})$ with Pl\"ucker coordinates
    \[\bigr( (p^1_{I_1})_{I_1 \in \binom{N_1}{k_1}}, \dots, (p^{\ell}_{I_\ell})_{I_\ell \in \binom{N_\ell}{k_\ell}}\bigr) \]
    to the point with Pl\"ucker coordinates
    \[
        p_{I}= (-1)^{\tot(I)} \sum_{(J_1,\dots,J_\ell) \in L_{I^c}} (-1)^{\cross(J_1,\dots,J_\ell)} \prod_{i = 1}^\ell \bigg( (-1)^{\sigma(J_i)} p^i_{N_i \setminus J_i}\bigg).
    \]

    By construction, this map satisfies $\bM^G= \Phi^G \circ \prod_{i=1}^\ell \bM^{G_i}$.
\end{proof}

\begin{rem}
    The assumptions of \zcref{lem:gluing exists} do not meaningfully restrict what graphs can be considered. Let $G$ be a type $(k,n)$ graph with lobes $\{G_i\}_{i \in [\ell]}$. One can apply expand moves at the boundary vertices of $G$ to obtain a graph $G'$ such that the boundary vertices of each lobe $G'_i$ are incident to at most one edge in $G'_i$. There is also a disk immersion of $G'$ that satisfies the conditions of the lemma, namely each pair of crossing edges belonging to different lobes are both boundary edges. One way to accomplish this is to place the interior vertices of the lobes in disjoint disks, and arrange the disks so they do not intersect the boundary edges of other lobes. See \zcref{fig:lobe-immersion} for an example.
\end{rem}

Unlike expand/contract moves, square moves can affect the decomposition of a graph into lobes, as illustrated in \zcref{fig:lobe-move}.

\begin{figure}[h!]
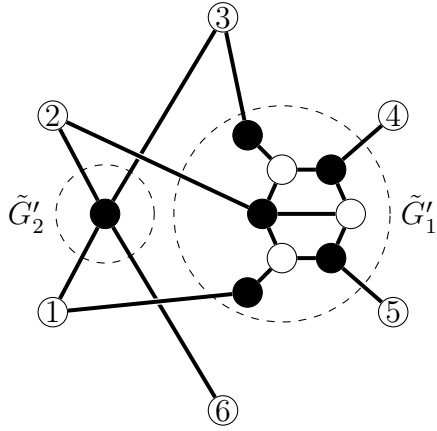

    \[
        \onshellgraphwithcode{
            {(210:2)/1,(150:2)/2,(90:2)/3,(30:2)/4,(-30:2)/5,(-90:2)/6}
         }{
            {(210:2)/W1,(150:2)/W2,(90:2)/W3,(30:2)/W4,(-30:2)/W5,(-90:2)/W6,(0.6,-0.45)/WW1,(0.6,0.45)/WW3,(1.3,0)/WW}
        }{
            {(-1.2,0)/B4,(0.25,-0.8)/BB1,(0.25,0.8)/BB3,(0.4,0)/B1,(1.1,-0.45)/BB5,(1.1,0.45)/BB4}
        }{
            {W1/B4/black,W2/B4/black,W3/B4/black,W6/B4/black,W1/BB1/black,W3/BB3/black,BB1/WW1/black,BB3/WW3/black,W2/B1/black,WW1/B1/black,WW3/B1/black,B1/WW/black,WW3/BB4/black,W4/BB4/black,WW/BB4/black,WW1/BB5/black,W5/BB5/black,WW/BB5/black}
        }{1.3}{
            \draw[dashed] (-1.2,0) circle (0.5);
            \node() at (-2,0) {$\tilde G_2'$};
            \draw[dashed] (0.6,0) circle (1.1);
            \node() at (2,0) {$\tilde G_1'$};
        }
    \]
    \caption{The graph $G'$ obtained from the graph $G$ of \zcref{fig:lobe-example} using expand moves on the boundary vertices $1$ and $3$. The immersion satisfies the conditions of \zcref{lem:gluing exists}.}\label{fig:lobe-immersion}
\end{figure}

\begin{figure}[h!]
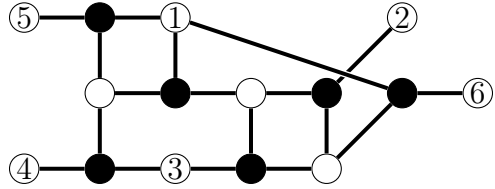

    \[
        \onshellgraph{
            {(0,0)/1,(-2,0)/5,(-2,-2)/4,(0,-2)/3,(3,0)/2,(4,-1)/6}
        }{
            {(0,0)/W1,(-2,0)/W5,(-2,-2)/W4,(0,-2)/W3,(3,0)/W2,(4,-1)/W6,(-1,-1)/A,(1,-1)/B,(2,-2)/C}
        }{
            {(-1,0)/B1,(-1,-2)/B2,(0,-1)/B3,(1,-2)/B4,(2,-1)/B5,(3,-1)/B6}
        }{
            {W1/B1,W5/B1,A/B1,W4/B2,A/B2,W3/B2,W1/B3,A/B3,B/B3,W3/B4,B/B4,C/B4,B/B5,C/B5,W2/B5,W1/B6,C/B6,W6/B6}
        }{1}
    \]
    \caption{A type (3,6) top-cell graph $G''$, related to $G$ from \zcref{fig:lobe-example} by a sequence of expand/contract and square moves. This graph has only one lobe.}\label{fig:lobe-move}
\end{figure}

In the situation when $G$ and all of its lobes are top-cell graphs, the gluing map $\Phi^G$ has a very easy inverse map, given by coordinate projection. Before proving this, we need a combinatorial lemma.

\begin{lemma}\label{lem:top-lobes-k}
    Let $G$ be type $(k,n)$. Suppose $G$ and each of its lobes $G_i$ are all top-cell graphs. Then, $G_i$ are type $(k,|\partial G_i|)$ graphs.
\end{lemma}
\begin{proof}
    Since $G$ is a top-cell graph, for any $k$-element subset $I$, there exists a perfect matching $M$ which avoids $I$. The restriction of a perfect matching to $G_i$ avoids $k_i$ boundary vertices and also avoids at least the $k$ boundary vertices $I$, so we must have $k_i \geq k$.

    Counting black and white vertices in each lobe separately, one identifies the constraint
    \[
        k(n - k) = k(|B| - |W|) = k \sum_{i} (|B_i| - |W_i|) = k\sum_{i} (|\partial G_i| - k_i).
    \]
     We also have that $|E|-|V| = \sum_{i} |E_i| -|V_i|$. Since $G$ and all its lobes are reduced top-cell graphs, $|E|-|V|$ is the dimension of $\Gr_{k,n}$ and $|E_i| -|V_i|$ is the dimension of $\Gr_{k_i, \bdG_i}$. So we have
    \[
        k(n-k) = \sum_i k_i (|\partial G_i| - k_i).
    \]
    Using the preceding expressions for $k(n-k)$,
    \[k\sum_{i} (|\partial G_i| - k_i) =\sum_i k_i (|\partial G_i| - k_i) \]
    which together with $k_i \ge k$ implies $k_i = k$ for every lobe.
\end{proof}

If $G_i$ is a lobe of $G$, we write $(\bu,\bfR)_{G_i}$ for the restriction of the VRC to $G_i$.

\begin{lemma}\label{lem:gluing-inverse}
    Suppose immersed top-cell graph $G$ has lobes $G_1, \dots, G_\ell$, each of which is itself a top-cell graph, and satisfies the assumptions of \zcref{lem:gluing exists}. The gluing map $\Phi^G$ is the inverse of the coordinate projection map
    \[
        \begin{aligned}
            \pi^G : \Gr_{k,n} &\dashrightarrow \prod_{i \in [\ell]} \Gr_{k,\partial G_i}, \\
            (p_I)_{I \in \binom{[n]}{k}} &\mapsto \prod_{i \in [\ell]} (p_I)_{I \in \binom{\partial G_i}{k}}.
        \end{aligned}
    \]
\end{lemma}

\begin{proof}
    Suppose $(\sigma_e r_e)$ is in the domain of $\bM^G$, so $\bM(\sigma_e r_e)$ is the boundary of a VRC $[\bu, (r_e)]$ by \zcref{lem:bdry-meas-vs-VRC}. Since $k_i = k$ for each lobe by \zcref{lem:top-lobes-k}, we can apply \zcref{lem:bdry-meas-vs-VRC} to each boundary measurement map appearing in $\Phi^G \circ \prod_{i \in [\ell]}\bM^{G_i} = \bM^G$. Then, we obtain  
    \[
        \Phi^G( \partial[\bu,(r_e)]_{G_1} , \dots , \partial[\bu,(r_e)]_{G_\ell} )=\partial[\bu,(r_e)].
    \]
    Applying $\pi^G$ on both sides we obtain
    \[
        (\pi^G \circ \Phi^G) ( \partial[\bu,(r_e)]_{G_1} , \dots , \partial[\bu,(r_e)]_{G_\ell} ) = (\partial[\bu,(r_e)]_{G_1} , \dots , \partial[\bu,(r_e)]_{G_\ell}),
    \]
    determining that $\pi^G \circ \Phi^G = \id$. This implies that $\Phi^G$ is injective with one-sided inverse $\pi^G$. By \zcref{lem:gluing exists}, the image of $\bM^G$ is contained in the image of $\Phi^G.$ Since $G$ is top-cell, $\bM^G$ is a dominant rational map and thus $\Phi^G$ is also dominant. So $\pi^G$ is in fact the unique two-sided inverse of $\Phi^G$.

\end{proof}

\begin{cor}\label{cor:degree-factors}
    Under the assumptions of \zcref{lem:gluing-inverse}, the degree of the gluing map is $\deg \Phi^G = 1$, and the degree of the boundary measurement map $\bM^G$ is the product of the component-wise degrees
    \[
        \deg \bM^G = \prod_{i \in [\ell]} \deg \bM^{G_i}.
    \]
\end{cor}
In future sections, we will consider those graphs $G$ for which $\deg \bM^G = 1$. Since $\deg \bM^{G_i} = 1$ for a reduced plabic graph $G_i$, we learn from the corollary that the boundary measurement map of a gluing of plabic graphs is also degree 1.

\section{Perfect matching toric varieties} \label{sec:toric}

In this section, we define a toric variety $X_G$ for each type $(k,n)$ graph $G$. The nonnegative part of a toric variety is a positive geometry \cite{ahbl2017}, and so has a distinguished rational differential top-form called the \emph{canonical form}, which in this case is the dlog form of the dense torus. The boundary measurement map $\bM^G: (\bC^*)^{E-V} \dashrightarrow \Gr_{k,n}$ factors through $X_G$; we will soon use this to define the on-shell form of $G$, by pushing forward the canonical form of $X_G$.

We will define $X_G$ using the \emph{perfect matching polytope} of $G$, so we begin with the combinatorics of this polytope. We then define the toric variety and discuss its properties.

\subsection{Perfect matching polytope}

\begin{definition}
    For $G$ a type $(k,n)$ graph and $M \in \match(G)$, let $\chi(M) \in \bR^E$ denote the indicator vector of $M$. The \emph{perfect matching polytope} of $G$ is
    \[P(G):=\conv(\chi(M): M \in \match(G)) \subset \bR^E.\]
\end{definition}

The perfect matching polytope $P(G)$ (and corresponding toric variety $X_G$) were first defined in \cite{psw2009} for plabic $G$, and the following results on the inequality description and dimension of $P(G)$ were proved in ibid. for plabic $G$. We provide the proofs for general $G$ for the convenience of the reader.

\begin{proposition}\label{prop:matching-polytope-inequality}
    The polytope $P(G)$ has inequality description
    \[P(G)= \{x \in \bR^E: x_e \ge 0 \text{ for all }e \in E, \sum_{e \ni v} x_e =1 \text{ for all }v \in V, \sum_{e \ni v} x_e \le 1 \text{ for all }v \in \bdG\}.\]
\end{proposition}

\begin{proof}
    For the inequality description, temporarily denote by $Q(G)$ the polytope on the right-hand side. It is clear that $P(G) \subset Q(G)$. We also have that $Q(G) \cap \bZ^E = P(G) \cap \bZ^E$, since the condition on coordinate sums forces $x_e \le 1$ for all $e$ and a 0/1 vector satisfying the inequalities for $Q(G)$ is the indicator vector of a perfect matching. So it suffices to show that all vertices of $Q(G)$ are integer vectors. 

    Consider a non-integer point $x \in Q(G)$. We will produce two other points $x', x'' \in Q(G)$ such that $x = \frac{1}{2} (x' + x'')$, showing that $x$ is not a vertex of $Q(G)$. 
    
    Let $H$ be the subgraph of $G$ consisting of edges where $0 < x_e < 1$. Any internal vertex $v$ which is in an edge of $H$ must have degree at least 2 in $H$, since $\sum_{e \ni v} x_e =1$. This implies that $H$ contains either a cycle or a path from a boundary vertex to a boundary vertex. Note that we may assume that in the latter case, the boundary vertices $v$ involved satisfy $\sum_{e \ni v} x_e <1$, since otherwise $v$ is also degree at least 2 in $H$ and the path can be extended to a cycle.
    
   Let $e_1, \dots, e_m$ be the edges of either the cycle or the path, in order. Let $x'$ be obtained from $x$ by adding a very small positive $\epsilon$ to $x_{e_i}$ for the odd edges and subtracting $\epsilon$ from $x_{e_i}$ for the even edges. Define $x''$ similarly, swapping the role of even and odd. For sufficiently small $\epsilon$, $x', x'' \in Q(G)$ and clearly $x = \frac{1}{2} (x' + x'')$.

\end{proof}

\begin{rem} \label{rem:bdry-vt-deg1-easier-polytope}
If we assume that all boundary vertices of $G$ are degree 1, then the inequalities
\[\sum_{e \ni v} x_e \le 1 \text{ for all }v \in \bdG\]
are implied by the other inequalities of $P(G)$.
\end{rem}

Before providing the dimension of $P(G)$, it will be useful to record that (EC) moves essentially do not affect the perfect matching polytope.

\begin{lemma}\label{lem:polytope-ec-moves}
    Suppose $G$ and $G'$ are related by an (EC) move. Then $P(G)$ and $P(G')$ are linearly isomorphic by a map which is bijective on lattice points in their respective affine hulls. 
\end{lemma}
\begin{proof}
    Suppose $G$ is as on the left in \zcref{fig:moves} and $G'$ is as on the right, with edge sets $E, E'$ respectively. We consider $E'$ as a subset of $E$. Say $e_\ell$ and $e_r$ are the edges of $G$ which are not edges of $G'$, and the other edges adjacent to $e_r$ are $e_1, \dots, e_p$. All points in $P(G)$ satisfy
    \[x_{e_r}= 1-(x_{e_1} + \dots + x_{e_p}) \qquad \text{and} \qquad x_{e_{\ell}}= 1-x_{e_r} = x_{e_1} + \dots + x_{e_p}.\]
    So $P(G)$ is contained in the codimension 2 affine subspace $S$ cut out by the above two equations. Consider the projection map $\pi$ from $S \subset \bR^E$ to $\bR^{E'}$ which forgets the coordinates $x_{e_\ell}$ and $x_{e_r}$. Clearly $\pi$ is an isomorphism, and takes $\bZ^E \cap S$ bijectively to $\bZ^{E'}$. Further, $\pi$ sends the vertices of $P(G)$ to the vertices of $P(G')$, so the polytopes are linearly isomorphic as claimed.
\end{proof}

Recall that $G/\bdG$ is the graph obtained from $G$ by identifying all boundary vertices.

\begin{lemma}\label{lem:polytopedim}
     If every edge of $G$ is used in some perfect matching, then $P(G)$ has dimension $|E|-|V|+(c-1)$, where $c$ is the number of connected components of $G/\bdG$. 
\end{lemma}

\begin{proof}
    We first assume that all boundary vertices of $G$ are degree 1. Under this assumption, 
    \[P(G)= \{x \in \bR^E: x_e \ge 0 \text{ for all }e \in E, \sum_{e \ni v} x_e =1 \text{ for all }v \in V\}\]
    and $P(G)$ intersects the interior of the positive orthant in $\bR^E$. So its affine hull is the same as the affine linear space which cuts $P(G)$ out of the positive orthant. This is the linear space 
    \[\{x \in \bR^E: \sum_{e \ni v} x_e =1 \text{ for all }v \in V\}\]
    which is a parallel translate of the vector space satisfying $\sum_{e \ni v} x_e =0$ for all $v \in V$. This vector space is precisely the relative homology group $H_1(G, \bdG)$, which has dimension $|E|-|V| + (c-1)$. 

    Now, if any boundary vertices are not degree 1, we may apply an (EC) move to make them degree 1. By \zcref{lem:polytope-ec-moves}, the resulting graph has isomorphic polytope, of the dimension claimed in the lemma statement.
\end{proof}

Recall that an integral polytope $P \subset \bR^d$ 
has the \emph{integer decomposition property} if for every positive $k \in \bZ$, every integer point in $k P$ is a sum of $k$ integer points in $P$. The integer decomposition property in particular implies the polytope is normal. The following result is a consequence of \cite[Lemma 2.1]{OH01}.

\begin{proposition}\label{prop:matching-polytope-normal}
    The polytope $P(G)$ has the integer decomposition property and in particular is normal. 
\end{proposition}
\begin{proof}
 Note that $P(G)$ is a 0/1 polytope, and is described by the inequalities
 \[0 \le x_e \leq 1\text{ for $ e \in E$ and } b_v \leq \sum_{e \ni v} x_e \leq b_v + \epsilon_v \text{ for }v \in V \cup \bdG\]
where $b_v = 1, \epsilon_v =0$ if $v \in V$ and $b_v=0, \epsilon_v=1$ if $v \in \bdG$.

If all boundary vertices of $G$ are degree 1, then by \zcref{rem:bdry-vt-deg1-easier-polytope}, we can omit the inequalities above involving $v \in \bdG$. So in this case, we have $\epsilon_v=0$ in all inequalities, and the integer decomposition property (IDP) is implied by \cite[Lemma 2.1]{OH01}. If any boundary vertices are higher degree, we may perform (EC) moves to obtain $G'$ with all boundary vertices of degree 1. So $P(G')$ has IDP. By \zcref{lem:polytope-ec-moves}, $P(G')$ is linearly isomorphic to $P(G)$ via an isomorphism which is bijective on integer points in their affine hulls. Such an isomorphism preserves IDP.
\end{proof}

\subsection{Perfect matching varieties}

\begin{definition}\label{def:perfect-matching-variety}
For $G$ a type $(k,n)$ graph, let $\match(G)$ be the set of perfect matchings of $G$. We define a monomial map 
\begin{align*}
    (\bC^*)^E &\to \bP^{|\match(G)|-1}\\
    (r_e)_{e \in E}&\mapsto \left(\prod_{e \in M} r_e\right)_{M \in \match(G)}.
\end{align*}

The closure of the image of this map is a toric variety $X_G$, which we call the \emph{perfect matching variety}.
\end{definition}

Using monomial maps to define toric varieties can sometimes result in undesired behavior. The next lemma shows that is not the case here.

\begin{lemma}
    The variety $X_G \subset \bP^{|\match(G)|-1}$ is normal. Its moment polytope is $P(G)$.
\end{lemma}

\begin{proof}
    The variety $X_G$ is defined using a monomial map. The exponent vectors for the monomial map are given by the vertices of $P(G)$, which, since $P(G)$ is a $0/1$ polytope, is the same as $P(G) \cap \bZ^E$. Since $P(G)$ has IDP, it is normal and in particular very ample. Thus, the embedded variety $X_G$ is the same as the abstract toric variety defined using the normal fan of $P(G)$, so is normal. The statement about moment polytopes is a standard fact about toric varieties defined using monomial maps.
\end{proof}

Note that the monomial map in \zcref{def:perfect-matching-variety} factors through $(\bC^*)^{E-V}$ so rational functions on $X_G$ pull back to gauge-invariant functions on $(\bC^*)^E$. In fact, the monomial map often descends to an isomorphism.

\begin{proposition}\label{prop:bdry-meas-torus-vs-dense-torus}
    Suppose every edge of $G$ is used in at least one perfect matching. Then the monomial map of \zcref{def:perfect-matching-variety} induces an isomorphism of algebraic tori between $(\bC^*)^{E-V}$ and the dense open torus of $X_G$.
\end{proposition}
\begin{proof}
    Let $T$ be the dense torus of $X_G$. Choose any perfect matching $M_0$, and let $L \subset \bZ^E$ be the lattice generated by the differences $\{\chi(M)-\chi(M_0): M \in \match(G)\}$. By \cite[Proposition 2.1.6]{CoxLittleSchenck}, $L$ is the character lattice of $T$. Its rank is $\dim P(G)$.

    Let $L' \subset \bZ^E$ be the character lattice of $(\bC^*)^{E-V}$. This is the kernel of $Q^T$, the $|V| \times |E|$ incidence matrix. Its rank is $|E|-|V|+ (c-1)$ by \zcref{cor:basis-for-torus-mod-gauge}, where $c$ is the number of connected components of $G/\bdG$.

    The monomial map induces a surjective homomorphism $s:(\bC^*)^{E-V} \to T$, which is dual to the inclusion $L \subset L'$. We would like to show that in fact $L=L'$.

    By \zcref{prop:matching-polytope-normal}, $P(G)$ has IDP. This implies that $L$ is saturated in $\bZ^E$, meaning that $L= \mathrm{span}_{\bR}(L) \cap \bZ^E$. So $L'/L$ is torsion-free. By \zcref{lem:polytopedim}, $\dim P(G)=|E|-|V|+ (c-1)$. So $L$ and $L'$ have the same rank. This implies $L'/L$ is trivial, as desired. So $s$ is dual to the identity map of character lattices, and is an isomorphism.
\end{proof}

Recall the moves (EC), (PV), (S) on weighted immersed graphs from \zcref{fig:moves}. The (PV) moves have no effect on the perfect matching variety $X_G$, since the variety does not depend on the immersion. \zcref{lem:polytope-ec-moves} implies that if $G, G'$ are related by an (EC) move, then $X_G, X_{G'}$ are isomorphic. It is an immediate corollary of \zcref{lem:boundary-measurement-moves} that the moves also determine birational maps between perfect matching varieties which preserve boundary measurements.

\begin{cor} \label{cor:toric-var-moves}
    Let $G, G'$ be immersed type $(k,n)$ graphs which are related by a (PV), (EC), or (S) move, and let $\phi$ be the corresponding rational map on edge weights. Then $\phi$ determines a birational map $\phi: X_G \dashrightarrow X_{G'}$ such that $\bM^{G'} \circ \phi = \bM^{G}.$
\end{cor}

\subsection{Perfect matching variety as GIT quotient}\label{ssec:GIT}

Finally, we show that the toric variety $X_G$ may be viewed as a GIT quotient of $\bC^E$ by the action of the torus $(\bC^*)^V.$ Let $Q$ be the $|E| \times |V|$ incidence matrix of internal vertices and consider the lattice $N$ defined by
\[\bZ^V \xrightarrow{Q} \bZ^E \to N \to 0.\]
Note that $Q$ is a submatrix of the incidence matrix of $G,$ given by deleting the columns corresponding to $\partial G$. Since the graph $G$ is bipartite, its incidence matrix is totally unimodular, and so is $Q.$ Thus the cokernel is free. 

Indeed, $N$ is the co-character lattice of the torus $(\bC^*)^{E - V}.$ Let $M$ be the dual lattice to $N,$ which can be obtained via the dual sequence
\begin{equation}\label{eq:duallattice}
    0 \to M \to \Hom(\bZ^E, \bZ) \xrightarrow{Q^T} \Hom(\bZ^V, \bZ).
\end{equation}

To get an embedded toric variety via the toric GIT construction, one requires a character $a = (a_1, \ldots, a_{|E|})$ of the torus $(\bC^*)^E.$ The corresponding polytope is then 
\begin{equation}\label{eq:git polytope}
P_a = \left\{ m \in M_{\mathbb{R}} \ \middle|\ \langle m, \nu_e \rangle \geq -a_e,\ e \in E \right\} \subseteq M_{\mathbb{R}},\end{equation}
where $v_e$ are the images in $N$ of the standard basis of $\bZ^E.$ By \cite[Theorem 14.2.13]{CoxLittleSchenck}, the quotient $\bC^E /\!/_a (\bC^*)^V$ is isomorphic to the toric variety of $P_a.$  

\begin{lemma}\label{lem:matching-variety-is-git-quotient}
    Suppose $G$ is a type $(k,n)$ graph whose boundary vertices have degree one. Let $a$ be the incidence vector of any perfect matching. Under the inclusion induced by \eqref{eq:duallattice} the polytope $P_a$ maps to a translate of the polytope $P(G).$ In particular, $X_{P(G)}$ is a toric GIT quotient of $\bC^E.$ 
\end{lemma}
\begin{proof}
    By \eqref{eq:git polytope}, the inclusion $\iota: \ker Q^T \hookrightarrow \bZ^E$ of \eqref{eq:duallattice} sends the polytope $P_a$ to
\[\{x \in \bR^E: x_e \geq -a_e, \, \, Q^Tx = 0\}.\]
Now, suppose that $a$ is a vector such that $Q^Ta = 1$. For instance, we may take $a$ to be any incidence
vector of any perfect matching. Then under the translation $y := x + a$, we get
\[\iota({P_a}) + a = \{y \in \bR^E: y_e \geq 0, \, \, Q^Ty = 1\}.\]
This agrees with the inequality description of the perfect matching polytope given in \zcref{prop:matching-polytope-inequality}. By Remark \ref{rem:bdry-vt-deg1-easier-polytope} the inequalities $\sum_{e \ni v} x_e \leq 1$ for $v \in \partial G$ are implied by the other inequalities. Note that translating the polytope does not change the toric variety.
\end{proof}

The condition of degree one vertices in Lemma \ref{lem:matching-variety-is-git-quotient} is not critical. Indeed, any graph $G$ may be converted to this form via (EC) moves, which by \zcref{lem:polytope-ec-moves} do not change the toric variety.

Because the matrix $Q$ is totally unimodular, for a general choice of $a,$ the toric variety of $P_a$ is smooth. Our choice of character $a$ above is not general, and indeed the toric variety $X_G$ is typically singular. 
\begin{eg}\label{eg:twosix-toric}
    Consider the type $(2,6)$ graph $G$ of \zcref{eg:twosix}. In any perfect matching, each black vertex is matched to exactly one of its neighbors. Exactly thirty of the $3^4 = 81$ choices of one incident edge per black vertex give a matching. We have $X_G \subset \bP^{29}$, and it is dimension 8. 
    
    For a generic choice of $a,$ the GIT quotient $X_{P_a} \subset \bP^{80}$ is isomorphic to $(\bP^2)^4$. The boundary measurement map factors as 
    \[(\bC^*)^E \to (\bP^2)^4 \dashrightarrow X_G \dashrightarrow \Gr_{2,6}.\]
    The middle map is coordinate projection from the Segre embedding $(\bP^2)^4 \hookrightarrow \bP^{80}$ to $\bP^{29}.$ 

    This $G$ decomposes into lobes $G_1, \dots,G_4 $, which are all ``tripods'' consisting of one black vertex attached to three boundary vertices. The toric varieties $X_{G_i}$ are all isomorphic to $\bP^2$. So this example also shows that the toric variety $X_G$ does not factor as the product of toric varieties $X_{G_i}$ for the lobes of $G$. 
    
\end{eg}

\section{On-shell forms as pushforwards}\label{sec:form}

In this section, we define the on-shell form of a reduced graph $G.$ On-shell forms were originally defined in the physics literature \cite{Arkani-Hamed:2012zlh} in a rather different way. On-shell forms for plabic graphs in particular were defined in a similar way as we do here in \cite[Section 14]{Lam-grassmann-polytope}.

Throughout this section, we assume $G$ is reduced. Recall that by \zcref{lem:reduced-conditions}, for reduced graphs we have that $G/\bdG$ is connected and every edge is in some perfect matching. So $(\bC^*)^{E - V}$ has dimension $|E|-|V|$ (\zcref{lem:polytopedim}), and is isomorphic to the dense open torus in the perfect matching variety $X_G$ (\zcref{prop:bdry-meas-torus-vs-dense-torus}).

Any torus $(\bC^*)^m$ has a natural differential top-form, which may be written in coordinates $(t_1, \, \ldots, \, t_m)$ as 
\begin{equation}\label{eq:dlog-form}
    \eta = \frac{dt_1}{t_1}\wedge \ldots \wedge \frac{dt_m}{t_m}.
\end{equation}
For the particular torus $(\bC^*)^{E - V}$, we may obtain explicit coordinates via \zcref{rem:torus-coordinates}, yielding a differential top-form $\eta_G$ on the torus. The form $\eta_G$ extends uniquely to a rational top-form on the toric variety $X_G$ with at most simple poles along each boundary component in $X_G \setminus (\bC^*)^{E-V}$ \cite[Section 5.3]{BrownDupont}. In particular, the form $\eta_G$ is the canonical form of the semialgebraic set $(X_G)^{\ge0} \subset X_G(\bR)$, viewed as a \emph{positive geometry} \cite[Section 5.6]{ahbl2017}.

\begin{rem}
    All differential forms we consider in this paper are defined only up to overall choice of sign. In particular, if one changes the ordering of the torus coordinates, then the sign of the form as written in \eqref{eq:dlog-form} will change.
\end{rem}

The standard form on $(\bC^*)^{E - V}$ may be written conveniently in coordinates on the larger torus $(\bC^*)^E$ using spanning trees of $G / \partial G.$ As in previous sections, we let $Q$ be the $|E| \times |V|$ incidence matrix of the internal vertices, where we have chosen an ordering of the edges. For a subset $I \subset E$ of the edges, we let $\tot(I)$ be the total index with respect to the chosen ordering of $|E|.$ We use $Q_I$ to denote the submatrix using the rows indexed by $I.$ As $Q$ is totally unimodular, determinants of submatrices are $\pm 1$.
    
    \begin{proposition}\label{prop:torus-form-spanning-trees}
    Let $G$ be a reduced graph. Let $\mathcal{T}$ be the set of spanning trees of $G / \partial G.$
    In coordinates $x_e,$ we have 
    \begin{equation}\label{eq:invariant-torus-form}
        \eta_G = \sum_{T  \in \mathcal{T}} \mathrm{sgn}(T)\bigwedge_{e \notin T}\frac{dx_e}{x_e},
    \end{equation}
    where $\mathrm{sgn}(T):= (-1)^{\mathrm{tot}(T)} \det Q_{E \setminus T}.$
    \end{proposition}
    \begin{proof}
    Recall from \eqref{eq:cokernel} that $(\mathbb C^*)^{E-V}$ is the cokernel of $Q$; by  \zcref{prop:bdry-meas-torus-vs-dense-torus} it is the dense torus of $X_G$. Let $Y$ be any matrix representative of the cokernel of $Q.$ Then by  \cite[Section 8.2]{CoxLittleSchenck}, we have that
        \begin{align*}
            \eta_G = \sum_{\substack{I \subset E \\ |I| = |E|-|V|}} \det Y_I \bigwedge_{e \in I}\frac{dx_e}{x_e}.
        \end{align*}
        Since we assume that $G$ is reduced, by \zcref{lem:reduced-conditions}, $G / \bdG$ is connected, and thus $Q$ has full rank. By Cramer's rule, any minor of $Y$ is, up to a sign, the same as the minor of $Q$ with the complementary index. Concretely, this yields the formula
        \begin{align*}
            \eta_G = \sum_{\substack{I \subset E \\ |I| = |E|-|V|}}(-1)^{\mathrm{tot}(I)} \det Q_{E \setminus I} \bigwedge_{e \in E \setminus I}\frac{dx_e}{x_e}.
        \end{align*}
        
        The matrix $Q$ is the incidence matrix of $G / \partial G$ with one row removed. Thus the maximal minor $Q_{E \setminus I}$ is nonzero if and only if the edges ${E \setminus I}$ form a spanning tree of $G / \partial G$ \cite[Proposition 5.4]{biggs}. 
    \end{proof}

\begin{rem}
    We note that \eqref{eq:invariant-torus-form} becomes particularly simple if one gauge-fixes as in \zcref{rem:torus-coordinates}; that is, choose a spanning tree $T$ and fix all of those edge weights to 1. In this case, most terms of \eqref{eq:invariant-torus-form} become zero, and the only term that contributes is 
    \begin{equation}\label{eq:eta-as-single-wedge}\eta_G=\mathrm{sgn}(T)\bigwedge_{e \notin T}\frac{dx_e}{x_e}.\end{equation}
    This expression for $\eta_G$, as the dlog form for the free edge-weights, is common in the literature.
\end{rem}

\begin{eg}
    \newcommand{\dln}[1]{\frac{dx_{#1}}{x_{#1}}}
    Let $G$ be the graph from \zcref{fig:24-contraction-for-spanning}. The contraction $G / \partial G$ of this graph has 9 spanning trees, labeled by $\mathcal T = \{0,1,2\} \times \{3,4,5\}$ in correspondence with the edge weights $x_0,\dots,x_5$. The corresponding form $\eta_G$ on $(\mathbb C^*)^{6}$ is
    \[
        \begin{aligned}
            \eta_G &= -\dln{1} \wedge \dln{2} \wedge \dln{4} \wedge \dln{5} + \dln{1} \wedge \dln{2} \wedge \dln{3} \wedge \dln{5} + \dots - \dln{0} \wedge \dln{1} \wedge \dln{3} \wedge \dln{4} \\
            &= \frac{x_2 dx_0 \wedge dx_1 - x_1 dx_0 \wedge dx_2 + x_0 dx_1 \wedge dx_2}{x_0 x_1 x_2} \wedge \frac{-x_5 dx_3 \wedge dx_4 + x_4 dx_3 \wedge dx_5 - x_3 dx_4 \wedge dx_5}{x_3 x_4 x_5}.
        \end{aligned}
    \]
    On $(\mathbb C^*)^{E-V} = (\mathbb C^*)^3 / \mathbb C^* \times (\mathbb C^*)^3 / \mathbb C^* \subset \mathbb P^2 \times \mathbb P^2$, the form is simply
    \[
        \eta_G = \frac{dt_1}{t_1} \wedge \frac{dt_2}{t_2} \wedge \frac{dt_3}{t_3} \wedge \frac{dt_4}{t_4},
    \]
    with coordinates $t_1 = x_1 / x_0$, $t_2 = x_2 / x_0$, $t_3 = x_4 / x_3$, $t_4 = x_5/x_3$.
\end{eg}

\begin{figure}[h!]
    \centering
    \[
        \vcenter{\hbox{\onshellgraphwithcode{
         {(0,0)/1,(1.71,0.71)/2,(3.42,0)/3,(1.71,-0.71)/4}
        }{
         {(0,0)/W1,(1.71,0.71)/W2,(3.42,0)/W3,(1.71,-0.71)/W4}
        }{
         {(1,0)/B1,(2.42,0)/B2}
        }{
         {W1/B1/black,W2/B1/black,W4/B1/black,W3/B2/black,W4/B2/black,W2/B2/black}
        }{
         1.3
        }{
        
        }}}
        \hspace{+1em}
        \vcenter{\hbox{\begin{tikzpicture}
            \node[draw,circle](W) at (0,0) {};
            \node[draw,fill,circle](B1) at (-2,0) {};
            \node[draw,fill,circle](B2) at (2,0) {};
            \draw[line width=1.5pt] (B1) -- (W);
            \draw[line width=1.5pt] (B1) to[out=45,in=135] (W);
            \draw[line width=1.5pt] (B1) to[out=-45,in=-135] (W);
            \draw[line width=1.5pt] (W) -- (B2);
            \draw[line width=1.5pt] (W) to[out=45,in=135] (B2);
            \draw[line width=1.5pt] (W) to[out=-45,in=-135] (B2);
            \node() at (0,0.5) {$\partial G$};
            \node() at (0,1) {};
            \node() at (0,-1) {};
            \node() at (-1,0.7) {$x_0$};
            \node() at (-1,0.2) {$x_1$};
            \node() at (-1,-0.3) {$x_2$};
            \node() at (1,0.7) {$x_3$};
            \node() at (1,0.2) {$x_4$};
            \node() at (1,-0.3) {$x_5$};
        \end{tikzpicture}}}
    \]
    \caption{A type (2,4) top-cell graph $G$ and its contraction $G/\partial G$.}\label{fig:24-contraction-for-spanning}
\end{figure}
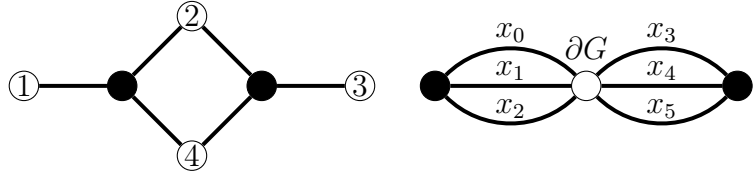

\begin{definition}
Let $G$ be a reduced graph of type $(k,n)$ where $\deg \bM^G =1$. Let $Y$ be the closure of $\im \bM^G$. The \emph{on-shell form} $\omega_G$ of $G$ is the pushforward of the form $\eta_G$ to $Y$ along the boundary measurement map, that is 
\[\omega_G := \bM^G_* \eta_G.\]
\end{definition}

\begin{rem}\label{rem:formulas-for-on-shell-form}
    In practice, one writes $\eta_G$ as in \eqref{eq:eta-as-single-wedge}. Using this, one writes formulas for $\omega_G$ by inverting the boundary measurement map, that is, expressing $x_e$ as a rational function in Pl\"ucker coordinates. From there, one may choose an affine chart on $\Gr_{k,n}$, fix coordinates on it, and then manipulate the dlog formula to obtain formulas for $\omega_G$ of the form $\omega_G=f ~dX$, where $f$ is rational and $dX$ is the standard top form on the chosen affine chart, in the fixed coordinates.
\end{rem}

\begin{eg}
    Let $G$ be the weighted graph in \zcref{fig:25-for-formula}; note that the edges weighted 1 are a spanning tree of $G / \partial G$. The form on the torus is
    \[
        \eta_G = \frac{dy_1}{y_1} \wedge \frac{dy_2}{y_2} \wedge \frac{dy_3}{y_3} \wedge \frac{dy_4}{y_4} \wedge \frac{dy_5}{y_5} \wedge \frac{dy_6}{y_6}.
    \]
    Since the boundary measurements include
    \[
        \bM^G_{13} = 1 , ~\bM^G_{23} = -y_1,~\bM^G_{34} = -y_2,~\bM^G_{35} = y_3,~\bM^G_{12} = -y_4,~\bM^G_{14} = -y_6,~\bM^G_{15}=y_5,
    \]
    one deduces that $\omega_G := \bM^G_* \eta_G$ is
    \[
        \omega_G = \frac{dp_{23}}{p_{23}} \wedge \frac{dp_{34}}{p_{34}} \wedge \frac{dp_{35}}{p_{35}} \wedge \frac{dp_{12}}{p_{12}} \wedge \frac{dp_{14}}{p_{14}} \wedge \frac{dp_{15}}{p_{15}},
    \]
    on the chart where $p_{13} = 1$. One can write $\omega_G$ on an arbitrary affine chart as
    \[
        \omega_G = \frac{p_{13}}{p_{12}p_{14}p_{15}p_{23}p_{34}p_{35}}\,dX
    \]
    where as in \zcref{rem:formulas-for-on-shell-form}, $dX$ is the standard top form on the chosen affine chart.
\end{eg}

\begin{figure}
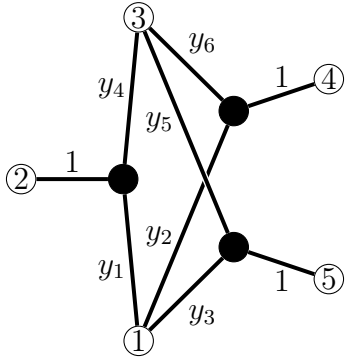

    \centering
    \[
        \onshellgraphwithcode{
         {(252:1.5)/1,(180:1.5)/2,(108:1.5)/3,(36:1.5)/4,(-36:1.5)/5}
        }{
         {(252:1.5)/W1,(180:1.5)/W2,(108:1.5)/W3,(36:1.5)/W4,(-36:1.5)/W5}
        }{
         {(-0.6,0)/B1,(0.375,0.6)/B2,(0.375,-0.6)/B3}
        }{
         {W1/B1/black,W3/B1/black,W2/B1/black,W1/B2/black,W3/B2/black,W4/B2/black,W1/B3/black,W3/B3/black,W5/B3/black}
        }{1.5}{
            \node() at (-0.7,-0.8) {$y_1$};
            \node() at (-0.28,-0.5) {$y_2$};
            \node() at (0.1,-1.2) {$y_3$};
            \node() at (-0.7,0.8) {$y_4$};
            \node() at (-0.28,0.5) {$y_5$};
            \node() at (0.1,1.2) {$y_6$};
            \node() at (-1.05,0.15) {$1$};
            \node() at (0.8,-0.9) {$1$};
            \node() at (0.8,0.9) {$1$};
        }
    \]
    \caption{A type (2,5) top-cell graph.}\label{fig:25-for-formula}
\end{figure}

On-shell forms are special among differential top-forms on $\Gr_{k,n}.$ In particular, they are invariant under the natural torus action on the Grassmannian.
\begin{lemma}\label{lem:onshell torus invariant}
    Let $G$ be a reduced graph of type $(k,n)$ where $\deg \bM^G = 1.$ Let $\star$ be the action of the torus $(\bC^*)^n$ on $\Gr_{k,n}$ as in \eqref{eq:torus action grkn}. Then the on-shell form $\omega_G$ is invariant under the induced action of $(\bC^*)^n$ on the cotangent bundle $T^*\Gr_{k,n}.$
\end{lemma}
\begin{proof}
Since $G$ is reduced, $G / \partial G$ is connected by Lemma \ref{lem:reduced-conditions}. Recall that the incidence matrix of $G / \partial G$ has the same image in $\bC^E$ as the matrix $Q,$ so in particular both contain the vector $1^E.$ By \zcref{lem:bdry-meas-torus-eq}, we have 
    \[t \star [\bM^G_{1 2 \dots k} (x): \dots : \bM^G_{n-k + 1  \dots n}(x) ] = \bM^G(t^{-1} \cdot x),\]
    where the factor of $t_1 \ldots t_n$ disappears. Indeed, since $1^E$ is in the image of $Q,$ multiplying each edge $x_e$ by the same factor acts trivially on the toric variety $X_G$. Then 
    \begin{align*}
         t \star \omega_G = t \star (\bM^G)_*\eta_G = (\bM^G)_*(t^{-1} \cdot \eta_G).
    \end{align*}
    But $t^{-1}$ acts trivially on $\eta_G,$ because each term in $\eta_G$ is a wedge of dlog forms and therefore torus-invariant.
\end{proof}

\begin{rem} As mentioned in \zcref{rem:bdry-meas-vs-Postnikov}, even if $G$ is a reduced graph, $\bM^G$ may have degree larger than 1. In this case, one can still define the pushforward of $\eta_G$ along $\bM^G$ as in \cite[(4.1)]{ahbl2017}, and thus can still define the on-shell form $\omega_G:=\bM^G_* \eta_G$. Writing expressions for $\omega_G$ using Pl\"ucker coordinates becomes more complicated. The on-shell forms of top-cell graphs of type $(3,6)$ were computed\footnote{The forms appearing in \cite{Bourjaily:2016mnp} for $\bM^G$ of degree 2 are not rational, so are not quite the same as $\bM^G_* \eta_G$ in the higher degree case.} in \cite{Bourjaily:2016mnp}.
\end{rem}

While we used the reduced immersed graph $G$ to define the on-shell form, the form depends only on $G$ up to moves. In particular, $\omega_G$ does not depend on the immersion.

\begin{lemma}\label{lem:moves-preserve-form}
    Suppose $G, G'$ are related by a (PV), (EC) or (S) move. Then $\omega_G$ and $\omega_{G'}$ are equal, up to sign.
\end{lemma}
\begin{proof}
    By \zcref{cor:toric-var-moves}, in this case there is a birational map $\phi: X_{G} \dashrightarrow X_{G'}$ such that $\bM^G = \bM^{G'} \circ \phi$. So it suffices to show that $\phi_* \eta_G$ is equal to $\eta_{G'}$, up to sign.

    This is straightforward for (EC) moves, and for (S) moves is checked in \cite[Proof of Theorem 13.5]{Lam-grassmann-polytope}. For (PV) moves, we note that $\phi$ simply changes $x_q$ to $-x_q$. Since $\frac{d(-x_q)}{-x_q} = \frac{dx_q}{x_q}$, \zcref{prop:torus-form-spanning-trees} implies that $\phi_* \eta_G$ is equal to $\eta_{G'}$ in this case as well.
\end{proof}

We note that there are additional moves which preserve the form $\omega_G$, called \emph{sphere moves}. See \cite{Cachazo:2018wvl,CastravetTevelev,Lisitsyn:2025prd} for more details.

Finally, recall from \zcref{sec:lobes} the decomposition of a graph $G$ into lobes $G_1, \dots, G_\ell$. Under the assumptions of \zcref{lem:gluing exists}, one can write $\bM^G$ in terms of the $\bM^{G_i},$ followed by a gluing map $\Phi^G$. This implies that the form $\omega_G$ can also be understood in terms of the forms $\omega_{G_i}$.

\begin{lemma}
    Let $G$ be a reduced graph with lobes $G_1$,...,$G_\ell$ and with $\deg \bM^G =1$. Suppose $G$ satisfies the assumptions of \zcref{lem:gluing exists} and let $\Phi^G$ be the associated gluing map. Then,
    \[
        \omega_G = (\Phi^G)_* \bigwedge_{i \in [\ell]} \omega_{G_i}.
    \]
\end{lemma}
\begin{proof}
    Starting with $\eta_G$ on the torus $(\mathbb C^*)^{E-V}$, one makes use of $(\bC^*)^{E-V} = \prod_{i \in [\ell]} (\bC^*)^{E_i - V_i}$ to write
    \[
        \eta_G = \bigwedge_{i \in [\ell]} \eta_{G_i}.
    \]
    Then, making use of $\bM^G = \Phi^G \circ \prod_{i \in [\ell]} \bM^{G_i}$,
    \[
        \omega_G = (\bM^G)_* \eta_G = (\Phi^G)_* \bigwedge_{i \in [\ell]} (\bM^{G_i})_* \eta_{G_i} = (\Phi^G)_* \bigwedge_{i \in [\ell]} \omega_{G_i}.
    \]
\end{proof}

\section{A determinantal formula for $k=2$ top-cell forms}\label{sec:k=2}

In this section, we focus on top-cell graphs of type $(2,n)$. There is a combinatorial characterization of these graphs, which we carry out in \zcref{subsec:topcell k=2}; see \zcref{prop:top cell k=2} in particular. 
 
 \zcref{subsec:det formula} then proves a formula in Pl\"ucker coordinates for the on-shell forms of such graphs. This formula originally appeared in the physics literature in \cite{Arkani-Hamed:2014bca}, and we give a rigorous mathematical proof in \zcref{thm:pushforward-hypertree}. 

\subsection{Classification of top-cell $(2,n)$ graphs}\label{subsec:topcell k=2}
It is convenient to give a combinatorial characterization of which type $(2,n)$ graphs are top-cell graphs. The following notion of \emph{convexity} is the key criterion. 

\begin{definition}\label{def:convexity}
    Let $G$ be a type $(2,n)$ graph whose internal vertices are all black and trivalent. Such a graph has lobes $G_1, \, \ldots, \, G_{\ell}$, each of which is a ``tripod'', i.e. a black vertex connected to three boundary vertices. We say that $G$ is \emph{convex} if for any non-empty $S \subset [\ell]$ we have \[\left| \bigcup_{j \in S} \partial G_j \right| - 2 \geq |S|.\] 
\end{definition}

Convexity ensures that no Pl\"ucker coordinates vanish on the image of the boundary measurement map.

\begin{lemma}\label{lem:convexity-iff-matchings}
    Let $G$ be a type $(2,n)$ graph with $n-2$ internal vertices, all of which are black and trivalent. Then $G$ is convex if and only if for all $i,j \in [n]$, $G$ has a matching with boundary $[n] \setminus \{i,j\}$.
\end{lemma}
\begin{proof}
    This is an application of Hall's Marriage Theorem to the graphs $G \setminus \{i,j\}$.
\end{proof}

The following gives a combinatorial classification of type $(2,n)$ top-cell graphs.

\begin{proposition}
    \label{prop:top cell k=2}
    Let $G$ be a graph with $n$ boundary vertices and no bivalent internal vertices\footnote{These can be removed by using contraction moves.}. Then $G$ is a type $(2,n)$ top-cell graph if and only if it has exactly $n-2$ internal vertices, all of which are black and trivalent, and $G$ is convex. 

    Moreover, for type $(2,n)$ top-cell graphs, $\bM^G$ has degree 1, that is, is generically injective.
\end{proposition}
\begin{proof}

    ($\implies$:) For $G$ to be a type $(2,n)$ top-cell graph, one must have $\dim (\bC^*)^{E-V} = |E| - (|W| + |B|) = 2(n-2)$ and $k = n+|W|-|B| = 2$, so $|E|-2|B| = n-2$. \zcref{lem:reduced-black-vt-bounded-degree} implies the black vertices of $G$ are degree at most 3. Degree 2 is excluded by assumption. Degree 1 is impossible: by the ``leafless'' assumption on reduced graphs, such a black vertex $b$ must be adjacent to a boundary vertex $i$. But then $i$ is in the boundary of every perfect matching, meaning that any Pl\"ucker coordinate containing $i$ vanishes on $\im \bM^G$. But $\im \bM^G$ is dense in $\Gr_{2,n}$, so this is impossible. Thus all black vertices are trivalent. Since the graph is bipartite, the number of edges is $|E| = 3|B|$, so $|B| = n-2$ and $|W| = 0$.

    For $G$ to be top-cell, we must also have that no Pl\"ucker coordinates vanish identically on the image of $\bM^G$. This is equivalent to the condition that for all $i, j \in [n]$, $G$ has a perfect matching with boundary $[n] \setminus \{i,j\}$, which is in turn equivalent to convexity by \zcref{lem:convexity-iff-matchings}.

($\impliedby$): Suppose $G$ has $n-2$ internal vertices, all of which are black and trivalent, and is convex. Each lobe of $G$ is a tripod. We have $|E|-|V|= 3(n-2) -(n-2) = 2(n-2)$, so $G$ is a top-cell type $(2,n)$ graph as long as $\bM^G$ has dense image in $\Gr_{2,n}$. 

Let $X \in \Gr_{2,n}$ have all Pl\"ucker coordinates nonvanishing. We claim that as long as $X$ is sufficiently generic, $X$ is in the image of boundary measurement. We will show this by finding a VRC with boundary $X$ and applying \zcref{cor:fiber-vs-VRC-set}. 

Choose a particular matrix representative of $X$, which we denote in the same way. For the VRC, since there are no internal white vertices, we need only to specify the coefficients $r_e$. Suppose black vertex $b$ is incident to boundary vertices $i<j<m$. Then, using notation $bi$ for the edge $\{b, i\}$, we define
\[r_{bi} = p_{j m}(X), \quad r_{bj} = -p_{i m}(X), \quad r_{bm} = p_{i j}(X).\]
These numbers are all nonzero by assumption. This choice of coefficients gives the unique VRC with boundary $X$ modulo gauge, since 3 generic vectors in $\bC^2$ satisfy a unique linear relation up to simultaneous rescaling. If the weighted adjacency matrix $K(r_e)$ is full-rank, then $X$ is in the image of $\bM^G$ by \zcref{cor:fiber-vs-VRC-set}. The matrix $K(r_e)$ is full-rank for at least one $X \in \Gr_{2,n}$ (any $X$ in the image of $\bM^G$) and so is full-rank on an open subset of $\Gr_{2,n}$. Thus, restricting $X$ further to lie in this open subset, $X$ is in the image of $\bM^G$.

For the `Moreover' part: in the paragraph above, generic $X$ have fiber of size exactly 1, since there is a unique VRC with boundary $X$.
    
\end{proof}

In light of \zcref{prop:top cell k=2}, in the subsequent section we will restrict our study to graphs whose internal vertices are all black and trivalent.

\subsection{Proving a determinantal formula for $\omega_G$ for top-cell $(2,n)$ graphs}\label{subsec:det formula}

We fix the following notation for the rest of this section. 

\begin{notation}\label{not:A matrix}
    The symbol $G$ will always denote a top-cell (2,$n$) graph. We denote by $\mathcal{S}$ the set $\{\partial G_1, \, \ldots, \, \partial G_r\}$ of boundaries of lobes. Each element of $\mathcal{S}$ is an ordered triplet $(abc)$ with ordering inherited from $\partial G$. We define the $(n-2) \times n$ matrix $A$ with rows indexed by $\mathcal{S}$, columns indexed by $[n]$ and nonzero entries
\begin{equation}\label{eq:hypertree matrix}
    A_{(abc, a)} = p_{bc}, \quad A_{(abc, b)} = -p_{ac}, \quad A_{(abc, c)} = p_{ab}.
\end{equation}
    The matrix $A$ depends on $G$, but we omit this from the notation. We may evaluate this matrix on a $2 \times n$ matrix $X$ or a point $X \in \Gr_{2,n}$, in which case we write $A(X)$.
\end{notation}

\begin{rem}\label{rem:kastelyn A}
    The matrix $A(X)$ is exactly the weighted adjacency matrix $K(r_e)$ from the proof of \zcref{prop:top cell k=2}. That is, for generic $X \in \Gr_{2,n}$, the entries of $A(X)$ give the coefficients of the unique VRC on $G$ with boundary $X$.
\end{rem}

  Our main theorem is \zcref{thm:pushforward-hypertree}, which gives a determinantal formula for the on-shell form $\omega_G$ using the triplets $\mathcal{S}$ and the matrix $A$. This result has previously appeared in the physics literature \cite{Arkani-Hamed:2014bca}. We defer the proof to later in the section.

  \begin{theorem}\label{thm:pushforward-hypertree}
    Let $G$ be a top-cell type $(2,n)$ graph. The on-shell form of $G$ is given by 
\begin{equation}\label{eq:pushforward}
        \omega_G = \frac{(\det A')^2}{\displaystyle\prod_{(abc) \in \mathcal{S}}p_{ab}p_{ac}p_{bc}}dX,
    \end{equation}
    where $A'$ is any submatrix of $A$ obtained by deleting a row and the three columns with nonzero entries in that row. 
\end{theorem}

The differential $dX$ in the theorem statement should be interpreted as follows. Consider any affine chart in $\Gr_{2,n}$ of the form $\{p_{ij} \neq 0\}.$ Points in this chart are parameterized by matrices $X$ whose submatrix $X_{ij}$
is the $2 \times 2$ identity matrix, and whose other entries are $x_{rs}$ for $r \in [2], s \in [n] \setminus \{i,j\}$. Then 
\begin{equation}\label{eq:x coordinates}
    dX = \displaystyle\bigwedge_{s \in [n] \setminus \{i, j\}} dx_{1s} \wedge dx_{2s}.
\end{equation}
The other factor in \eqref{eq:pushforward} is a rational function in the variables $x_{rs}.$ Since rational forms are determined on Zariski open sets, this uniquely describes the differential form $ \omega_G.$ 

\begin{eg}\label{eg:twosix-hypertree} 
    We consider several graphs for $n \leq 8$, each of which we identify with the sets of triplets $\{\partial G_i\}_{i \in [n-2]}$, pictured in \zcref{fig:k2-irreducible}.
    \[
        \begin{aligned}
            \{\partial G_i^{I}\}_{i \in [4]} &= \{123,345,156,246\}, \\
            \{\partial G_i^{II}\}_{i \in [5]} &= \{123,456,147,257,367\}, \\
            \{\partial G_i^{III}\}_{i \in [6]} &= \{123,345,567,178,246,268\}.
        \end{aligned}
    \]
    The corresponding matrices are
    \[
        \begin{aligned}
            A^I &= \begin{bmatrix}
                p_{23} & -p_{13} & p_{12} & 0 & 0 & 0 \\
                0 & 0 & p_{45} & -p_{35} & p_{34} & 0 \\
                p_{56} & 0 & 0 & 0 & -p_{16} & p_{15} \\
                0 & p_{46} & 0 & -p_{26} & 0 & p_{24}
            \end{bmatrix} \\
            A^{II} &= \begin{bmatrix}
                p_{23} & -p_{13} & p_{12} & 0 & 0 & 0 & 0 \\
                0 & 0 & 0 & p_{56} & -p_{46} & p_{45} & 0 \\
                p_{47} & 0 & 0 & -p_{17} & 0 & 0 & p_{14} \\
                0 & p_{57} & 0 & 0 & -p_{27} & 0 & p_{25} \\
                0 & 0 & p_{67} & 0 & 0 & -p_{37} & p_{36}
            \end{bmatrix} \\
            A^{III} &= \begin{bmatrix}
                p_{23} & -p_{13} & p_{12} & 0 & 0 & 0 & 0 & 0 \\
                0 & 0 & p_{45} & -p_{35} & p_{34} & 0 & 0 & 0 \\
                0 & 0 & 0 & 0 & p_{67} & -p_{57} & p_{56} & 0 \\
                p_{78} & 0 & 0 & 0 & 0 & 0 & -p_{18} & p_{17} \\
                0 & p_{46} & 0 & -p_{26} & 0 & p_{24} & 0 & 0 \\
                0 & p_{68} & 0 & 0 & 0 & -p_{28} & 0 & p_{26}
            \end{bmatrix}
        \end{aligned}
    \]
    \zcref{thm:pushforward-hypertree} yields
    \[
        \begin{aligned}
            \omega_{G^I} &= \frac{(p_{16}p_{24}p_{35} - p_{15}p_{26}p_{34})^2}{p_{12}p_{13}p_{23}p_{34}p_{35}p_{45}p_{15}p_{16}p_{56}p_{24}p_{26}p_{46}}\,dX, \\
            \omega_{G^{II}} &= \frac{\left(p_{14}p_{27}p_{37}p_{56} + p_{17}p_{27}p_{36}p_{45} - p_{17}p_{25}p_{37}p_{46}\right)^2}{p_{12}p_{13}p_{23}p_{45}p_{46}p_{56}p_{14}p_{17}p_{47}p_{25}p_{27}p_{57}p_{36}p_{37}p_{67}}\,dX, \\
            \omega_{G^{III}} &= \frac{\left(p_{18}p_{24}p_{35}p_{67} - p_{18}p_{26}p_{34}p_{57} + p_{17}p_{28}p_{34}p_{56}\right)^2}{p_{12}p_{13}p_{23}p_{34}p_{35}p_{45}p_{56}p_{57}p_{67}p_{17}p_{18}p_{78}p_{24}p_{46}p_{28}p_{68}}\,dX.
        \end{aligned}
    \]
    The rational factors in the above forms have total degree $-6$,$-7$,$-8$ in the Pl\"ucker variables and thus yield well-defined rational differential top-forms on $\Gr_{2,6}$, $\Gr_{2,7}$, and $\Gr_{2,8}$ respectively.
\end{eg}

\begin{figure}
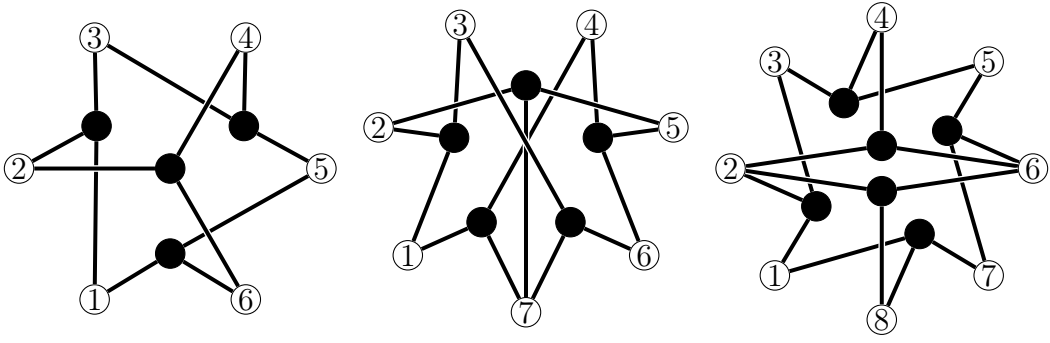

        \[\vcenter{\hbox{\onshellgraph{{(-120:2)/1,(-180:2)/2,(120:2)/3,(60:2)/4,(0:2)/5,(-60:2)/6}}{{(-120:2)/W1,(-180:2)/W2,(120:2)/W3,(60:2)/W4,(0:2)/W5,(-60:2)/W6}}{{(150:1.125)/B1,(30:1.125)/B2,(-90:1.125)/B3,(0,0)/B4}}{{W1/B1, W2/B1, W3/B1, W3/B2, W4/B2, W5/B2, W1/B3, W5/B3, W6/B3, B4/W2, B4/W4, B4/W6}}{1}}}
        \vcenter{\hbox{\onshellgraph{ {(-141.43:2)/1,(-192.86:2)/2,(-244.26:2)/3,(-295.71:2)/4,(-347.14:2)/5,(-398.57:2)/6,(-90:2)/7} }{ {(-141.43:2)/W1,(-192.86:2)/W2,(-244.26:2)/W3,(-295.71:2)/W4,(-347.14:2)/W5,(-398.57:2)/W6,(-90:2)/W7} }{ {(162:1)/B1,(18:1)/B2,(234:1)/B3,(90:1)/B4,(-54:1)/B5} }{ {W1/B1,W2/B1,W3/B1,W4/B2,W5/B2,W6/B2,W1/B3,W4/B3,W7/B3,W2/B4,W5/B4,W7/B4,W3/B5,W6/B5,W7/B5} }{1}}}
        \vcenter{\hbox{\onshellgraph{ {(225:2)/1,(180:2)/2,(135:2)/3,(90:2)/4,(45:2)/5,(0:2)/6,(-45:2)/7,(-90:2)/8} }{ {(225:2)/W1,(180:2)/W2,(135:2)/W3,(90:2)/W4,(45:2)/W5,(0:2)/W6,(-45:2)/W7,(-90:2)/W8} }{ {(210:1)/B1,(120:1)/B2,(30:1)/B3,(-60:1)/B4,(90:0.3)/B5,(-90:0.3)/B6} }{ {W1/B1,W2/B1,W3/B1,W3/B2,W4/B2,W5/B2,W5/B3,W6/B3,W7/B3,W7/B4,W8/B4,W1/B4,W2/B5,W4/B5,W6/B5,W2/B6,W6/B6,W8/B6} }{1}}}
        \]
        \caption{Some type $(2,n)$ graphs for $n=6, 7, 8$.}\label{fig:k2-irreducible}
\end{figure}

\begin{rem}The formula \eqref{eq:pushforward} was first written in \cite{Arkani-Hamed:2014bca} in the context of MHV leading singularities of non-planar $N=4$ super Yang-Mills scattering amplitudes. The physics derivation relies on certain constraints between the entries of the $\Gr_{2,n}$ matrix and kinematical variables $\lambda$, $\widetilde{\lambda}$.

When $G$ is a top-cell $(2,n)$ plabic graph, the numerator of \eqref{eq:pushforward} factors as a product of Pl\"ucker coordinates and there is cancellation between it and the denominator. The on-shell form specializes to the Parke--Taylor form
\[\omega_{\mathrm{PT}} = \frac{1}{p_{12}p_{23} \ldots p_{n-1,n}p_{1n}}dX.\]

The Parke--Taylor form was first written in the physics literature in \cite{Parke:1986gb} using kinematical variables in the context of scattering of gluons in quantum chromodynamics. It was later realized in \cite{Arkani-Hamed:2009ljj} that the same form can be interpreted using minors of the Grassmannian $\Gr_{2,n}$. Furthermore, in \cite{Arkani-Hamed:2012zlh} the authors associated the Parke--Taylor form with the canonical form of the positive Grassmannian $\Gr^{\geq 0}_{2, n}$ and plabic graphs (on-shell diagrams in the physics literature).
\end{rem}

 We will need the following two lemmas about $A$ to prove \zcref{thm:pushforward-hypertree}. 

\begin{lemma}\label{lem:A perp}
    Consider a point $X \in \Gr_{2,n}$. 
    Then $ \mathrm{rowsp} A(X)$ lies inside $X^\perp,$ with equality whenever $A(X)$ has full rank.
\end{lemma}
\begin{proof} 
    Let $X$ denote a $2 \times n$ matrix whose rowspan is the corresponding point in $\Gr_{2,n}.$ We want to show $AX^T = 0.$ Writing $x, y$ for rows of $X,$ this reduces to showing the equality
\[\begin{matrix}
    \begin{vmatrix}
        x_b & x_c \\
        y_b & y_c
    \end{vmatrix} \cdot  \begin{pmatrix}
        x_a \\
        y_a
    \end{pmatrix} & + & \begin{vmatrix}
        x_c & x_a \\
        y_c & y_a
    \end{vmatrix} \cdot \begin{pmatrix}
        x_b \\
        y_b
    \end{pmatrix}& + & \begin{vmatrix}
        x_a & x_b \\
        y_a & y_b
    \end{vmatrix} \cdot \begin{pmatrix}
        x_c \\
        y_c
    \end{pmatrix}= 0.
\end{matrix}\]
This follows by Cramer's rule. 
\end{proof}
\begin{lemma}\label{lem:greatest common factor}
    Consider a point $X \in \Gr_{2,n}.$ Suppose that $A(X)$ has full rank. Let $A'$ be any submatrix of $A(X)$ obtained by deleting a row and the three columns with nonzero entries in that row. Then $\det A'$ is the greatest common factor of the maximal minors of $A(X),$ and the maximal minor given by deleting columns $i$ and $j$ equals $\pm p_{ij} \det A'.$ 
\end{lemma}
\begin{proof}
    Let $X$ be a rank two $2 \times n$ matrix. By \zcref{lem:A perp} and Cramer's rule, we have that
    \[[\det X_{1,2}: \dots : \det X_{n-1,n}] = [\det A(X)_{(1,2)^c}: \dots : \det A(X)_{(n-1,n)^c}],\]
    where equality is in projective space $\bP^{\binom{n}{2}-1}.$ Consider a row indexed by $(abc).$ By construction of $A(X),$ we know that $A(X)_{abc, c} = (\det X_{a,b}) \cdot (\det A'),$ where $A'$ is the $(n-3) \times (n-3)$ submatrix formed by deleting row $(abc)$ and columns $a,b,c.$ Thus $\det A'$ is the overall factor of proportionality, and the greatest common factor of the minors.
\end{proof}

By \zcref{lem:greatest common factor} and \zcref{rem:kastelyn A}, \zcref{thm:pushforward-hypertree} tells us in particular that the on-shell form vanishes whenever the weighted adjacency matrix in the Pl\"ucker coordinates drops rank. 

Recall the coordinate projection map $\pi^G$ from \zcref{lem:gluing-inverse}. We will also need a lemma about the Jacobian of $\pi^G$. 

We first fix some notation. For any subset $S \subset [n],$ let $\bC^S \subset \bC^n$ be the subspace spanned by the standard basis vectors indexed by $S.$ Let $\Gr_{k, S}$ denote the Grassmannian of $k$-dimensional vector spaces in $\bC^S.$ 

To write the Jacobian, we use local coordinates on the source and the target, so that the map $\pi^G$ becomes a birational map between affine spaces of the same dimension. On $\Gr_{2,n}$ we use
the affine chart $\{p_{ij}\neq 0\}$, and for
$S=(abc)\in\mathcal{S}$ we use on $\Gr_{2,S}$ the affine chart $\{p_{ab}\neq 0\}.$ 
\begin{lemma}\label{lem:jacobian}
        Let $G$ be a top-cell type $(2,n)$ graph, and let $A'$ be any submatrix of $A$
obtained by deleting a row and the three columns with nonzero entries in that row. In the coordinates above, the Jacobian 
of the projection map 
        \begin{equation*}
            \pi^G: \Gr_{2,n} \dashrightarrow \prod_{S \in \mathcal{S}} \Gr_{2, S}
        \end{equation*}
        has determinant
$\pm(\det A')^2/\prod_{(abc)\in\mathcal{S}}p_{ab}^3.$
    \end{lemma}
    \begin{proof}
    We first set up some notation. Let $L$ denote a point in $\Gr_{2,n}.$ For $S \in \mathcal{S},$ let $\pi^S: L \to \bC^S$ be the corresponding projection. In terms of Pl\"ucker coordinates, for $S = (abc)$, the map $\pi^S$ induces a map on Grassmannians, which we denote the same way:
    \begin{equation}
        \begin{split}
            \pi^S: \Gr_{2,n} & \dashrightarrow \Gr_{2, S} \\
            (p_{I})_{I \in \binom{n}{2}} & \mapsto [p_{bc} : -p_{ac}: p_{ab}].
        \end{split}
    \end{equation}
    Indeed, a general point $L$ lies outside the base locus of this projection, and the linear spaces $L$ and $\pi^S(L)$ are isomorphic. In particular, $\pi^S(L)$ has dimension two inside of $\bC^S.$ Then $\pi^{G}$ is the direct sum of the maps $\pi^S$ for $S$ in ${\mathcal S}.$

    Similarly, we define
    $\pi_S:\bC^n / L \to \bC^S / \pi^S(L)$ to be the induced projection on the quotient. Let $\pi_G$ be the product of these maps. That is, we define
    \begin{align*}
       \pi_G: \bC^n / L \to W, \qquad \qquad W := \bigoplus_{S \in \mathcal{S}} \bC^S / \pi^S(L)  .
    \end{align*}
    
    We compute the Jacobian by using the natural identification of the tangent space $T_L\Gr_{k,n}$ with $ \Hom(L, \bC^n / L)$, for every point $L \in \Gr_{2,n}$. This identification may be found in standard textbooks, e.g. \cite[Chapter 3]{3264}.

    The pushforward along $\pi^G$ gives a map of tangent bundles
    \begin{equation}\label{eq:tangent map}
        (\pi^G)_*: T\Gr_{2,n} \dashrightarrow \bigoplus_{S \in \mathcal{S}} T\Gr_{2, S}.
    \end{equation}
    
 At a general point $L \in \Gr_{2,n}$, the projection $\pi^S$ has an inverse, and \eqref{eq:tangent map} is given by
    \begin{equation}\label{eq:tangent map explicit}
    \begin{split}
        (\pi^G)_*: \Hom(L, \bC^n/L) & \to \bigoplus_{S \in \mathcal{S}} \Hom(\pi^S(L), \bC^S / \pi^S(L))  \cong \Hom(L, W) \\
        f &  \mapsto \bigoplus_{S \in \mathcal{S}} \pi_S \circ f \circ (\pi^S)^{-1}.
    \end{split}
    \end{equation}
    
Recall that for vector spaces $A, B$ we have the natural identification
     \[\Hom(A, B) = A^\vee \otimes B.\]
     Via the identification $\Hom(L, \bC^n / L)^\vee = \Hom(\bC^n/L, L),$ we view $(\pi^G)_*$ as an element of 
    \[\Hom(\bC^n/L, L) \otimes \Hom(L, W) = (\bC^n / L)^\vee \otimes L \otimes L^\vee \otimes W = \Hom(L^\vee, L^\vee) \otimes \Hom(\bC^n / L, W).\]
    Under this identification we have
    $(\pi^G)_*= \id_{L^\vee} \otimes \pi_G.$ Using the determinantal identity 
    \begin{equation}\label{eq:det identity}
       \det (\id_{L^\vee} \otimes \pi_G) = (\det \pi_G)^{\otimes 2} 
    \end{equation}
    for maps on exterior powers yields the map
    \begin{equation*}
        (\det \pi_G)^2: \Omega^\vee_L{\Gr_{2,n}} \to \displaystyle\bigotimes_{s \in \mathcal{S} }\Omega^\vee_{\pi^S(L)}{(\Gr_{2,S})}
    \end{equation*}
    of one-dimensional vector spaces, for each general point $L.$ Pullback of differential top-forms is thus given at $L$ by $(\det \pi_G^\vee)^2.$ This gives the exponent two in the numerator of~\eqref{eq:pushforward}.
    
    To conclude, it remains to compute $\det \pi_G^\vee$  in our chosen affine chart. For $S = (abc),$ we define the linear functional \[\lambda_S := p_{ab}e_c^\vee - p_{ac}e_b^\vee + p_{bc}e_a^\vee \in (\bC^S)^\vee.\] Then we obtain a concrete description of the dual map $\pi_G^\vee$ as
    \begin{equation*}
        \begin{split}
            \pi_G^\vee: \bigoplus_{s \in \mathcal{S}} (\bC^S/\pi^S(L))^\vee & \to (\bC^n / L)^\vee \\
            \bigoplus_{s \in \mathcal{S}} c_S \lambda_S & \mapsto \sum_{s \in \mathcal{S}} c_S \lambda_S.
        \end{split}
    \end{equation*}

    Now, consider the basis of $(\bC^n / L)^\vee = L^\perp$ given by
    $$\ell_s:=e_s^\vee-x_{1s}\,e_i^\vee-x_{2s}\,e_j^\vee\qquad s\notin\{i,j\},$$
    where $x_{rs}$ are the coordinates on our chart $p_{ij} \neq 0$ as in \eqref{eq:x coordinates}. 
    In terms of this basis we have
    $$\lambda_S=\sum_{s\notin\{i,j\}}\lambda_S(e_s)\,\ell_s=\sum_{s\notin\{i,j\}}A_{S,s}\,\ell_s .$$
    For the chart $\{p_{ab}\neq0\}$ on $\mathrm{Gr}_{2,S}$, the analogous basis of $(\mathbb C^S/\pi^S(L))^\vee$ is the single functional $\lambda_S/p_{ab}$. In these bases, the matrix of $\pi_G^\vee$ is $A$ with columns $i,j$ deleted and each row $(abc)$ divided by $p_{ab}$. Thus $$\det\pi_G^\vee=\det A_{(ij)^c} / \prod_{(abc) \in \mathcal{S}} p_{ab} =\pm p_{ij}\det A' / \prod_{(abc) \in \mathcal{S}} p_{ab},$$ where the final equality follows from Lemma \ref{lem:greatest common factor}. The additional terms $p_{ab}^2$ is because the identity $(\pi^G)_*=\mathrm{id}_{L^\vee}\otimes\pi_G$ uses the rows of $X$ as a basis of both $L$ and $\pi_S(L),$ whereas the chart on $\Gr_{2,S}$ uses coordinates $p_{ac}/p_{ab}, p_{bc}/p_{ab}.$ This contributes an extra factor of $1/p_{ab}^2$ for each $S,$ thus giving the desired formula for the determinant of the Jacobian with respect to the two charts.
\end{proof}

\begin{proof}[Proof of \zcref{thm:pushforward-hypertree}] 
Since the graph $G$ is top cell, by \zcref{prop:top cell k=2} it has lobes $G_1, \ldots, G_{n-2}$, each consisting of a black vertex $b_i$ attached to three white boundary vertices $\partial G_i$. By \zcref{lem:gluing-inverse}, the boundary measurement map factors as 
\[\bM^G = \Phi^G \circ \prod_{i = 1}^{n-2} \bM^{G_i}.\]
Let $\bC^{\partial G_i}$ denote the subspace of $\bC^n$ spanned by basis vectors indexed by $\partial G_i.$ The domain of $\bM^{G_i}$ is $(\bC^*)^{\partial G_i - b_i} = (\bC^*)^{\partial G_i} / \bC^*,$ where $\bC^*$ is included via $t \mapsto (t,t,t).$ Thus $\bM^G$ factors as
\begin{equation*}
    \bM^G: (\bC^*)^{\partial G_1 - b_1} \times \ldots \times (\bC^*)^{\partial G_{n-2} - b_{n-2}} \dashrightarrow \Gr_{2,n}.
\end{equation*}
Any matching in $G_i$ involves exactly one edge, so the boundary measurement map $\bM^{G_i}$ has target $\Gr_{2,3}.$
Since $\pi^G$ is the inverse of $\Phi^G$, the inverse of $\bM^G$ is the composition \[
\begin{tikzcd}[ampersand replacement=\&, baseline=-0.5ex, column sep=large]
\Gr_{2,n}
\arrow[r, dashed, "\pi^G"]
\&
\displaystyle \prod_{i=1}^{n-2}\Gr_{2,\partial G_i}
\arrow[r]
\&
\displaystyle \prod_{i=1}^{n-2}(\bC^*)^{\partial G_i - b_i} .
\end{tikzcd}
\]
where the second map inverts the boundary measurement in each factor. However, up to an automorphism of the torus, we may view the second map as restriction to the subtorus of $\Gr_{2, \partial G_i}$ where Pl\"ucker coordinates are nonzero\footnote{This is indeed a torus since $\Gr_{2, \partial G_i} \simeq \Gr_{2,3}$.}. By Lemma \ref{lem:gluing-inverse}, we have that $\Phi^G$ is birational with inverse $\pi^G$, so $(\Phi^G)_*=(\pi^G)^*$ on rational top forms. Thus 
$$
\omega_G= (\pi^G)^*\Big(\bigwedge_{(abc)\in\mathcal S}d\log{{\frac{p_{ac}}{p_{ab}}}\wedge d\log {\frac{p_{bc}}{p_{ab}}}}\Big)
=\frac{\pm (\det A')^2}{\displaystyle\prod_{(abc) \in \mathcal{S}}p_{ab}^3\,\frac{p_{ac}}{p_{ab}}\frac{p_{bc}}{p_{ab}}}\,dX
=\frac{\pm(\det A')^2}{\displaystyle\prod_{(abc)\in\mathcal S}p_{ab}p_{ac}p_{bc}}\,dX ,
$$
where the second equality follows from Lemma \ref{lem:jacobian}.

    \end{proof}

\section{Relation between $k=2$ top-cell forms and hypertree divisors}\label{subsec:hypertree}


Hypertree divisors are divisors on $\overline{M}_{0,n}$ introduced by Castravet and Tevelev \cite{CastravetTevelev}. When they are irreducible, they correspond to extremal rays in the effective cone of $\overline{M}_{0,n}$ \cite[Theorem 1.5]{CastravetTevelev}. Hypertree divisors are indexed by combinatorial objects called \emph{hypertrees}; the irreducible hypertree divisors are indexed by the subclass of irreducible hypertrees, which also has a straightforward combinatorial characterization.

 In this section we give a procedure for going from an irreducible hypertree to a set of graphs with boundary, which we call \emph{irreducible graphs} (\zcref{def:irreducible}). Our main result in this section is \zcref{prop:graph-to-hypertree}, which gives an alternate formula for on-shell forms of irreducible graphs using the defining equation of the corresponding hypertree divisor. We also show in \zcref{prop:base-locus-hypertree} that for top-cell graphs, the boundary measurement map is an isomorphism precisely on the complement of the corresponding hypertree divisor. Finally, we remark on the reducible case at the end of the section, describing how the on-shell form factors.

\subsection{Background on hypertrees} We begin by recalling the relationship between $\overline{M}_{0,n}$ and the Grassmannian $\Gr_{2,n}.$ The moduli space $\overline{M}_{0,n}$ is a compactification of the open moduli space $M_{0,n}$ of configurations of $n$ distinct points in $\bP^1.$ This open moduli space can be expressed as a torus quotient of the Grassmannian $\Gr_{2,n}$ as follows. 

\begin{definition}
    We define the open subset
    \[\Gr_{2,n}^\circ:= \bigr\{L \ : \ p_{ij}(L) \neq 0 \  \ \forall \ \{i,j\} \in \binom{[n]}{2}\bigr\}\] 
    of $\Gr_{2,n}$ of points whose Pl\"ucker coordinates  are nonzero. We define $M_{0,n}$ to be the quotient of $\Gr_{2,n}^\circ$ by the action of $(\bC^*)^n$ scaling the columns of any matrix representative, or equivalently, by $ (t_1, \ldots, t_n) \cdot p_{ij} = t_it_jp_{ij}.$  
\end{definition}

In particular, any differential form on $\Gr_{2,n}$ which is invariant under the torus action descends to a differential form on $M_{0,n}$ (and therefore to a rational differential form on any compactification). By \zcref{lem:onshell torus invariant}, on-shell forms have this property. One can see this concretely for the forms in \zcref{eg:twosix-hypertree}; the rational function in front of $dX$ has weight $-2$ in each torus coordinate $t_i$, which is then canceled out by $dX,$ which has weight $2.$  

We now introduce irreducible hypertrees, to which we associate divisors in $M_{0,n}$ and $\Gr_{2,n}$.

\begin{definition}
A \emph{hypertree}
\[
\Gamma=\{\Gamma_1,\ldots,\Gamma_d\}
\]
on $[n]$ is a collection of subsets of $[n]$ such that the following conditions are satisfied:
\begin{itemize}
    \item Any subset \(\Gamma_j\) has at least three elements;

    \item Any \(i\in [n]\) is contained in at least two subsets \(\Gamma_j\);
    \item Normalization: \[
    n-2
    =
    \sum_{j=1}^d \bigl(|\Gamma_j|-2\bigr);
    \]

    \item Convexity: \[\left|\bigcup_{j\in S}\Gamma_j\right|-2
    \geq
    \sum_{j\in S}\bigl(|\Gamma_j|-2\bigr)
    \qquad
    \text{for any non-empty } S\subseteq \{1,\ldots,d\}.\]
\end{itemize}

A hypertree \(\Gamma\) is \emph{irreducible} if the final inequality is strict whenever $1<|S|<d.$
\end{definition}
See \zcref{eg:hypertree-graph} for some examples of irreducible hypertrees.

Irreducible hypertrees yield divisors as follows.
Let $\pi: \bP^2 \dashrightarrow \bP^1$ be any surjective linear projection. Fix an irreducible hypertree $\Gamma = (\Gamma_1, \ldots, \Gamma_d)$. Consider the subset in $M_{0,n}$ of configurations of distinct points in $\bP^1$ given by
\[D^\circ_\Gamma = \{(\pi(q_1), \ldots, \pi(q_n)) \ : \ q_i, q_j, q_k \in \bP^2 \text{ are collinear iff } i, j, k \in \Gamma_\ell \text{ for some } 1 \leq \ell \leq d\}.\]
The \emph{hypertree divisor} $D_\Gamma$ corresponding to $\Gamma$ is the Zariski closure of $D^\circ_\Gamma$ in $\overline{M}_{0,n}$.
The result \cite[Theorem 1.5]{CastravetTevelev} states that the divisor $D_\Gamma$ is nonempty and irreducible. 

\begin{rem}\label{rem:m0n-and-gr2n-coordinates}
    For a hypertree $\Gamma,$ let $f_\Gamma$ denote the defining equation of the corresponding hypertree divisor. A priori, $f_\Gamma$ is a polynomial in coordinates on $M_{0,n}.$ However, we may instead think of $f_\Gamma$ as a torus-fixed polynomial in the Pl\"ucker coordinates using the explicit substitution $p_{ij} = x_i - x_j,$ where $x_1, \ldots, x_n$ are coordinates on $M_{0,n}$ as in \cite{CastravetTevelev, Opie}. Viewing $f_\Gamma$ as a polynomial in Pl\"ucker coordinates, the vanishing locus $\{f_{\Gamma}=0\} \subset \Gr_{2,n}^\circ$ is the pre-image of the hypertree divisor $D_\Gamma$ under the torus quotient. Abusing notation we also often call this pre-image the hypertree divisor and use the same notation.
\end{rem}

Opie \cite[Theorem 3.1]{Opie} gives a determinantal formula for the defining polynomial of a hypertree divisor, generalizing \cite[Proposition 8.1]{CastravetTevelev} for the case where each $\Gamma_i$ has size three. These formulas are quite similar to the numerator of \eqref{eq:pushforward}, a similarity we make precise in this section.

\subsection{Irreducible graphs} We may associate to each irreducible hypertree a set of graphs with boundary.

\begin{definition}\label{def:irreducible}
    Let $\Gamma = \{\Gamma_1,\dots,\Gamma_d\}$ be an irreducible hypertree on the set $[n]$. The \emph{irreducible graphs} associated to $\Gamma$ are a set $\mathcal G(\Gamma)$ of type $(2,n)$ graphs, each of which is constructed as follows:
    \begin{itemize}
        \item For each set $\Gamma_j$, choose an ordering of the elements. Choose any type $(2,|\Gamma_j|)$ top-cell plabic graph $P_j$ whose boundary vertices are the elements of $\Gamma_j$, in clockwise order\footnote{Every such plabic graph $P_j$ can be obtained by taking a triangulation of a $|\Gamma_j|$-gon, making the vertices of the polygon white, and placing a ``tripod'' with a trivalent black vertex inside of each triangle.}.
        \item The graph $G$ is defined as the union $\bigcup_{j = 1}^d P_j$.
    \end{itemize}
\end{definition}

Note that this definition produces a set of abstract graphs, without any particular choice of disk immersion. This is good enough for our purposes as the property of being top-cell and the form $\omega_G$ do not depend on the immersion. 

\begin{lemma}
    Irreducible graphs are top-cell graphs.
\end{lemma}
\begin{proof}
    Recall that a graph is top-cell if and only if it has $n-2$ internal vertices, all of which are black and trivalent, and the graph is convex (\zcref{prop:top cell k=2}). Making use of these properties for the constituent graphs $P_j$, alongside the normalization and convexity of the hypertree $\Gamma$, the properties follow for an irreducible graph $G \in \mathcal G(\Gamma)$.
\end{proof}

\begin{eg}\label{eg:hypertree-graph}
    The graphs $G^I$, $G^{II}$, and $G^{III}$ from \zcref{eg:twosix-hypertree} and \zcref{fig:k2-irreducible} are irreducible graphs, associated to the irreducible hypertrees
    \begin{align*}
            \Gamma^I&=\{123,345,156,246\}, \\
            \Gamma^{II} &=\{123,456,147,257,367\}, \\
            \Gamma^{III} &=\{123,345,567,178,2468\},
    \end{align*}
    respectively. 
    Furthermore, $G^I$ and $G^{II}$ are the unique graphs associated to $\Gamma^{I}$ and $\Gamma^{II}$ respectively, as there is a unique choice of plabic graph for each size-3 set $\Gamma_j$. This is not true for $G^{III}$, as there are other choices of graphs $P_j$ with boundary $\partial P_j = 2468$.
\end{eg}

A key feature of irreducible graphs is the absence of non-trivial top-cell subgraphs.

\begin{lemma}\label{lem:irreducible graph no subgraph}
    Let $\Gamma$ be an irreducible hypertree on $[n]$. Suppose $G = \bigcup_{j=1}^d P_j$ is an irreducible graph in the set $\mathcal G(\Gamma)$, where $P_j$ are the plabic graphs in \zcref{def:irreducible}. Let $T = \bigcup_i G_i$ be a union of lobes that is itself a top-cell graph. Then, either $T = G$ or $T \subseteq P_j$ for some $j \in [d]$.
\end{lemma}
\begin{proof}
    Let $\{P_j\}_{j \in S}$ be the set of plabic graphs for which $T \cap P_j \neq \emptyset$. Let $\ell_T$ be the number of lobes in $T$. For some fixed $j \in S$, $\ell_{j}$ be the number of lobes in $P_j$, and $\ell_{\cap}$ be the number of lobes in $T \cap P_j$. Then, the number of lobes in $T \cup P_j$ is $\ell_T + \ell_j - \ell_\cap$. Furthermore, making use of \zcref{prop:top cell k=2} and the convexity of $G$, one finds
    \[
        \begin{aligned}
            |\partial(T \cup P_j)| &= |\partial T| + |\partial P_j| - |\partial T \cap \partial P_j| \\
            &\leq |\partial T| + |\partial P_j| - |\partial(T \cap P_j)| \\
            &\leq (\ell_T + 2) + (\ell_j + 2) - (\ell_\cap + 2) \\
            &\leq (\ell_T + \ell_j - \ell_\cap) + 2 \\
            &\leq |\partial(T \cup P_j)|.
        \end{aligned}
    \]
    Consequently, the inequalities are in fact equalities and $(\ell_T + \ell_j - \ell_\cap) = |\partial(T \cup P_j)| - 2$. Since $T \cup P_j$ inherits convexity from $G$, and has the correct number of lobes for \zcref{prop:top cell k=2} to apply, $T \cup P_j$ is itself a top-cell graph. By repeating this procedure for each relabeled plabic graph that $T$ intersected, one concludes that $T \cup \bigr(\bigcup_{j \in S} P_j\bigr) = \bigcup_{j \in S} P_j$ is a top-cell graph.

    Each $P_j$ has $|\partial P_j|-2 = |\Gamma_j|-2$ lobes. Their union $\bigcup_{j \in S} P_j$ has $|\partial \bigcup_{j \in S} P_j|-2 = |\bigcup_{j \in S} \Gamma_j|-2$ lobes, so
    \[
        \left|\bigcup_{j \in S} \Gamma_j\right|-2 = \sum_{j \in S} (|\Gamma_j|-2).
    \]
    Using the constraint that $\Gamma$ is irreducible, one concludes that either $|S| = 1$ and $T \subset P_j$ for some $j \in [d]$, or $|S| = d$ and $T$ intersects every relabeled plabic graph in $G$.

    In the latter case, unless $T = G$, one may identify some $r \in [d]$ such that $P_r \not\in T$. Then, consider the top-cell graph $T' := T \cup \bigr(\bigcup_{j \neq r} P_j\bigr)$, which intersects $P_r$. Supposing $T \neq G$, $T'$ contains fewer lobes and consequently has a smaller boundary than $G$. In that case, there exists some index $i \in [n] \setminus \partial T \subset [n] \setminus \bigr(\bigcup_{j \neq r} \Gamma_j\bigr)$, which would be an index $i \in \Gamma_r$ contained in no other $\Gamma_j$. Since every $i \in [n]$ must be contained in at least two subsets of the hypertree, we obtain a contradiction and conclude that $T = G$.
\end{proof}

It is possible to characterize irreducible graphs directly by studying coarsenings of the lobe decomposition. One can use the result to read off the corresponding irreducible hypertree.

\begin{lemma}\label{lem:planar-decomposition}
    Suppose $G$ is a type $(2,n)$ top-cell graph. It is irreducible if and only if there is a unique way to coarsen the lobe decomposition $G= \bigcup_{i \in [n-2]} G_i$ into pieces $\{P_j\}$ such that each piece $P_j=\cup_{i} G_i$ of the coarsening has the following properties
    \begin{itemize}
        \item each $P_j$ is a relabeling of a top-cell plabic graph,
        \item there does not exist a top-cell union of lobes $T = \cup_i G_i \subsetneq G$ that strictly contains any $P_j$,
        \item each index $i \in [n]$ appears in at least two elements of $\Gamma(G) := \{\partial P_1 , \dots , \partial P_d\}$.
    \end{itemize}
\end{lemma}
\begin{proof}
    If $G$ is irreducible, then such a decomposition exists and is unique from \zcref{def:irreducible} and \zcref{lem:irreducible graph no subgraph}. One has $G \in \mathcal G(\Gamma(G))$.

    If $G$ has a unique decomposition with the desired properties, then we define $\Gamma(G) := \{\partial P_1, \dots, \partial P_d\}$ and check it is an irreducible hypertree. The conditions to be a hypertree are easy to check using $G$ being a top-cell graph and the third condition of the decomposition. Suppose $\Gamma(G)$ is not irreducible. Then, there exists some $S$ with $1 < |S| < d$ for which
    \[
        \left| \bigcup_{j \in S} \partial P_j\right| - 2 = \sum_{j \in S} (|\partial P_j|-2).
    \]
    From \zcref{prop:top cell k=2}, one concludes that $\bigcup_{j \in S} P_j$ is a top-cell graph containing some components of the decomposition, contradicting the second condition.
\end{proof}

\subsection{Relating on-shell forms for irreducible graphs and hypertree divisors} The planar decomposition of an irreducible graph $G$ and the irreducible hypertree $\Gamma(G)$ allows us to write a simpler formula for the on-shell form, which makes it clear the numerator vanishes exactly along the hypertree divisor. We use \zcref{rem:m0n-and-gr2n-coordinates} to write the defining polynomial of the hypertree divisor using Pl\"ucker coordinates.

\begin{theorem}\label{prop:graph-to-hypertree}
    Let $G$ be an irreducible graph, let $\{P_i\}_{i \in [d]}$ be the unique planar decomposition of \zcref{lem:planar-decomposition}, and let $\Gamma(G)$ be the corresponding irreducible hypertree. Then
    \begin{equation}\label{eq:pushforward irreducible}
        \omega_G= \frac{(f_{\Gamma(G)})^2}{\displaystyle\prod_{i=1}^d p_{a_1 a_{2}} p_{a_2 a_3} \dots p_{a_{|\partial P_i|}a_1}} dX,
    \end{equation}
    where $f_{\Gamma(G)}$ is the defining polynomial of the hypertree divisor $D_{\Gamma(G)}$, and $\partial P_i = (a_1 , \dots , a_{|\partial P_i|})$ are ordered via any planar immersion of $P_i$.
\end{theorem}
\begin{proof}
   To ease notation, set $\Gamma=\Gamma(G)$.

    Note that the choice of planar immersion and induced ordering for each $P_i$ only affects the denominator by a change of sign; indeed, cyclic shifts of the ordering preserve the denominator, and reflections only change the overall sign.
    
    By the same argument as the proof of \cite[Theorem 3.1]{Opie}, the determinant $\det A'$ from \zcref{thm:pushforward-hypertree} factors as the product of a degree $d-1$ irreducible polynomial $f_\Gamma$ and a degree $(n-2-d)$ monomial in the Pl\"ucker coordinates. It remains to identify this monomial explicitly and show that cancellation yields \eqref{eq:pushforward irreducible}.

    Let $ab$ be a pair of indices appearing together in two triplets $abc$ and $abc'$ of $\mathcal S$. Then, computing $\det A'$ by removing the row corresponding to $abc$ and columns corresponding to $a,b,c$, the only entry in the row corresponding to $abc'$ is $p_{ab}$, so $\det A'$ is divisible by it.

    One may identify each union $P_i = \bigcup_{j \in R_i} G_j$ with the triangulation of a polygon with vertices $\partial P_i$ by triangles with vertices $\partial G_j$. The chords of this triangulation are exactly pairs of indices appearing together in two triplets. With $d$ polygons triangulated by $n-2$ triangles, there are exactly $n-2-d$ such chords. Furthermore, all such chords are distinct; otherwise one would have $|\partial P_i \cup \partial P_{i'}| \leq |\partial P_i| + |\partial P_{i'}| - 2$ implying that $P_i \cup P_{i'}$ is a non-plabic union of lobes violating strict convexity. Thus, using the degree of the monomial, one concludes that
    \[
        \det A' = f_\Gamma \times \prod_{ab \in \mathcal C} p_{ab},
    \]
    where $\mathcal C$ denotes the set of chords in all triangulations. Using the description in terms of triangulations, one can check that the denominator of $\omega_G$ factors as
    \[
        \prod_{(abc) \in \mathcal{S}}p_{ab}p_{ac}p_{bc} = \prod_{i = 1}^d p_{a_1a_2}p_{a_2a_3}\dots p_{a_{|\partial P_i|}a_1} \times \bigg(\prod_{ab \in \mathcal C} p_{ab}\bigg)^2.
    \]
    Plugging the two expressions into \zcref{eq:pushforward},
    \[
        \omega_G = \frac{(\det A')^2}{\displaystyle\prod_{(abc) \in \mathcal{S}}p_{ab}p_{ac}p_{bc}}dX = \frac{(f_\Gamma)^2}{\prod_{i = 1}^d p_{a_1a_2}p_{a_2a_3}\dots p_{a_{|\partial P_i|}a_1}}dX,
    \]
    as desired.
\end{proof}

The following proposition connects the hypertree divisor (or rather, its preimage in $\Gr_{2,n}^\circ$) and the boundary measurement map. It implies that $\omega_G$ vanishes precisely on the complement of $\im \bM^G$.

\begin{proposition}\label{prop:base-locus-hypertree}
    Let $G$ be an irreducible graph and let $\Gamma(G)$ be the corresponding irreducible hypertree. Then the hypertree divisor $D_{\Gamma(G)} \subset \Gr_{2,n}^\circ$ is the complement of $\mathrm{im}\, \bM^G$.
\end{proposition}
\begin{proof}
The proof of \zcref{prop:top cell k=2} shows that for each $X \in \Gr_{2,n}^{\circ}$, there is a unique VRC on $G$ with boundary $X$. The associated weighted adjacency matrix is exactly the matrix $A(X)$, by \zcref{rem:kastelyn A}. By \zcref{cor:fiber-vs-VRC-set}, $X$ is in $\im \bM^G$ if and only if $A(X)$ is full rank. So $X$ is in the complement if and only if $A(X)$ is not full rank. Since some maximal minors of $A(X)$ differ from $\det A'$ by only a Pl\"ucker coordinate of $X$, $A(X)$ is not full rank if and only if $\det A'$ vanishes. By the proof of \zcref{prop:graph-to-hypertree}, the determinant $\det A' = M \cdot f_{\Gamma(G)}$ where $M$ is a monomial in Pl\"ucker coordinates.
\end{proof}

Finally, we comment on the case where $G$ is a top-cell graph that is not irreducible. In this case the matrix $A$ develops a block structure, and both the numerator and denominator of the rational polynomial in \eqref{eq:pushforward} factor into smaller pieces. This factorization is controlled by the presence of a non-trivial union of lobes that is itself a top-cell graph, which was forbidden in the irreducible case by \zcref{lem:irreducible graph no subgraph}.

\begin{lemma}[{\cite[Lemma 2]{Lisitsyn:2025prd}}]
    Let $G$ be a top-cell type $(2,n)$ graph, and $G_R = \bigcup_{i \in R} G_i$ a union of lobes that is itself a top-cell type $(2,n)$ graph. Then, with the set $M = (\bigcup_{i \in R} \partial G_i) \cap (\bigcup_{i \in [n-2] \setminus R} \partial G_i)$, we define a new graph $H = (\bigcup_{i \in [n-2] \setminus R} G_i) \cup P_M$, where $P_M$ is a plabic graph with ordered boundary $(m_1,\dots, m_{|M|})$. Defining $m_{|M|+1} := m_1$, one has
    \[
        \omega_G = \frac{(\det A'_R)^2}{\prod_{(abc) \in \mathcal S_R}p_{ab} p_{ac} p_{bc}}  \times \frac{(\det A'_H)^2}{\prod_{(abc) \in \mathcal S_H}p_{ab} p_{ac} p_{bc}} \times \prod_{i = 1}^{|M|} p_{m_i m_{i+1}} \, dX.
    \]
\end{lemma}
In the case that $|M| = 2$, one has $H = (\bigcup_{i \in [n-2] \setminus R} G_i)$, as the plabic graph with two vertices is empty and the formula uses $\prod_{i=1}^{|M|} p_{m_i m_{i+1}} = (p_{m_1m_2})^2$. The cases $|M| < 2$ would violate the convexity of $G$.

\begin{rem}
    In the case that $G_R$ is itself a planar diagram with no lone indices, one can choose $P_M = G_R$ and the last term cancels the first term. Since these are the only top-cell unions of lobes for an irreducible graph, the lemma does not provide a meaningful factorization as expected.

    In the case that $G$ contains a lone index $r$, $G_R$ can be chosen to be the triplet $rst$ containing the lone index. Then, the form can be computed as
    \[
        \omega_G = \frac{p_{st}}{p_{rs} p_{rt}} \times \frac{(\det A'_{H})^2}{\prod_{(abc) \in \mathcal S_H} p_{ab} p_{ac} p_{bc}}\,dX,
    \]
    in terms of the graph with the lone index removed. In physics, this is known as an inverse soft factor \cite{Arkani-Hamed:2009ljj,Arkani-Hamed:2009kmp,Bourjaily:2010kw,Bullimore:2010pa,Boucher-Veronneau:2011rwd,Nandan:2012rk}. 

    Starting with a diagram that is not irreducible, one can repeatedly apply the factorization until the result is a product of terms coming from either irreducible diagrams or diagrams that contain a single trivalent black vertex.
\end{rem}

\begin{eg}
    We give three examples of reducible diagrams, pictured in \zcref{fig:reducible}, identifying each with the set of triplets $\{\partial G_i\}_{i \in [n-2]}$. In each case, the irreducible type (2,6) top-cell graph is a proper subgraph.
    \[
        \begin{aligned}
            \{\partial G_i^A\}_{i \in [5]} &= \{123,345,156,246,127\}, \\
            \{\partial G_i^B\}_{i \in [7]} &= \{123,345,156,246,178,389,579\}, \\
            \{\partial G_i^B\}_{i \in [10]} &= \{123,345,156,246,1\ta\tb,2\tb\ttc,3\ttc\td,4\td\te,5\te\tf,6\tf\ta\}.
        \end{aligned}
    \]
    The first case demonstrates how a lone index can be removed: excluding the lobe containing the lone index leaves a top-cell graph. The second case demonstrates a factorization with 3 shared indices: the second factor has an extra lobe corresponding to this triplet. The third case demonstrates a factorization with many shared indices; in particular, it is enough that the second factor still has the same boundary as the original graph.

    The forms factor as
    \[
        \begin{aligned}
            \omega_{G^A} &= \frac{(p_{16}p_{24}p_{35} - p_{15}p_{26}p_{34})^2}{p_{12}p_{13}p_{23}p_{34}p_{35}p_{45}p_{15}p_{16}p_{56}p_{24}p_{26}p_{46}} \times \frac{1}{p_{12}p_{17}p_{27}} \times (p_{12})^2\, dX \\
            \omega_{G^B} &= \frac{(p_{16}p_{24}p_{35} - p_{15}p_{26}p_{34})^2}{p_{12}p_{13}p_{23}p_{34}p_{35}p_{45}p_{15}p_{16}p_{56}p_{24}p_{26}p_{46}} \\ & \quad \quad \times \frac{(p_{17}p_{35}p_{89} + p_{15}p_{78}p_{39})^2}{p_{17}p_{18}p_{78}p_{38}p_{39}p_{89}p_{57}p_{59}p_{79}p_{13}p_{15}p_{35}} \times (p_{13}p_{15}p_{35}) \, dX \\
            \omega_{G^C} &= \frac{(p_{16}p_{24}p_{35} - p_{15}p_{26}p_{34})^2}{p_{12}p_{13}p_{23}p_{34}p_{35}p_{45}p_{15}p_{16}p_{56}p_{24}p_{26}p_{46}}  \\ & \quad \quad \times \frac{\left(p_{1\tb}p_{2\ttc}p_{3\td}p_{4\te}p_{5\tf}p_{6\ta} - p_{1\ta}p_{2\tb}p_{3\ttc}p_{4\td}p_{5\te}p_{6\tf}\right)^2
            }{
            p_{12}p_{23}p_{34}p_{45}p_{56}p_{16}
            p_{1\ta}p_{1\tb}p_{\ta\tb}
            p_{2\tb}p_{2\ttc}p_{\tb\ttc}
            p_{3\ttc}p_{3\td}p_{\ttc\td}
            p_{4\td}p_{4\te}p_{\td\te}
            p_{5\te}p_{5\tf}p_{\te\tf}
            p_{6\ta}p_{6\tf}p_{\tf\ta}
            } \\ & \quad \quad \quad \times (p_{12}p_{23}p_{34}p_{45}p_{56}p_{16}) \, dX
        \end{aligned}
    \]
\end{eg}

\begin{figure}
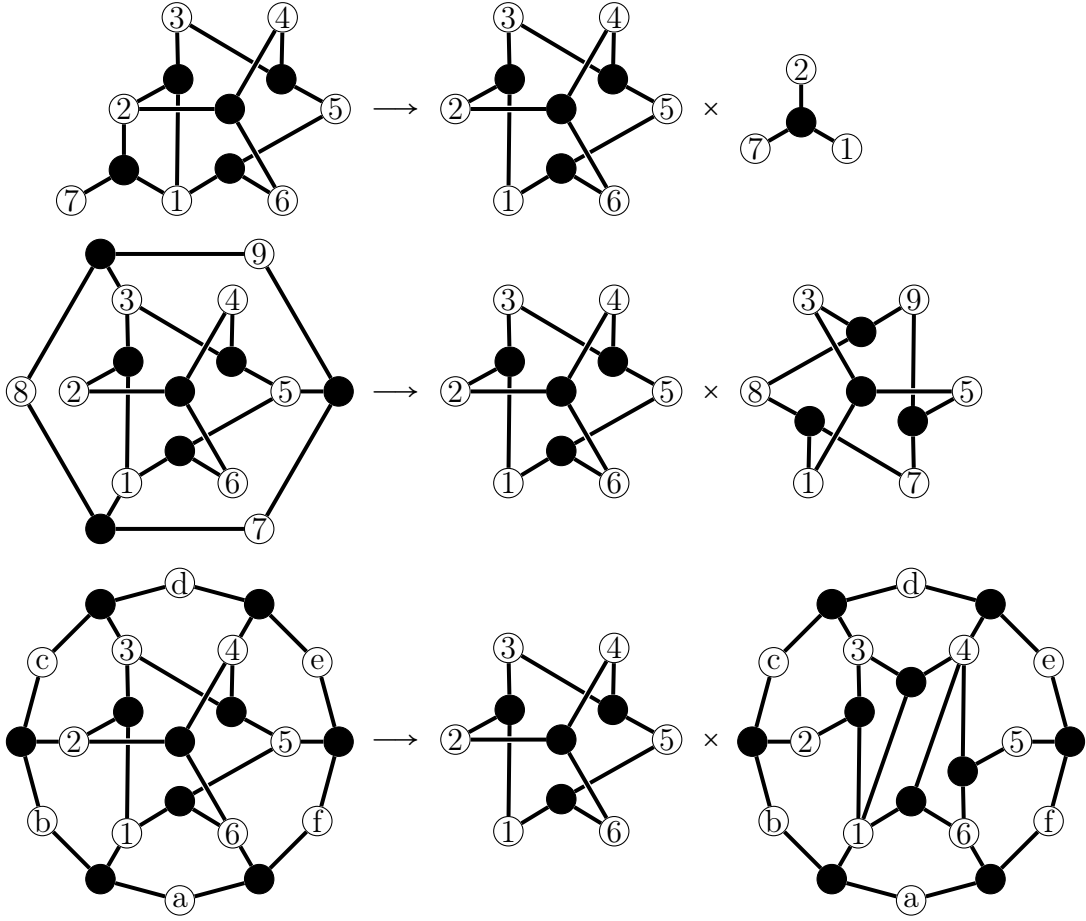

    \[
        \begin{aligned}
            \vcenter{\hbox{\onshellgraph{{(-120:2)/1,(-180:2)/2,(120:2)/3,(60:2)/4,(0:2)/5,(-60:2)/6,(-150:3.46)/7}}{{(-120:2)/W1,(-180:2)/W2,(120:2)/W3,(60:2)/W4,(0:2)/W5,(-60:2)/W6,(-150:3.46)/W7}}{{(150:1.125)/B1,(30:1.125)/B2,(-90:1.125)/B3,(0,0)/B4,(-150:2.31)/B5}}{{W1/B1, W2/B1, W3/B1, W3/B2, W4/B2, W5/B2, W1/B3, W5/B3, W6/B3, B4/W2, B4/W4, B4/W6, W1/B5, W2/B5, W7/B5}}{0.7}}} &\longrightarrow \vcenter{\hbox{\onshellgraph{{(-120:2)/1,(-180:2)/2,(120:2)/3,(60:2)/4,(0:2)/5,(-60:2)/6}}{{(-120:2)/W1,(-180:2)/W2,(120:2)/W3,(60:2)/W4,(0:2)/W5,(-60:2)/W6}}{{(150:1.125)/B1,(30:1.125)/B2,(-90:1.125)/B3,(0,0)/B4}}{{W1/B1, W2/B1, W3/B1, W3/B2, W4/B2, W5/B2, W1/B3, W5/B3, W6/B3, B4/W2, B4/W4, B4/W6}}{0.7}}} \times \vcenter{\hbox{\onshellgraph{{(-30:1)/1,(-150:1)/7,(90:1)/2}}{{(-30:1)/W1,(-150:1)/W7,(90:1)/W2}}{{(0,0)/B}}{ {W1/B,W2/B,W7/B} }{0.7}}} \\
            \vcenter{\hbox{\onshellgraph{{(-120:2)/1,(-180:2)/2,(120:2)/3,(60:2)/4,(0:2)/5,(-60:2)/6,(-180:3)/8,(60:3)/9,(-60:3)/7}}{{(-120:2)/W1,(-180:2)/W2,(120:2)/W3,(60:2)/W4,(0:2)/W5,(-60:2)/W6,(-180:3)/W8,(60:3)/W9,(-60:3)/W7}}{{(150:1.125)/B1,(30:1.125)/B2,(-90:1.125)/B3,(0,0)/B4,(-120:3)/B5,(120:3)/B6,(0:3)/B7}}{{W1/B1, W2/B1, W3/B1, W3/B2, W4/B2, W5/B2, W1/B3, W5/B3, W6/B3, B4/W2, B4/W4, B4/W6, W1/B5, W7/B5, W8/B5, W3/B6, W8/B6, W9/B6, W5/B7, W9/B7, W7/B7}}{0.7}}} &\longrightarrow \vcenter{\hbox{\onshellgraph{{(-120:2)/1,(-180:2)/2,(120:2)/3,(60:2)/4,(0:2)/5,(-60:2)/6}}{{(-120:2)/W1,(-180:2)/W2,(120:2)/W3,(60:2)/W4,(0:2)/W5,(-60:2)/W6}}{{(150:1.125)/B1,(30:1.125)/B2,(-90:1.125)/B3,(0,0)/B4}}{{W1/B1, W2/B1, W3/B1, W3/B2, W4/B2, W5/B2, W1/B3, W5/B3, W6/B3, B4/W2, B4/W4, B4/W6}}{0.7}}} \times \vcenter{\hbox{\onshellgraph{{(-60:2)/7,(-120:2)/1,(180:2)/8,(120:2)/3,(60:2)/9,(0:2)/5}}{{(-60:2)/W1,(-120:2)/W2,(180:2)/W3,(120:2)/W4,(60:2)/W5,(0:2)/W6}}{{(210:1.125)/B1,(90:1.125)/B2,(-30:1.125)/B3,(0,0)/B4}}{{W1/B1, W2/B1, W3/B1, W3/B2, W4/B2, W5/B2, W1/B3, W5/B3, W6/B3, B4/W2, B4/W4, B4/W6}}{0.7}}} \\
            \vcenter{\hbox{\onshellgraph{{(-120:2)/1,(-180:2)/2,(120:2)/3,(60:2)/4,(0:2)/5,(-60:2)/6,(-90:3)/a,(-150:3)/b,(150:3)/c,(90:3)/d,(30:3)/e,(-30:3)/f}}{{(-120:2)/W1,(-180:2)/W2,(120:2)/W3,(60:2)/W4,(0:2)/W5,(-60:2)/W6,(-90:3)/Wa,(-150:3)/Wb,(150:3)/Wc,(90:3)/Wd,(30:3)/We,(-30:3)/Wf}}{{(150:1.125)/B1,(30:1.125)/B2,(-90:1.125)/B3,(0,0)/B4,(-120:3)/B5,(-180:3)/B6,(120:3)/B7,(60:3)/B8,(0:3)/B9,(-60:3)/B10}}{{W1/B1, W2/B1, W3/B1, W3/B2, W4/B2, W5/B2, W1/B3, W5/B3, W6/B3, B4/W2, B4/W4, B4/W6, W1/B5, Wa/B5, Wb/B5, W2/B6, Wb/B6, Wc/B6, W3/B7, Wc/B7, Wd/B7, W4/B8, Wd/B8, We/B8, W5/B9, We/B9, Wf/B9, W6/B10, Wf/B10, Wa/B10}}{0.7}}} &\longrightarrow \vcenter{\hbox{\onshellgraph{{(-120:2)/1,(-180:2)/2,(120:2)/3,(60:2)/4,(0:2)/5,(-60:2)/6}}{{(-120:2)/W1,(-180:2)/W2,(120:2)/W3,(60:2)/W4,(0:2)/W5,(-60:2)/W6}}{{(150:1.125)/B1,(30:1.125)/B2,(-90:1.125)/B3,(0,0)/B4}}{{W1/B1, W2/B1, W3/B1, W3/B2, W4/B2, W5/B2, W1/B3, W5/B3, W6/B3, B4/W2, B4/W4, B4/W6}}{0.7}}} \times \vcenter{\hbox{\onshellgraph{{(-120:2)/1,(-180:2)/2,(120:2)/3,(60:2)/4,(0:2)/5,(-60:2)/6,(-90:3)/a,(-150:3)/b,(150:3)/c,(90:3)/d,(30:3)/e,(-30:3)/f}}{{(-120:2)/W1,(-180:2)/W2,(120:2)/W3,(60:2)/W4,(0:2)/W5,(-60:2)/W6,(-90:3)/Wa,(-150:3)/Wb,(150:3)/Wc,(90:3)/Wd,(30:3)/We,(-30:3)/Wf}}{{(150:1.125)/B1,(-30:1.125)/B2,(90:1.125)/B3,(-90:1.125)/B4,(-120:3)/B5,(-180:3)/B6,(120:3)/B7,(60:3)/B8,(0:3)/B9,(-60:3)/B10}}{{W1/B1, W2/B1, W3/B1, W4/B2, W5/B2, W6/B2, W1/B3, W3/B3, W4/B3, B4/W1, B4/W4, B4/W6, W1/B5, Wa/B5, Wb/B5, W2/B6, Wb/B6, Wc/B6, W3/B7, Wc/B7, Wd/B7, W4/B8, Wd/B8, We/B8, W5/B9, We/B9, Wf/B9, W6/B10, Wf/B10, Wa/B10}}{0.7}}}
        \end{aligned}
    \]
    \caption{Three reducible diagrams, alongside the graphs remaining after factoring. Note that to more clearly illustrate the factoring, the drawings here are not disk immersions.}\label{fig:reducible}
\end{figure}

\subsection{Outlook}
In this section and in \zcref{sec:k=2} we focused on the case of $k=2$ top-cell graphs. This choice was to make the connection to hypertrees clearer and to avoid overburdening the notation. However, \zcref{thm:pushforward-hypertree} applies more generally to top-cell type $(k, n)$ graphs whose internal vertices are all black and $(k+1)$-valent, which necessarily have exactly $n-k$ internal vertices. 

As in the $k=2$ case, we let $\mathcal{S} := \{\partial G_1, \ldots, \partial G_{n-k}\}$ denote the set of boundaries of lobes. Let $A$ be the $(n-k) \times n$ matrix with rows indexed by $\mathcal{S},$ columns indexed by $[n]$ and entries
\[A_{S, i} = \begin{cases}
    (-1)^{{\rm ind}(i) -1}p_{S \setminus i} & i \in S \\
    0 & \text{else}
\end{cases},\]
where ${\rm ind}(i)$ is the index of $i$ in $S.$ The following is a generalization of \zcref{thm:pushforward-hypertree}, and the proof is essentially identical. The only change is that in \zcref{lem:jacobian}, the exponent two in \eqref{eq:det identity} is replaced by the parameter $k$. 
\begin{theorem}\label{thm:pushforward-higher-k}
    Let $G$ be a top-cell type $(k,n)$ graph whose internal vertices are all black and $(k+1)$-valent. Suppose that $\deg \bM^G = 1.$ The on-shell form of $G$ is given by
    \begin{equation}
        \omega_G = \frac{(\det A')^k}{\displaystyle\prod_{S \in \mathcal{S}} \prod_{i \in S}p_{S \setminus i}}dX
    \end{equation}
    where $A'$ is any submatrix of $A$ obtained by deleting a row and the $k+1$ columns with nonzero
entries in that row.
\end{theorem}

It is also straightforward to generalize \zcref{prop:base-locus-hypertree} to this setting: the vanishing locus of the numerator $\det A'$ of $\omega_G$ is the complement of $\im \bM^G$, by the same proof.

These results suggest many further questions about this class of graphs. For instance, it would be interesting to understand the analogue of the convexity condition, and whether the divisors defined by the determinantal expressions $\det A'$ are special from the perspective of algebraic geometry. The relevant algebro-geometric object to consider here may be \emph{configuration space}, that is $\Gr_{k,n}^{\circ}/T$, or some compactification of it. 

More generally, we do not know much about top-cell type $(k,n)$ graphs. A basic question is to understand the degree of the boundary measurement map (cf. \zcref{eg:threesix}). A more ambitious question is to find a rational formula for the resulting on-shell forms. Guided by \zcref{prop:base-locus-hypertree}, one might hope that the numerator of such a formula vanishes precisely on the complement of $\im \bM^G$. 

\printbibliography

@article{Arkani-Hamed:2009ljj,
    author = "Arkani-Hamed, Nima and Cachazo, Freddy and Cheung, Clifford and Kaplan, Jared",
    label = {AH+10},
    title = "{A Duality For The S Matrix}",
    eprint = "0907.5418",
    archivePrefix = "arXiv",
    primaryClass = "hep-th",
    doi = "10.1007/JHEP03(2010)020",
    journal = "JHEP",
    volume = "03",
    pages = "020",
    year = "2010"
}

@article{Arkani-Hamed:2009kmp,
    author = "Arkani-Hamed, Nima and Bourjaily, Jacob and Cachazo, Freddy and Trnka, Jaroslav",
    label = {AH+11},
    title = "{Unification of Residues and Grassmannian Dualities}",
    eprint = "0912.4912",
    archivePrefix = "arXiv",
    primaryClass = "hep-th",
    doi = "10.1007/JHEP01(2011)049",
    journal = "JHEP",
    volume = "01",
    pages = "049",
    year = "2011"
}

@article{Bourjaily:2010kw,
    author = "Bourjaily, Jacob L. and Trnka, Jaroslav and Volovich, Anastasia and Wen, Congkao",
    title = "{The Grassmannian and the Twistor String: Connecting All Trees in N=4 SYM}",
    eprint = "1006.1899",
    archivePrefix = "arXiv",
    primaryClass = "hep-th",
    doi = "10.1007/JHEP01(2011)038",
    journal = "JHEP",
    volume = "01",
    pages = "038",
    year = "2011"
}

@article{Bullimore:2010pa,
    author = "Bullimore, Mathew",
    title = "{Inverse Soft Factors and Grassmannian Residues}",
    eprint = "1008.3110",
    archivePrefix = "arXiv",
    primaryClass = "hep-th",
    doi = "10.1007/JHEP01(2011)055",
    journal = "JHEP",
    volume = "01",
    pages = "055",
    year = "2011"
}

@article{Boucher-Veronneau:2011rwd,
    author = "Boucher-Veronneau, Camille and Larkoski, Andrew J.",
    label = {BVL},
    title = "{Constructing Amplitudes from Their Soft Limits}",
    eprint = "1108.5385",
    archivePrefix = "arXiv",
    primaryClass = "hep-th",
    reportNumber = "SLAC-PUB-14549",
    doi = "10.1007/JHEP09(2011)130",
    journal = "JHEP",
    volume = "09",
    pages = "130",
    year = "2011"
}

@article{Nandan:2012rk,
    author = "Nandan, Dhritiman and Wen, Congkao",
    title = "{Generating All Tree Amplitudes in N=4 SYM by Inverse Soft Limit}",
    eprint = "1204.4841",
    archivePrefix = "arXiv",
    primaryClass = "hep-th",
    reportNumber = "QMUL-PH-12-09",
    doi = "10.1007/JHEP08(2012)040",
    journal = "JHEP",
    volume = "08",
    pages = "040",
    year = "2012"
}

@book{3264,
  title={3264 and All That: A Second Course in Algebraic Geometry},
  author={Eisenbud, David and Harris, Joe},
  year={2016},
  publisher={Cambridge University Press}
}

@article{Cachazo:2018wvl,
    author = "Cachazo, Freddy and Early, Nick and Guevara, Alfredo and Mizera, Sebastian",
    title = "{$\Delta$-algebra and scattering amplitudes}",
    eprint = "1812.01168",
    archivePrefix = "arXiv",
    primaryClass = "hep-th",
    doi = "10.1007/JHEP02(2019)005",
    journal = "JHEP",
    volume = "02",
    pages = "005",
    year = "2019"
}

@BOOK{biggs,
  title     = "Algebraic Graph Theory",
  author    = "Biggs, Norman",
  publisher = "Cambridge University Press",
  series    = "Cambridge tracts in mathematics",
  edition   =  2,
  month     =  aug,
  year      =  2012,
  address   = "Cambridge, England",
  language  = "en"
}

@ARTICLE{BrownDupont,
  title     = "Positive geometries and canonical forms via mixed Hodge theory",
  author    = "Brown, Francis and Dupont, Cl{\'e}ment",
  journal   = "Commun. Math. Phys.",
  publisher = "Springer Science and Business Media LLC",
  volume    =  406,
  number    =  11,
  month     =  nov,
  year      =  2025,
  language  = "en"
}

@book{CoxLittleSchenck,
  title={Toric Varieties},
  author={Cox, David A and Little, John B and Schenck, Henry K},
  volume={124},
  year={2011},
  publisher={American Mathematical Society},
  series={Graduate Studies in Mathematics},
  isbn={978-0821848197}
}

@ARTICLE{psw2009,
  title     = "Matching polytopes, toric geometry, and the totally non-negative
               Grassmannian",
  author    = "Postnikov, Alexander and Speyer, David and Williams, Lauren",
  journal   = "J. Algebr. Comb.",
  publisher = "Springer Science and Business Media LLC",
  volume    =  30,
  number    =  2,
  pages     = "173--191",
  month     =  sep,
  year      =  2009,
  language  = "en"
}

@ARTICLE{ahbl2017,
  title     = "Positive geometries and canonical forms",
  author    = "Arkani-Hamed, Nima and Bai, Yuntao and Lam, Thomas",
  label = {AHBL},
  journal   = "J. High Energy Phys.",
  publisher = "Springer Science and Business Media LLC",
  volume    =  2017,
  number    =  11,
  month     =  nov,
  year      =  2017,
  copyright = "https://creativecommons.org/licenses/by/4.0/",
  language  = "en"
}

@ARTICLE{Opie,
  title     = "Extremal divisors on moduli spaces of rational curves with
               marked points",
  author    = "Opie, Morgan",
  journal   = "Michigan Math. J.",
  publisher = "Michigan Mathematical Journal",
  volume    =  65,
  number    =  2,
  month     =  jun,
  year      =  2016
}

@article{CastravetTevelev,
url = {https://doi.org/10.1515/CRELLE.2011.189},
title = {Hypertrees, projections, and moduli of stable rational curves},
author = {Ana-Maria Castravet and Jenia Tevelev},
pages = {121--180},
volume = {2013},
number = {675},
journal = {Journal für die reine und angewandte Mathematik (Crelles Journal)},
doi = {doi:10.1515/CRELLE.2011.189},
year = {2013},
lastchecked = {2026-04-18}
}

@article{postnikov,
      title={Total positivity, Grassmannians, and networks}, 
      author={Alexander Postnikov},
      year={2006},
      eprint={math/0609764},
      archivePrefix={arXiv}
}

@article{Bourjaily:2016mnp,
    author = "Bourjaily, Jacob L. and Franco, Sebastian and Galloni, Daniele and Wen, Congkao",
    title = "{Stratifying On-Shell Cluster Varieties: the Geometry of Non-Planar On-Shell Diagrams}",
    eprint = "1607.01781",
    archivePrefix = "arXiv",
    primaryClass = "hep-th",
    reportNumber = "CCNY-HEP-16-06",
    doi = "10.1007/JHEP10(2016)003",
    journal = "JHEP",
    volume = "10",
    pages = "003",
    year = "2016"
}

@article{Arkani-Hamed:2014bca,
    author = "Arkani-Hamed, Nima and Bourjaily, Jacob L. and Cachazo, Freddy and Postnikov, Alexander and Trnka, Jaroslav",
    label = {AH+15},
    title = "{On-Shell Structures of MHV Amplitudes Beyond the Planar Limit}",
    eprint = "1412.8475",
    archivePrefix = "arXiv",
    primaryClass = "hep-th",
    reportNumber = "CALT-TH-2014-169",
    doi = "10.1007/JHEP06(2015)179",
    journal = "JHEP",
    volume = "06",
    pages = "179",
    year = "2015"
}

@article{Franco:2015rma,
    author = "Franco, Sebastian and Galloni, Daniele and Penante, Brenda and Wen, Congkao",
    title = "{Non-Planar On-Shell Diagrams}",
    eprint = "1502.02034",
    archivePrefix = "arXiv",
    primaryClass = "hep-th",
    doi = "10.1007/JHEP06(2015)199",
    journal = "JHEP",
    volume = "06",
    pages = "199",
    year = "2015"
}

@article{Lisitsyn:2025prd,
    author = "Lisitsyn, Artyom and Oktem, Umut and Sherman-Bennett, Melissa and Trnka, Jaroslav",
    label = {L+25},
    title = "{Grassmannian Geometries for Non-Planar On-Shell Diagrams}",
    eprint = "2512.25005",
    archivePrefix = "arXiv",
    primaryClass = "hep-th",
    month = "12",
    year = "2025"
}

@article{Parke:1986gb,
    author = "Parke, Stephen J. and Taylor, T. R.",
    title = "{An Amplitude for $n$ Gluon Scattering}",
    reportNumber = "FERMILAB-PUB-86-042-T",
    doi = "10.1103/PhysRevLett.56.2459",
    journal = "Phys. Rev. Lett.",
    volume = "56",
    pages = "2459",
    year = "1986"
}

@book{Arkani-Hamed:2012zlh,
    author = "Arkani-Hamed, Nima and Bourjaily, Jacob L. and Cachazo, Freddy and Goncharov, Alexander B. and Postnikov, Alexander and Trnka, Jaroslav",
    label = {AH+16},
    title = "{Grassmannian Geometry of Scattering Amplitudes}",
    publisher = "Cambridge University Press",
    month = "4",
    year = "2016"
}

@article{Affolter_2023,
   title={Vector-relation configurations and plabic graphs},
   volume={30},
   ISSN={1420-9020},
   url={http://dx.doi.org/10.1007/s00029-023-00898-z},
   DOI={10.1007/s00029-023-00898-z},
   number={1},
   journal={Selecta Mathematica},
   publisher={Springer Science and Business Media LLC},
   author={Affolter, Niklas and Glick, Max and Pylyavskyy, Pavlo and Ramassamy, Sanjay},
   year={2023},
   month=dec }

@article {Speyer16,
    AUTHOR = {Speyer, David E.},
     TITLE = {Variations on a theme of {K}asteleyn, with application to the
              totally nonnegative {G}rassmannian},
   JOURNAL = {Electron. J. Combin.},
  FJOURNAL = {Electronic Journal of Combinatorics},
    VOLUME = {23},
      YEAR = {2016},
    NUMBER = {2},
     PAGES = {Paper 2.24, 7},
      ISSN = {1077-8926},
   MRCLASS = {05C30 (05A15)},
  MRNUMBER = {3512646},
MRREVIEWER = {Markus\ E.\ Fulmek},
       DOI = {10.37236/5871},
       URL = {https://doi.org/10.37236/5871},
}

@article {MSTwist,
    AUTHOR = {Muller, Greg and Speyer, David E.},
     TITLE = {The twist for positroid varieties},
   JOURNAL = {Proc. Lond. Math. Soc. (3)},
  FJOURNAL = {Proceedings of the London Mathematical Society. Third Series},
    VOLUME = {115},
      YEAR = {2017},
    NUMBER = {5},
     PAGES = {1014--1071},
      ISSN = {0024-6115,1460-244X},
   MRCLASS = {14M15 (05C22 13F60)},
  MRNUMBER = {3733558},
MRREVIEWER = {Truong\ Le\ Hoang},
       DOI = {10.1112/plms.12056},
       URL = {https://doi.org/10.1112/plms.12056},
}

@article {OH01,
    AUTHOR = {Ohsugi, Hidefumi and Hibi, Takayuki},
     TITLE = {Convex polytopes all of whose reverse lexicographic initial
              ideals are squarefree},
   JOURNAL = {Proc. Amer. Math. Soc.},
  FJOURNAL = {Proceedings of the American Mathematical Society},
    VOLUME = {129},
      YEAR = {2001},
    NUMBER = {9},
     PAGES = {2541--2546},
      ISSN = {0002-9939,1088-6826},
   MRCLASS = {13P10 (52B20)},
  MRNUMBER = {1838375},
MRREVIEWER = {Joseph\ Gubeladze},
       DOI = {10.1090/S0002-9939-01-05853-1},
       URL = {https://doi.org/10.1090/S0002-9939-01-05853-1},
}

@article{Lam-grassmann-polytope,
      title={Totally nonnegative Grassmannian and Grassmann polytopes},
      label= {La15},
      author={Thomas Lam},
      year={2015},
      eprint={1506.00603},
      archivePrefix={arXiv},
      primaryClass={math.CO},
      url={https://arxiv.org/abs/1506.00603}, 
}

@article{Lam-pos-geom,
      title={An invitation to positive geometries}, 
      label= {La22},
      author={Thomas Lam},
      year={2022},
      eprint={2208.05407},
      archivePrefix={arXiv},
      primaryClass={math.CO},
      url={https://arxiv.org/abs/2208.05407}, 
}

@article{Mason,
author = {Mason, J. H.},
title = {On a Class of Matroids Arising From Paths in Graphs},
journal = {Proceedings of the London Mathematical Society},
volume = {s3-25},
number = {1},
pages = {55-74},
doi = {https://doi.org/10.1112/plms/s3-25.1.55},
url = {https://londmathsoc.onlinelibrary.wiley.com/doi/abs/10.1112/plms/s3-25.1.55},
eprint = {https://londmathsoc.onlinelibrary.wiley.com/doi/pdf/10.1112/plms/s3-25.1.55},
year = {1972}
}

@article {TW,
    AUTHOR = {Talaska, Kelli and Williams, Lauren},
     TITLE = {Network parametrizations for the {G}rassmannian},
   JOURNAL = {Algebra Number Theory},
  FJOURNAL = {Algebra \& Number Theory},
    VOLUME = {7},
      YEAR = {2013},
    NUMBER = {9},
     PAGES = {2275--2311},
      ISSN = {1937-0652,1944-7833},
   MRCLASS = {14M15 (05E10)},
  MRNUMBER = {3152014},
MRREVIEWER = {Deborah\ Olayide A. Ajayi},
       DOI = {10.2140/ant.2013.7.2275},
       URL = {https://doi.org/10.2140/ant.2013.7.2275},
}

@article{GLSBS,
      title={Braid variety cluster structures, I: 3D plabic graphs}, 
      author={Pavel Galashin and Thomas Lam and Melissa Sherman-Bennett and David Speyer},
      year={2022},
      eprint={2210.04778},
      archivePrefix={arXiv},
      primaryClass={math.CO},
      url={https://arxiv.org/abs/2210.04778}, 
}
\end{document}